%% file: main.tex
\documentclass[leqno,12pt]{amsart} 
\usepackage{multirow, multicol}
\usepackage{mathrsfs}
\usepackage{comment}
\usepackage{yhmath}
\usepackage{graphicx}
\usepackage[inline]{enumitem}
\usepackage{indentfirst,amscd}
\usepackage{amsthm,amsmath, amssymb, amsfonts, empheq,mathtools,dsfont, stmaryrd}
\usepackage{upgreek}
\usepackage{tikz-cd}
\usepackage{calc}
\newtheorem{mainthm}{Main Theorem}

\theoremstyle{plain} 
\newtheorem{theorem}{\indent\sc Theorem}[section]
\newtheorem{lemma}[theorem]{\indent\sc Lemma}
\newtheorem{corollary}[theorem]{\indent\sc Corollary}
\newtheorem{proposition}[theorem]{\indent\sc Proposition}

\theoremstyle{definition} 
\newtheorem{definition}[theorem]{\indent\sc Definition}
\newtheorem{remark}[theorem]{\indent\sc Remark}
\newtheorem{example}[theorem]{\indent\sc Example}
\newtheorem{notation}[theorem]{\indent\sc Notation}

\usepackage{hyperref}
\hypersetup{
    colorlinks=true,
    linkcolor=black,
    filecolor=black,
    urlcolor=black,
}
\usepackage[capitalize]{cleveref}

\DeclareMathOperator{\ZZ}{\mathbb{Z}}
\DeclareMathOperator{\NN}{\mathbb{N}}

\DeclareMathOperator{\OO}{\mathcal{O}}
\DeclareMathOperator{\FF}{\mathbb{F}}
\DeclareMathOperator{\KK}{\mathbb{K}}
\DeclareMathOperator{\DD}{\mathscr{D}}

\DeclareMathOperator{\Spec}{Spec}

\DeclareMathOperator{\Sch}{\mathbf{Sch}}

\DeclareMathOperator{\Aff}{\mathbf{Aff}}
\DeclareMathOperator{\Mod}{\mathbf{Mod}}
\DeclareMathOperator{\BMod}{\mathbf{BMod}}
\DeclareMathOperator{\Band}{\mathbf{Band}}
\DeclareMathOperator{\Idyll}{\mathbf{Idyll}}
\DeclareMathOperator{\BAff}{\mathbf{BAff}}
\DeclareMathOperator{\BSch}{\mathbf{BSch}}
\DeclareMathOperator{\Comm}{Comm}
\DeclareMathOperator{\BAlg}{\mathbf{BAlg}}
\DeclareMathOperator{\Alg}{\mathbf{Alg}}
\DeclareMathOperator{\Frac}{Frac}

\DeclareMathOperator{\Hom}{Hom}

\DeclareMathOperator{\Sh}{Sh}

\DeclareMathOperator{\coeq}{coeq}
\DeclareMathOperator{\eq}{eq}

\DeclareMathOperator{\im}{Im}

\DeclareMathOperator{\id}{id}
\DeclareMathOperator{\Iso}{Iso}

\DeclareMathOperator{\pr}{pr}

\newcommand{\op}{
  \mathop{
    \vphantom{\bigoplus} 
    \mathchoice
      {\vcenter{\hbox{\resizebox{\widthof{$\displaystyle\bigoplus$}}{!}{$\boxplus$}}}}
      {\vcenter{\hbox{\resizebox{\widthof{$\bigoplus$}}{!}{$\boxplus$}}}}
      {\vcenter{\hbox{\resizebox{\widthof{$\scriptstyle\oplus$}}{!}{$\boxplus$}}}}
      {\vcenter{\hbox{\resizebox{\widthof{$\scriptscriptstyle\oplus$}}{!}{$\boxplus$}}}}
  }\displaylimits 
}

\begin{document}

\title[On the category of modules over bands]{On the category of modules over bands:\\ Relative Schemes, Hyperring schemes and Proto-Exactness} 

\author[Lucas Hamada]{Lucas Hamada} 

%
\address{
Department of Mathematics \endgraf
Institute of Science Tokyo \endgraf
2-12-1, O-okayama, Meguro-ku, Tokyo 152-8551 \endgraf
Japan
}
\email{lucas.h.r.hamada@gmail.com}


\input{abstract}

\maketitle

\tableofcontents

\input{Introduction2}

\input{Preliminaries}

\input{Bands}

\input{Modulesoverbands}

\input{RelativeSchemes}

\input{HyperringasBand}

\input{protoexact}

\end{document}

%% file: abstract.tex
\begin{abstract}
Bands and idylls are algebraic structures introduced recently by M.~Baker, N.~Bowler, T.~Jin, and O.~Lorscheid in the context of matroid theory. Bands generalize hyperrings and provide a new approach to geometry over the field with one element \(\FF_1\).

In the first part of this paper, we develop the theory of modules over a band, establishing several of its fundamental properties. In particular, we prove that the category of modules over a band is a closed symmetric monoidal category that is both complete and cocomplete.

We then apply this theory in two directions. First, we prove that the category of band schemes is equivalent to the category of schemes relative to the category of modules over a band, in the sense of B.~To\"en and M.~Vaqui\'e.

Second, we investigate the relationship between band schemes and the affine hyperring schemes as developed by R.~Procesi-Ciampi, R.~Rota, and J.~Jun. Although bands generalize hyperrings, we see that the corresponding scheme theories are not compatible. Finally, we prove that the category of modules over a band admits a proto-exact structure, generalizing a result of J.~Jun for hypermodules over hyperrings.
\end{abstract}

%% file: Introduction2.tex
\section{Introduction}

The \textit{field with one element} \(\FF_1\) is a hypothetical object proposed by M.~J.~Tits~\cite{Tits}, expected to play the role of a \textit{field of coefficients} for the ring of integers \(\ZZ\), in analogy with the way a finite field \(\FF_p\) serves as the field of coefficients for the polynomial ring \(\FF_p[x]\). Later, motivated by the works of C.~Deninger and N.~Kurokawa, Y.~Manin~\cite{Manin} suggested that ideas from \(\FF_1\)-geometry might provide a framework for a proof of the Riemann hypothesis. This proposal led to the development of several approaches to \(\FF_1\)-geometry. In \cite{Deitmar}, A.~Deitmar introduced the theory of monoid schemes by replacing commutative rings with monoids in the definition of schemes. This approach became a prototype for \(\FF_1\)-geometry and is contained in most subsequent approaches. For a comprehensive survey of existing approaches to \(\FF_1\)-geometry and their interrelations, we refer to \cite{LorscheidSurvey}.

Recently, M.~Baker, T.~Jin, and O.~Lorscheid introduced in \cite{Lorscheid1} a new class of algebraic structures called \emph{bands}, together with their field-like counterparts, called \emph{idylls}. Bands generalize several algebraic structures, such as hyperrings, fuzzy rings, and partial fields. The theory of matroids over idylls, developed by M.~Baker, N.~Bowler, and O.~Lorscheid~\cite{Baker1, Baker2}, generalizes matroids, valuated matroids, oriented matroids, and regular matroids, asserting the importance of this theory not only from the point of view of \(\FF_1\)-geometry.

The main topic of this paper is to study the relation between this new theory of bands and two other important approaches to \(\FF_1\)-geometry: relative algebraic geometry and hyperring schemes.

Relative algebraic geometry was developed by B.~To\"en and M.~Vaqui\'e~\cite{ToenVaquie}, who generalized scheme theory by replacing the category of modules \(\Mod_{\ZZ}\) and the category of commutative rings \(\mathbf{CRings}\) with a general closed symmetric monoidal category that is complete and cocomplete, together with its category of commutative monoids. This framework not only recovers classical scheme theory, but also includes monoid schemes, as shown by A.~Vezzani~\cite{Vezzani}, and non-Archimedean analytic geometry, as shown by O.~Ben-Bassat and K.~Kremnizer~\cite{Bassat}. As the first main result of this paper, we prove that the theory of relative algebraic geometry also recovers the theory of band schemes when we consider the category of band modules as the base symmetric monoidal category.

Hyperrings are generalizations of commutative rings with multivalued addition. They were first considered by M.~Krasner in his study of limits of local fields~\cite{Krasner, Krasner2}. More recently, A.~Connes and C.~Consani showed that the adele class space \(\mathbb{A}_{\mathbb{K}}/\mathbb{K}^\times\) of a global field \(\mathbb{K}\) carries a natural hyperring structure~\cite{Connes1, Connes2}, thereby linking hyperrings to \(\FF_1\)-geometry and renewing interest in the subject.

In the second part of this paper, we generalize several results on hyperrings and affine hyperring schemes to the case of bands. The main result is the generalization of \emph{strict} morphisms of hypermodules to the setting of band modules, together with the proof that the category \(\BMod_B\) of band modules over a band \(B\) is proto-exact, generalizing the corresponding result of J.~Jun~\cite{Jaiung} for hypermodules.

\subsection*{Detailed outline of the paper and the main results} 

Section \ref{sec:preliminaries} consists of two parts. In the first part, we give a brief review of the definitions and basic properties of hyperrings and affine hyperring schemes, following R.~Procesi Ciampi and R.~Rota~\cite{Rota}, and J.~Jun~\cite{Jaiung}. A hyperring is a generalization of commutative rings in which the addition may be multivalued. The algebraic theory of hyperrings has several similarities with the theory of commutative rings. For instance, it has a well-behaved ideal theory, quotient hyperrings, and a localization theory. Unfortunately, it does not have a well-defined tensor product, and its scheme theory does not satisfy some important properties. For instance, let \(H\) be a hyperring and \((\Spec H, \OO_{\Spec H})\) be an affine hyperring scheme. Then J.~Jun~\cite{Jaiung} observed that it is not always true that 
\begin{equation*}
    \OO_{\Spec H}(\Spec H) \cong H.
\end{equation*}

In the second part of this section, we review the definition of proto-exact categories. They were defined by T.~Dyckerhoff and M.~Kapranov in \cite{Kapranov} as a non-additive generalization of Quillen exact categories. We also review the result of J.~Jun~\cite{Jaiung2} concerning the proto-exactness of the category of hypermodules over a hyperring. 

In Section \ref{sec:Bands and Band Schemes}, we give a more detailed review of the theory of bands and band schemes. Briefly speaking, a band is a pointed commutative monoid \(B\) equipped with a suitable ideal \(N_B\), called the \emph{null ideal}, of its associated semiring 
\begin{equation*}
    B^+ \coloneqq \NN[B]/ \left< 0_B \sim \text{empty sum} \right>.
\end{equation*}
We present some important examples, including the initial object \(\FF_1^\pm\), which plays the role of the \emph{field with one element} for bands (although it has three elements). The category of bands is complete and cocomplete. We review the constructions of limits and colimits, localization, and tensor products. After that, we review the theory of \textit{geometric} band schemes and some of its important properties. 

Section \ref{sec:Modules over Bands} is the heart of this paper. In this section, we develop the theory of modules over a band and prove several fundamental properties. In particular, we prove our first main theorem. 

\begin{mainthm}[Theorems \ref{thm:symmetricmonoidalcategory} and \ref{thm:eqofcategories}]\label{main theorem 1}
    Let \(B\) be a band and denote by \(\BMod_B\) the category of modules over \(B\). Then the following hold:
    \begin{enumerate}
        \item[\((1)\)] \((\BMod_B, \otimes_B, B)\) is a cosmos, that is, a complete and cocomplete closed symmetric monoidal category.
        \item[\((2)\)] The category of commutative monoid objects in \(\BMod_B\) is equivalent to the category \(\BAlg_B\) of band algebras over \(B\).
    \end{enumerate}
\end{mainthm}

This leads to the results of the next section. In Section \ref{sec:Relative Schemes}, we review the theory of schemes relative to a closed symmetric monoidal category that is complete and cocomplete, as defined by B.~To\"en and M.~Vaqui\'e in \cite{ToenVaquie}. By Theorem \ref{main theorem 1}, we obtain the notion of schemes relative to the category \(\BMod_{\FF_1^\pm}\). The main result of this section is the equivalence between the two theories of band schemes. 

\begin{mainthm}[Theorem \ref{thm:bandschemes=relativeschemes}]\label{main theorem 2}
    The category \(\BSch^{\mathrm{geom}}\) of geometric band schemes is equivalent to the category \(\Sch^{\mathrm{rel}}_{\BMod_{\FF_1^\pm}}\) of schemes relative to \(\BMod_{\FF_1^\pm}\).
\end{mainthm}

In the second half of the paper, we generalize several results on hyperrings to the case of bands. In Section \ref{sec: hyperring schemes}, we first characterize the subcategory of bands that is equivalent to the category of hyperrings. Then, we study the generalization of affine hyperring schemes to the case of bands. We see that although bands generalize hyperrings, their scheme theories are not compatible. The reason is that, on one hand, geometric band schemes use the notion of prime \(m\)-ideals, while on the other hand, affine hyperring schemes use prime ideals corresponding to prime \(k\)-ideals. This leads to several differences. The main result of this section concerns recovering the band \(B\) from the structure sheaf of the affine \(k\)-scheme \(\Spec^k B\), which is defined using the notion of prime \(k\)-ideals. 

\begin{mainthm}[Theorem \ref{main theorem hypbands}]\label{main theorem 3}
    Let \(B\) be a band satisfying the following conditions:
    \begin{enumerate}
    \item[\((D1)\)] If \(b\cdot \left(\sum_i a_i \right) \in N_B\) and \(b\neq 0\), then \(\sum_i a_i \in N_B\), and
    \item[\((D2)\)] \(\langle x \rangle_k = B \cdot x\) for every \(x \in B\),
\end{enumerate}
    and let \(X = \Spec^k(B)\) and \(\OO^k_X\) be the sheaf of bands as in Definition \ref{k-sheaf}. Then, \(\OO^k_X(D^k(f))\) is isomorphic to the band \(B_f\). In particular, if \(f = 1\), we have
        \begin{equation*}
             \OO^k_X(X) \cong B.
        \end{equation*}
\end{mainthm}

Finally, in Section \ref{sec: protoexact}, we generalize the notion of \emph{strict} morphisms of hypermodules to the case of morphisms of band modules.

\begin{definition}
    A morphism of \(B\)-modules \(f : M \to N\) is called \emph{strict} if for every sum \(\sum m_i \in M^{+}\) such that \(\sum f(m_i) \in I_N\), there exists a finite sequence \((c_j)_j\) of elements of \(\ker(f)\) such that \(\sum m_i + \sum c_j \in I_M\). Denote by \(\mathfrak{M}_{\BMod_B}\) (resp. \(\mathfrak{E}_{\BMod_B}\)) the class of strict monomorphisms (resp. strict epimorphisms) in \(\BMod_B\).
\end{definition}

The main result of this section is the generalization of the corresponding result of J.~Jun~\cite{Jaiung2} for the category of hypermodules over a hyperring.

\begin{mainthm}[Corollary \ref{main thm protoexactness}]\label{main theorem D}
      \((\BMod_B, \mathfrak{M}_{\BMod_B}, \mathfrak{E}_{\BMod_B})\) is a proto-exact category.
\end{mainthm}

This provides another example of a proto-exact category, contributing to the growing list of such structures studied in the literature~\cite{Jaiung3, Jaiung1, Jaiung2}.

\subsection*{Acknowledgments}
The author is especially grateful to Yuichiro Taguchi and Naganori Yamaguchi for their constant support during the writing of this article. He is also grateful to all the members of the laboratory for attending his seminars.

%% file: Preliminaries.tex
\section{Preliminaries}\label{sec:preliminaries}

\subsection{Hyperrings and hyperring schemes}

Let \(H\) be a nonempty set, and denote by \(\mathcal{P}(H)^*\) the collection of all nonempty subsets of \(H\). A \emph{hyperoperation} on \(H\) is a map
\begin{equation*}
    \boxplus \colon H \times H \longrightarrow \mathcal{P}(H)^*.
\end{equation*}

For nonempty subsets \(A, B \subseteq H\), we define their sum by
\begin{equation*}
A \boxplus B \coloneqq \bigcup_{a \in A,\; b \in B} (a \boxplus b).
\end{equation*}

We adopt the convention of writing \(A \boxplus z\) in place of \(A \boxplus \{z\}\). Moreover, if \(x \boxplus y\) is a singleton \(\{z\}\), we simply write \(x \boxplus y = z\).

\begin{definition}[Hyperrings]
A \emph{hyperring} is a tuple \(R = (R, \boxplus, \cdot, 0_R, 1_R)\), where \(R\) is a nonempty set, \(\boxplus\) is a hyperoperation on \(R\), \(\cdot\) is a binary operation on \(R\), and \(0_R, 1_R \in R\), satisfying the following conditions:
\begin{enumerate}
    \item \((R, \cdot, 1_R)\) is a commutative monoid.
    \item \(x \boxplus y = y \boxplus x\) for every  \(x, y \in R\).
    \item \((x \boxplus y) \boxplus z = x \boxplus (y \boxplus z)\) for every \(x, y, z \in R\).
    \item \(0_R\) is the unique element satisfying \(x \boxplus 0_R = x\) for every \(x \in R\).
    \item For every \(x \in R\), there exists a unique element \(y_x \in R\) such that \(0_R \in x \boxplus y_x\). We denote \(y_x\) by \(-x\).
    \item For every \(x, y, z \in R\) the following holds: \[x \in y \boxplus z\;\text{ if and only if }\; z \in x \boxplus (-y).\]
    \item For every \(x, y, z \in R\) the following holds:
    \[
        x \cdot (y \boxplus z) = (x \cdot y) \boxplus (x \cdot z)
        \quad \text{and} \quad
        (x \boxplus y) \cdot z = (x \cdot z) \boxplus (y \cdot z).
    \]
    \item  \(x \cdot 0_R = 0_R\) for every \(x \in R\).
    \item \(0_R \neq 1_R\).
\end{enumerate}
A hyperring \(R\) is called a \emph{hyperfield} if \((R \setminus \{0_R\}, \cdot, 1_R)\) is an abelian group.
\end{definition}

\begin{definition}[Morphisms of hyperrings]
Let \((R_1, \boxplus_1, \cdot_1)\) and \((R_2, \boxplus_2, \cdot_2)\) be hyperrings. A map \(f \colon R_1 \to R_2\) is called a \emph{morphism of hyperrings} if, for every \(a, b \in R_1\), the following conditions hold:
\begin{enumerate}
    \item \(f(a \boxplus_1 b) \subseteq f(a) \boxplus_2 f(b)\),
    \item \(f(a \cdot_1 b) = f(a) \cdot_2 f(b)\),
    \item \(f(0_{R_1}) = 0_{R_2}\) and \(f(1_{R_1}) = 1_{R_2}\).
\end{enumerate}
A morphism \(f\) is said to be \emph{strict} if, for every \(a, b \in R_1\), one has
\[
f(a \boxplus_1 b) = f(a) \boxplus_2 f(b).
\]
\end{definition}

We denote by \(\mathbf{HypRings}\) the category of hyperrings with morphisms as defined above.

\begin{definition}[Category of Hypermodules]
    Let \(R\) be a hyperring. A hypermodule over \(R\) is a tuple \((M, \oplus, 0_M)\), where \(M\) is a nonempty set, \(\oplus\) is a hyperoperation on \(M\) and \(0_M \in M\), together with a map  
    \begin{align*}
        R \times M &\longrightarrow M\\
        (a, m) &\longmapsto a m
    \end{align*}
    satisfying the following properties:
    \begin{enumerate}
        \item \(1m = m\) and \(0_Hm = 0_M\) for every \(m \in M\).
        \item \((a\cdot b)m = a(bm)\) for every \(a,b \in R\) and \(m \in M\).
        \item \((a \boxplus b)m = am \oplus bm\) for every \(a,b \in R\) and \(m \in M\).
        \item \(a(m \oplus n) = am \oplus  an\) for every \(a \in R\) and \(m,n \in M\).
    \end{enumerate}

    A map \(f: M \to N\) between hypermodules over \(R\) is a morphism if it satifies, for every \(a \in R\) and \(m, n \in M\), the following conditions:
    \begin{enumerate}
        \item \(f(m \oplus_M n) \subseteq f(m) \oplus_N f(n)\).
        \item \(f(am) = af(m)\).
    \end{enumerate}
    A morphism \(f\) is said to \emph{strict} if, for every \(m,n \in M\), one has
    \begin{equation*}
        f(m \oplus_M n) = f(m) \oplus_N f(n).
    \end{equation*}
\end{definition}

As an important example, we present a construction of a hyperring as a quotient of a commutative ring by a subgroup of its group of units.

\begin{example}[Proposition 2.6 of \cite{Connes1}]\label{ex: quotient hyperring}
    Let \(A\) be a commutative ring, and denote by \(A^\times\) its group of units. Let \(G \subseteq A^\times\) be a subgroup, and let
    \[
        A/G \coloneqq \{\, aG \mid a \in A \,\}
    \]
    be the set of multiplicative cosets. The \emph{quotient hyperring} of \(A\) by \(G\) is the hyperring \((A/G, \odot, \boxplus)\), where the operations are defined as follows:
    \begin{align*}
        \odot \colon A/G \times A/G &\longrightarrow A/G,\\
        (xG, yG) &\longmapsto xyG,
    \end{align*}
    and
    \begin{align*}
        \boxplus \colon A/G \times A/G &\longrightarrow \mathcal{P}(A/G) \setminus \{\emptyset\},\\
        (xG, yG) &\longmapsto \{\, qG \mid q = xt + ys \text{ for some } t,s \in G \,\}.
    \end{align*}
\end{example}

\begin{remark}
    Let \(\mathbb{K}\) be a global field and let \(\mathbb{A}_{\mathbb{K}}\) denote its ring of ad\`eles. Since \(\mathbb{K}^\times\) embeds as a subgroup of the group of units \(\mathbb{A}_{\mathbb{K}}^\times\), it follows from Example~\ref{ex: quotient hyperring} that the quotient \(\mathbb{A}_{\mathbb{K}}/\mathbb{K}^\times\) carries a natural structure of a hyperring.
\end{remark}

The theory of hyperrings admits a notion of ideals, called \emph{hyperideals}, which satisfy properties analogous to those in classical ring theory. 

\begin{definition}[Hyperideals]\label{hyperideals}
Let \(R\) be a hyperring. A nonempty subset \(I \subseteq R\) is called a \emph{hyperideal} if it satisfies the following conditions:
\begin{enumerate}
    \item \(a \boxplus b \subseteq I\) for every \(a, b \in I\).
    \item \(r \cdot a \in I\) for every \(r \in R\) and every \(a \in I\).
\end{enumerate}

A proper hyperideal \(\mathfrak{p} \subsetneq R\) is called \emph{prime} if, whenever \(x \cdot y \in \mathfrak{p}\), one has \(x \in \mathfrak{p}\) or \(y \in \mathfrak{p}\).

A proper hyperideal \(\mathfrak{m} \subsetneq R\) is called \emph{maximal} if it is maximal among proper hyperideals of \(R\), that is, if \(\mathfrak{m} \subseteq J \subsetneq R\) for a hyperideal \(J\), then \(J = \mathfrak{m}\).
\end{definition}

\begin{remark}
    B.~Davvaz and A.~Salasi proved in \cite[Propositions 2.12 and 2.13]{Davvaz} that every maximal hyperideal is prime, and that for every hyperring \(R\) and every proper hyperideal \(I \subsetneq R\), there exists a maximal hyperideal of \(R\) containing \(I\).
\end{remark}

\subsubsection{Localization}

Let \(R\) be a hyperring and let \(S \subseteq R\) be a multiplicative subset, that is, \(1 \in S\) and \(x\cdot y \in S\) whenever \(x,y \in S\). We define the \emph{localization} \(S^{-1}R\) as the hyperring \(((R \times S)/\sim, \oplus, \odot)\), where \(\sim\) is the equivalence relation defined by
\[
    (r_1, s_1) \sim (r_2, s_2) \;\Longleftrightarrow\; \text{there exists } x \in S \text{ such that } x\cdot r_1\cdot s_2 = x\cdot r_2\cdot s_1,
\]
and denote the equivalence class of \((r,s)\) by \(\frac{r}{s}\). The hyperaddition \(\boxplus\) on \(S^{-1}R\) is defined by
\[
    \frac{r_1}{s_1} \oplus \frac{r_2}{s_2}
    \coloneqq
    \left\{ \frac{y}{s_1\cdot s_2} \;\middle|\; y \in (r_1\cdot  s_2) \boxplus (r_2\cdot s_1) \right\},
\]
and the multiplication is defined by
\[
    \frac{r_1}{s_1} \cdot \frac{r_2}{s_2}
    \coloneqq
    \frac{r_1\cdot r_2}{s_1\cdot s_2}.
\]

It satisfies properties analogous to those of the localization of commutative rings.

\begin{proposition}[Proposition 2.14 of \cite{Jaiung}]
    Let \(R\) be a hyperring and let \(S \subseteq R\) be a multiplicative subset. Then the following hold:
    \begin{enumerate}
        \item For any hyperideal \(I \subseteq R\), the set
        \[
            S^{-1}I \coloneqq \left\{ \frac{i}{s} \;\middle|\; i \in I,\ s \in S \right\}
        \]
        is a hyperideal of \(S^{-1}R\).
        
        \item If \(\mathfrak{p}\) is a prime hyperideal of \(R\) such that \(\mathfrak{p} \cap S = \emptyset\), then \(S^{-1}\mathfrak{p}\) is a prime hyperideal of \(S^{-1}R\).
        
        \item If \(S = R \setminus \mathfrak{p}\) for some prime hyperideal \(\mathfrak{p}\) of \(R\), then \(S^{-1}R\) has a unique maximal hyperideal, namely \(S^{-1}\mathfrak{p}\).
    \end{enumerate}
\end{proposition}

\begin{definition}[Hyperdomain, {\cite[Definition 4.21]{Jaiung}}]
    A hyperring \(R\) is called a \emph{hyperdomain} if \(R\) does not have zero-divisors, that is, for \(x, y \in R\), if \(x\cdot y = 0\) then either \(x = 0\) or \(y = 0\).
\end{definition}

Let \(R\) be a hyperdomain. We denote by \(\Frac(R)\) the localization of \(R\) at the multiplicative set \(R \setminus \{0\}\). For every multiplicative subset \(S \subseteq R\), the canonical morphisms
\[
    R \longrightarrow S^{-1}R \quad \text{and} \quad S^{-1}R \longrightarrow \Frac(R)
\]
are strict and injective. In particular, we may view \(S^{-1}R\) as a hyperring extension of \(R\) contained in \(\Frac(R)\).

Moreover, if \(S_f \coloneqq \{1, f, f^2, \dots \}\) for some \(f \in R\), the hyperring \(S_f^{-1}R\) is denoted by \(R_f\). For every prime hyperideal \(\mathfrak{p}\) of \(R\), the hyperring \((R \setminus \mathfrak{p})^{-1}R\) is denoted by \(R_{\mathfrak{p}}\).

\subsubsection{Affine hyperring schemes}

 We follow \cite[Section 4.2]{Jaiung} for the definition of affine schemes of hyperrings.

\begin{definition}[Spectrum of a hyperring]\label{spec hyperring}
    Let \(R\) be a hyperring. The \emph{spectrum} of \(R\), denoted by \(\Spec R\), is the set of prime hyperideals of \(R\), endowed with the \emph{Zariski topology}, defined as follows: a subset \(A \subseteq \Spec R\) is closed if and only if
    \[
        A = V(I) \coloneqq \left\{ \mathfrak{p} \in \Spec R \;\middle|\; I \subseteq \mathfrak{p} \right\}
    \]
    for some hyperideal \(I \subseteq R\). 
\end{definition}

\begin{remark}
    For every \(f \in R\), denote by \(D(f)\) the set of prime hyperideals \(\mathfrak{p}\) of \(R\) not containing \(f\). As in the case of commutative rings, the family \(\{D(f)\}_{f \in R}\) forms a basis for the Zariski topology of \(\Spec R\).
\end{remark}

\begin{definition}[Structural sheaf]\label{hyperrings, structural sheaf}
    Let \(R\) be a hyperring and let \(X = \Spec R\). For an open subset \(U \subseteq X\), we define
    \[
        \OO_X(U) \coloneqq \left\{\, s : U \to \bigsqcup_{\mathfrak{p} \in U} R_{\mathfrak{p}} \;\middle|\; (\star_1), (\star_2) \,\right\},
    \]
    where \((\star_1)\) and \((\star_2)\) are the following conditions:
    \begin{enumerate}
        \item[\((\star_1)\)] For every \(\mathfrak{p} \in U\), one has \(s(\mathfrak{p}) \in R_{\mathfrak{p}}\).
        \item[\((\star_2)\)] For every \(\mathfrak{p} \in U\), there exist an open neighborhood \(V \subseteq U\) of \(\mathfrak{p}\) and elements \(a,f \in R\) such that \(f \notin \mathfrak{q}\) for all \(\mathfrak{q} \in V\), and
        \[
            s(\mathfrak{q}) = \frac{a}{f} \in R_{\mathfrak{q}} \quad \text{for all } \mathfrak{q} \in V.
        \]
    \end{enumerate}
    The multiplication on \(\OO_X(U)\) is defined pointwise:
    \[
        (s \cdot t)(\mathfrak{p}) \coloneqq s(\mathfrak{p}) \cdot t(\mathfrak{p}).
    \]
    The hyperaddition is defined by
    \[
        s \boxplus t \coloneqq \left\{\, r \in \OO_X(U) \;\middle|\; r(\mathfrak{p}) \in s(\mathfrak{p}) \boxplus t(\mathfrak{p}) \text{ for all } \mathfrak{p} \in U \,\right\}.
    \]
\end{definition}

\begin{remark}
    As pointed out by J.~Jun in \cite[Remark 4.20]{Jaiung}, it is not known whether \(\OO_X(U)\) is a hyperring for every \(U \subseteq X\).
\end{remark}

\begin{proposition}[Theorem 4.23 of \cite{Jaiung}]\label{main theorem hyperrings}
    Let \(R\) be a hyperdomain, \(K = \Frac(R)\), and \(X = \Spec(R)\). Let \(\OO_X\) be the structural sheaf of \(X\) as defined in Definition \ref{hyperrings, structural sheaf}. Then, the following hold:
    \begin{enumerate}
        \item \(\OO_X(D(f))\) is a hyperring isomorphic to \(R_f\). In particular, if \(f = 1\), we have 
        \begin{equation*}
           \OO_X(X) \cong R.
        \end{equation*}
        \item For each open subset \(U \subseteq X\), \(\OO_X(U)\) is a hyperring. Moreover, by considering the canonical map \(R_f \hookrightarrow K \), we have 
        \begin{equation*}
            \OO_X(U) \cong \bigcap_{D(f) \subseteq U} \OO_X(D(f)).
        \end{equation*}
        \item For each \(\mathfrak{p} \in X\), the stalk \(\OO_{X,\mathfrak{p}}\) exists and is isomorphic to \(R_\mathfrak{p}\).
    \end{enumerate}
\end{proposition}

\begin{remark}
    As pointed out by J. Jun in \cite[Example 4.24]{Jaiung}, the property 
    \begin{equation}\label{12345}
        \OO_{\Spec R}(\Spec R) \cong R
    \end{equation}
    do not hold for every hyperring \(R\). As an open problem we have the question if \eqref{12345} holds for every connected hyperring \(R\).
\end{remark}

\subsection{Proto-exact categories}

Proto-exact categories were introduced by T.~Dyckerhoff and M.~Kapranov in \cite{Kapranov} as a generalization of Quillen exact categories, allowing one to define \(K\)-theory for non-additive categories. Before giving the definition, we fix the following notation. Let \(\mathscr{E}\) be a category. A commutative square in \(\mathscr{E}\)
\begin{equation*}
    \begin{tikzcd}
        N \arrow[r] \arrow[d] & M \arrow[d] \\
        N' \arrow[r] & M'
    \end{tikzcd}
\end{equation*}
is said to be \textit{bi-Cartesian} if it is both Cartesian and co-Cartesian.

\begin{definition}[Proto-exact category {\cite[Definition 2.4.2]{Kapranov}}]
    A \textit{proto-exact category} is a category \(\mathscr{E}\) equipped with two classes of morphisms \(\mathfrak{M}\) and \(\mathfrak{E}\) whose elements are called \textit{admissible monomorphisms} and \textit{admissible epimorphisms} such that the following conditions are satisfied:
    \begin{enumerate}
        \item[(PE1)] \(\mathscr{E}\) has a zero object \(0\). Every morphisms \(0 \rightarrow M\) is in \(\mathfrak{M}\) and every morphism \(M \rightarrow 0\) is in \(\mathfrak{E}\).

        \item[(PE2)] The classes \(\mathfrak{M}\) and \(\mathfrak{E}\) contain every isomorphism and are closed under composition.

        \item[(PE3)] A commutative square in \(\mathscr{E}\) as 
        \begin{equation}\label{diag:bicartesian0}
            \begin{tikzcd}
            N \arrow[r, hook, "i"] \arrow[d, two heads,swap, "j"] & M \arrow[d, two heads, "j'"]\\
            N' \arrow[r, hook, "i'"] & M'
        \end{tikzcd}
        \end{equation}
        with \(i,i'\) in \(\mathfrak{M}\) and \(j,j'\) in \(\mathfrak{E}\) is Cartesian if and only if it is coCartesian.

        \item[(PE4)] Any diagram in \(\mathscr{E}\)
        \[
        \begin{tikzcd}
             & M \arrow[d, two heads, "j'"]\\
            N' \arrow[r, hook, "i'"] & M'
        \end{tikzcd}
        \]
        with \(i' \in \mathfrak{M}\) and \(j' \in \mathfrak{E}\), can be completed to a biCartesian square (\ref{diag:bicartesian0}) with \(i \in \mathfrak{M}\) and \(j \in \mathfrak{E}\).

        \item[(PE5)] Any diagram in \(\mathscr{E}\)
        \[
        \begin{tikzcd}
            N \arrow[r, hook, "i"] \arrow[d, two heads,swap, "j"] & M\\
            N'  &
        \end{tikzcd}
        \]
        with \(i \in \mathfrak{M}\) and \(j \in \mathfrak{E}\), can be completed to a biCartesian square (\ref{diag:bicartesian0}) with \(i' \in \mathfrak{M}\) and \(j' \in \mathfrak{E}\).
    \end{enumerate}
\end{definition}

\begin{example}
    \begin{enumerate}
        \item Every Quillen exact category is proto-exact. 
        \item C.Eppolito, J. Jun, and M. Szczesny proved in \cite{Jaiung3} that the category of pointed matroids is proto-exact.
        \item J. Jun, M. Szczesny, J. Tolliver proved in \cite{Jaiung1} that both the category of modules over a semiring and hypermodules over a hyperring are proto-exact. 
        \item J. Jun, A. Sistko, and C. Wright proved in \cite{Jaiung2} that the category of matroids over a perfect idyll is proto-exact.
    \end{enumerate}
\end{example}

Let \((\mathscr{E}, \mathfrak{M}_\mathscr{E}, \mathfrak{E}_\mathscr{E})\) be a proto-exact category and for every \(n \in \ZZ_{\geq 1}\), let \(\mathcal{W}_n(\mathcal{E})\) be the category where the objects are diagrams in \(\mathcal{E}\)

\begin{equation}\label{diagram K group}
    \begin{tikzcd}
        0 \arrow[r, hook] & A_{0,1} \arrow[r, hook] \arrow[d, two heads] & A_{0,2} \arrow[r, hook] \arrow[d, two heads] & \cdots \arrow[r, hook] \arrow[d, two heads] & A_{0,n-1} \arrow[r, hook] \arrow[d, two heads] & A_{0,n} \arrow[d, two heads]\\
        & 0 \arrow[r, hook] & A_{1,2} \arrow[r, hook] \arrow[d, two heads] & \cdots \arrow[r, hook] \arrow[d, two heads] & A_{1,n-1} \arrow[r, hook] \arrow[d, two heads] & A_{1,n} \arrow[d, two heads]\\
        & & 0 \arrow[r, hook] & \ddots \arrow[r, hook] \arrow[d, two heads] & \cdots \arrow[r, hook] \arrow[d, two heads] & \vdots \arrow[d, two heads] \\
        & & & 0 \arrow[r, hook] & A_{n-1, n-1} \arrow[d, two heads] \arrow[r, hook] & A_{n-1, n} \arrow[d, two heads]\\ 
        & & & & 0 \arrow[r, hook] & A_{n-1,n} \arrow[d, two heads]\\
        & & & & & 0.
    \end{tikzcd}
\end{equation}
 satisfying the following conditions:
\begin{enumerate}
    \item[\((W1)\)] All horizontal morphisms are in \(\mathfrak{M}_\mathscr{E}\) and all vertical morphisms are in \(\mathfrak{E}_\mathscr{E}\).
    \item[\((W2)\)] Each square in the diagram is bi-Cartesian.
\end{enumerate}
and the morphisms of \(\mathcal{W}_n(\mathcal{E})\) are morphisms of diagrams. Let \(\mathcal{S}_n(\mathcal{E})\) be the subcategory of \(\mathcal{W}_n(\mathcal{E})\) consisting of all objects and their isomorphisms. The family \(\{\mathcal{S}_n(\mathcal{E})\}_n\) forms a simplicial object in the category of groupoids, together with the \emph{face} and \emph{degeneracy} functors
\begin{equation*}
    \partial_i : \mathcal{S}_n(\mathcal{E}) \to \mathcal{S}_{n-1}(\mathcal{E}),\quad 
    s_i : \mathcal{S}_n(\mathcal{E}) \to \mathcal{S}_{n+1}(\mathcal{E}) \quad (i = 0, \dots, n).
\end{equation*}
Here, \(\partial_i\) is obtained by omitting the \(i\)-th row and the \(i\)-th column of diagram~\eqref{diagram K group}, while \(s_i\) is obtained by inserting a new row between the \(i\)-th and \((i+1)\)-st rows, identical to the \(i\)-th row and connected by identity morphisms, and similarly inserting a new column between the \(i\)-th and \((i+1)\)-st columns, also identical to the \(i\)-th column and connected by identity morphisms.

\begin{definition}[\(K\)-group of proto-exact categories]
    The \(n\)-th \(K\)-group of \(\mathscr{E}\) is defined by
    \begin{equation*}
        K_n(\mathscr{E}) \coloneqq \pi_{n+1}(|\mathcal{S}_\bullet (\mathcal{E})|),
    \end{equation*}
    where \(|\mathcal{S}_\bullet (\mathcal{E})|\) denotes the geometric realization of \(\mathcal{S}_\bullet (\mathcal{E})\) (see \cite[Section 1.2]{Kapranov}).
\end{definition}

There is an explicit description of the group \(K_0(\mathscr{E})\) in terms of admissible sequences.

\begin{definition}[Admissible sequence]
An \emph{admissible sequence} in a proto-exact category \((\mathscr{E}, \mathfrak{M}, \mathfrak{E})\) is a bi-Cartesian diagram
\begin{equation}\label{diag:bicartesian0}
\begin{tikzcd}
N \arrow[r, hook, "i"] \arrow[d, two heads] & M \arrow[d, two heads, "j'"]\\
0 \arrow[r, hook] & M'
\end{tikzcd}
\end{equation}
with \(i \in \mathfrak{M}\) and \(j' \in \mathfrak{E}\). We simply denote such a diagram by
\[
N \hookrightarrow M \twoheadrightarrow M'.
\]
\end{definition}

Let \(\Iso(\mathscr{E})\) be the set of isomorphism classes of objects of \(\mathscr{E}\). Then \(K_0(\mathscr{E})\) is isomorphic to the quotient of the free group generated by the symbols \([M]\), where \(M \in \Iso(\mathscr{E})\), modulo the relations \([M] = [N][M']\) for every admissible sequence \(N \hookrightarrow M \twoheadrightarrow M'\).

Moreover, when \(\mathscr{E}\) has finite coproducts and admits split admissible sequences of the form
\[
N \hookrightarrow N \oplus M' \twoheadrightarrow M',
\]
the group \(K_0(\mathscr{E})\) is abelian and isomorphic to the quotient
\[
\mathbb{Z}[\Iso(\mathscr{E})]/\sim,
\]
where \(\sim\) is the equivalence relation generated by the relations \([M] \sim [N] + [M']\) for every admissible sequence \(N \hookrightarrow M \twoheadrightarrow M'\).

%% file: Bands.tex
\section{Bands and Band Schemes}\label{sec:Bands and Band Schemes}

In this section, we review the notions from the theory of bands and band schemes that are needed in the sequel. We begin by presenting the definition and basic properties of bands. We then recall the geometric definition of band schemes, that is, a pair of a topological space together with a sheaf of bands that locally is isomorphic to some affine band schemes. For a complete treatment, we refer the reader to the original source \cite{Lorscheid3}.

\subsection{Bands}

    The idea of a band is to consider a multiplicative monoid equipped with a suitable ideal of formal sums, encoding additive relations. In this way, bands provide a framework in which additive relations can be studied without requiring the existence of an actual additive operation.

\begin{definition}
A \textit{pointed monoid} \(M\) is a commutative monoid, written multiplicatively, together with a distinguished element \(0_M\) (called the zero of \(M\)) such that \(x \cdot 0_M = 0_M\) for all \(x \in M\).
For every pointed monoid \(M\), we denote by \(M^{+}\) the semiring
\[
M^{+} \coloneq \NN[M]/\langle 0_M \sim \text{the empty sum} \rangle,
\]
that is, the semiring whose elements are finite formal sums \(\sum a_i\) with \(a_i \in M\). Then, \(0_M\) is the identity of the addition of \(M^{+}\).
\end{definition}

\begin{definition}[The Category of Bands] \mbox{}
    \begin{enumerate}
        \item A \textit{band} is a pair \(B=(B,N_B)\) consisting of a pointed monoid \(B\) and an ideal \(N_B\) of \(B^{+}\), called the \emph{null set} of \(B\), such that for every \(x \in B\) there exists a unique \(y_x \in B\) with \(x + y_x \in N_B\). We denote \(y_x\) by \(-x\).
        \item Let \(B\) and \(C\) be bands. A map \(f \colon B \to C\) is called a \textit{morphism of bands} if it satisfies:
\begin{enumerate}
\item \(f(ab) = f(a)\, f(b)\) for all \(a,b \in B\);
\item \(f(0_B)=0_C\) and \(f(1_B)=1_C\);
\item If\/ \(\sum_i b_i \in N_B\), then\/ \(\sum_i f(b_i) \in N_C\).
\end{enumerate}
    \end{enumerate}
We denote by \(\Band\) the category of bands with the above morphisms.
\end{definition}

\begin{definition}[The Category of Idylls]
An \textit{idyll} is a band \(B=(B,N_B)\) such that \(B \setminus \{0_B\}\) is a group and \(0_B \neq 1_B\). A \textit{morphism of idylls} is a morphism of bands. We denote this category by \(\Idyll\).
\end{definition}

\begin{example} \mbox{}
\begin{enumerate}
\item The zero band \(0=\{*\}\), for which \(*=1\), is the terminal object in the category of bands.

\item The \textit{regular partial field} \(\FF_1^{\pm}=\{0,1,-1\}\) is the idyll with the natural multiplication table (same as \(\FF_3\)) and null set
\[
N_{\FF_1^{\pm}}
\coloneq
\{\,0,\; 1+(-1),\; 1+(-1)+1+(-1),\; 1+(-1)+1+(-1)+1+(-1),\; \dots \,\}.
\]
It is the initial object in the categories of bands and of idylls.

\item The \textit{Krasner hyperfield} \(\KK=\{0,1\}\) is the idyll with the natural multiplication and null set
\[
N_{\KK}
\coloneq
\{\,0,\; 1+1,\; 1+1+1,\; \dots \,\}.
\]
It is the terminal object in the category of idylls.
\end{enumerate}
\end{example}

\begin{remark}\label{rmk:null-ideal-not-unique}
\begin{enumerate}
    \item For every pointed monoid \(M\), one can define a null set in \(M^{+}\) to make \(M\) into a band. Namely, one may take the null set of \(M\) generated by the element \(1+1 \in M^+\).  
    \item One issue that will arise repeatedly in this paper is that the null set of a pointed monoid is not uniquely determined. As a consequence, there may exist bijective morphisms between bands that are not isomorphisms. As a simple example, consider the bijective map \(f \colon \FF_1^{\pm} \to \FF_3\) between the regular partial field \(\FF_1^{\pm}\) and the field \(\FF_3\), which is considered as an idyll with null set
\[
N_{\FF_3}
\coloneq
\left\{
\sum a_i \in (\FF_3)^+
\;\middle|\;
\sum a_i = 0 \in \FF_3
\right\},
\]
where the first \(\sum\) is the formal addition in \((\FF_3)^+\) and the second \(\sum\) is the ordinary addition in \(\FF_3\). Observe that the element \(1+1+1\) belongs to \(N_{\FF_3}\) but not to \(N_{\FF_1^{\pm}}\). Hence, the inverse map of \(f\) is not a morphism of idylls.
\end{enumerate}
\end{remark}

\subsubsection{Examples}
In this subsection, we present two examples of algebraic structures that will be used throughout this paper and that naturally give rise to bands. We note that there are many more examples, such as partial fields, pastures, and fuzzy rings. We refer the interested reader to the original paper \cite{Lorscheid3} for more examples.

\begin{example}[Rings]\label{Ex: ring as band} Let \(R\) be a commutative ring. As in the case of \(\FF_3\) in Remark \ref{rmk:null-ideal-not-unique}, \(R\) becomes a band with its original multiplicative structure and with null set defined by the additive structure of \(R\), namely 
\begin{equation*}
    N_R \coloneq \left\{ \sum_i r_i \in R^+ \;\middle|\;\sum_i r_i = 0 \text{ in } R  \right\},
\end{equation*}
where the first \(\sum\) is the formal addition in \(R^+\) and the second \(\sum\) is the ordinary addition in \(R\). Every ring homomorphism is a morphism of bands and, conversely, every band morphism between commutative rings is a ring homomorphism. In this way, we obtain a fully-faithful functor from the category of commutative rings to the category of bands.
\end{example}

\begin{example}[Hyperrings]\label{Ex: hyperring as band}
Let \((H, \boxplus,\cdot)\) be a commutative hyperring in the sense of Krasner. Then \(H\) becomes a band with its original multiplicative structure and with null set defined by
\begin{equation*}
    N_H \coloneq \left\{ \sum_i a_i \in H^+ \;\middle|\;0 \in \op_i a_i \right\}.
\end{equation*}
Every morphism of hyperrings is a morphism of bands and, conversely, every band morphism between hyperrings is a morphism of hyperrings. In this way, we obtain a fully-faithful functor from the category of hyperrings to the category of bands.
\end{example}

\subsubsection{Algebras and free algebras} 
    
    Let \(k\) be a band. A \emph{\(k\)-algebra} is a band \(B\) together with a morphism of bands \(\alpha_B \colon k \to B\). Given two \(k\)-algebras \(\alpha_B \colon k \to B\) and \(\alpha_C \colon k \to C\), a \emph{morphism of \(k\)-algebras} is a morphism of bands \(\varphi \colon B \to C\) such that \(\varphi \circ \alpha_B = \alpha_C\). In this way, we obtain the category \(\mathbf{BAlg}_k\) of \(k\)-algebras.

    Let \(B\) be a band and let \(I\) be an index set. As in the case of monoids, a \textit{free \(B\)-algebra} in indeterminants \(\{x_i\; |\; i \in I \}\) is defined as follows: Its underlying pointed monoid is the free pointed monoid
    \begin{equation*}
        B[x_i]_{i \in I} \coloneq  \left\{ a \cdot \prod_{i \in I} x_i^{n_i}\; \middle|\; a\in B, (n_i)_{i\in I} \in \bigoplus_{i \in I} \mathbb{N}\right\}/\sim
    \end{equation*}
    where \(\sim\) is the equivalence relation generated by identifying \(0\cdot\prod_{i \in I} x_i^{n_i}\) for all \((n_i)_{i \in I} \in \bigoplus_{i \in I} \mathbb{N}\). To define its null set, consider the map \(\iota : B \to B[x_i]_{i \in I}\) defined by \(\iota(a) = a\cdot\prod x_i^0\), which extends naturally to a morphism of semirings \(\iota^+ : B^+ \to B[x_i]_{i \in I}^+\). Using this morphism, we define the null set of \(B[x_i]_{i \in I}\) as the null set generated by \(\iota^+(N_B)\).
    \begin{equation*}
        N_{B[x_i]_{i \in I}} \coloneq  \left<\iota^+(N_B) \right>_{B[x_i]_{i \in I}}.
    \end{equation*}
    Observe that, by definition, the map \(\iota\) is a morphism of bands. 
    
    Finally, we note that, by \cite[Proposition 1.8]{Lorscheid3}, a morphism of bands \(B[x_i]_{i \in I} \to C\) is completely determined by a morphism of bands \(B \to C\) and the image of each \(x_i\) in \(C\).

\subsubsection{Null ideals and quotients} To define a notion of quotient for bands, one requires that the ideal with respect to which one takes the quotient remains a null set in the quotient. This motivates the notion of a null ideal.

\begin{definition}
A \textit{null ideal} of a band \(B\) is an ideal \(I\) of \(B^{+}\) that contains the null set \(N_B\) and satisfies the following condition:
\[
\text{If } a + (- c) \in I \text{ and } c + \sum_j b_j \in I,\text{ then } a + \sum_j b_j \in I.
\]
\end{definition}

It follows from the definition that a null ideal \(I\) of \(B\) induces an equivalence relation \(\sim_I\) on \(B\) as following:
\begin{equation*}
    a \sim_I b \text{ if and only if } a + (-b) \in I.
\end{equation*}

Furthermore, we have that \(\sim_I\) is a congruence relation on \(B\), that is, if \(a \sim_I b\) and \(c \in B\), then \(c \cdot a \sim_I c \cdot b\). Consequently, the quotient \(B/\!\sim_I\) is a pointed monoid, and there is a natural morphism \(\pi_I : B \to B/\!\sim_I\) of pointed monoids, which extends to a morphism of semirings \(\pi_I^+ : B^+ \to (B/\!\sim_I)^+\).

\begin{definition}
    Let \(B\) be a band and let \(I \subseteq B^+\) be a null ideal. The \textit{quotient} of \(B\) by \(I\) is the band \(B\sslash I\) consisting of the pointed monoid \(B/\!\sim_I\) together with the null set 
    \begin{equation*}
        N_{B\sslash I} \coloneq  \pi_I^+(N_B).
    \end{equation*}
\end{definition}

With this definition, \(\pi_I\) becomes a morphism of bands.

\subsubsection{Base extension to \(\ZZ\)}
We recall how to construct a functor from the category of bands to the category of commutative rings.

\begin{definition}
Let \(B\) be a band with null ideal \(N_B\). The \textit{universal ring} of \(B\) is the quotient
\[
B_{\ZZ}^{+} \coloneq  \ZZ[B]/\langle N_B \rangle,
\]
where \(\langle N_B \rangle\) denotes the ideal generated by \(N_B\). It is equipped with a multiplicative map
\[
\rho_{\ZZ}^{+} \colon B \to B_{\ZZ}^{+}.
\]
\end{definition}

By \cite[Proposition~1.20]{Lorscheid3}, the category \(\mathbf{CRing}\) is a full subcategory of \(\Band\), and the functor
\(
(-)_{\ZZ}^{+} \colon \Band \to \mathbf{CRing}
\)
is left adjoint to the inclusion functor.

\subsubsection{Localization}

    Let \(B\) be a band and let \(S \subseteq B\) be a multiplicative set of \(B\), that is, a subset that is closed under multiplication and contains the identity element \(1\).  The \textit{localization} of \(B\) at \(S\), denoted by \(S^{-1}B\), is defined as follows: Its underlying pointed monoid is the quotient 
    \begin{equation*}
        S^{-1}B \coloneq  (S \times B)/\sim_S
    \end{equation*}
    where \(\sim_S\) is the equivalence relation defined by
    \[
    (s,a) \sim_S (s',a') \text{ if and only if there exists } t \in S \text{ such that } t s a' = t s' a.
    \] 
    We write \(a/s\) for the equivalence class of \((s,a)\).

    Let \(\iota_S : B \to S^{-1}B\) be the morphism of pointed monoids defined by \(\iota_S(a) = a/1\) for every \(a \in B\). It extends to a morphism of semirings \(\iota_S^+ : B^+ \to (S^{-1}B)^+\). Using this morphism, we define the null set of \(S^{-1}B\) as 
    \begin{equation*} 
        N_{S^{-1}B} \coloneq  \left<\iota_S^+(N_B) \right>_{S^{-1}B}.
    \end{equation*}
    By construction, \(\iota_S\) is a morphism of bands and it satisfies an universal property analogous to that of the localization of rings (see \cite[Proposition 1.22]{Lorscheid3}).

     \begin{definition}[Localization map]\label{def: localization map}
    A morphism of bands \(f: B \to C\) is a \textit{localization map} if there exists a multiplicative subset \(S \subseteq B\) and an isomorphism \(\varphi : S^{-1}B \overset{\sim}{\to} C\) such that \(f = \varphi \circ \iota_S\). That is, we have the following commutative diagram.
    \begin{equation*}   
    \begin{tikzcd}
        B \arrow[r, "f"] \arrow[d, swap, "\iota_S"] & C \\
        S^{-1}B \arrow[ru, swap, "\varphi"] & 
    \end{tikzcd}
    \end{equation*}
    If \(S\) is finitely generated, then we say that \(f\) is a \textit{finite localization map}.
\end{definition}

\subsubsection{Limits and colimits}

 The category of bands is known to be complete and cocomplete. Here, we recall the definition of some of the basic limits and colimits.

 \begin{definition}[Product]
     Let \((B_i)_{i \in I}\) be a family of bands. Their \textit{product}, denoted by \(\prod_{i \in I} B_i\), consists of the pointed monoid \(\prod_{i \in I} B_i\) together with the null set 
     \begin{equation*}
         N_{\prod_i B_i} \coloneq  \left\{\sum_{j} (a_{j,i})_{i \in I} \in \left(\prod_{i \in I}B_i \right)^+ \;\middle|\; \sum_j a_{j,i} \in N_{B_i} \text{ for every } i \in I \right\}.
     \end{equation*}
     Denote by \(\pr_k : \prod_j B_j \to B_k\) the projection maps for each coordinate \(k \in I\).
 \end{definition}

 \begin{definition}[Equalizer]
     Let \(f: B \to C\) and \(g: B \to C\) be morphisms of bands. Their \textit{equalizer}, denoted by \(\eq(f,g)\), consists of the pointed monoid \(\{ b \in B\; |\; f(b) = g(b)\}\) together with the null set 
     \begin{equation*}
         N_{\eq(f,g)} \coloneq  \{ b \in B\; |\; f(b) = g(b)\}^+ \cap N_B.
     \end{equation*}
 \end{definition}

 \begin{definition}[Tensor product]\label{tensor product algebras}
     Let \(k\) be a band and let \((B_i)_{i \in I}\) be a family of \(k\)-algebras with structure map \(\alpha_i : k \to B_i\) for each \(i \in I\). The \textit{tensor product of \((B_i)_{i \in I}\) over \(k\)}, denoted by \(\otimes_{k} B_i\), is defined as follows: Its underlying pointed monoid is the quotient
     \begin{equation*}
         \bigotimes_k B_i \coloneq  \left\{(a_i)_{i \in I} \in \prod_{i \in I}B_i \;\middle|\; a_i = 1 \text{ for all but finitely many } i \in I \right\}/\sim_k \;,
     \end{equation*}
     where \(\sim_k\) is the equivalence relation generated by the relations of the type: \((a_i)_{i \in I} \sim (b_i)_{i \in I}\) if there exist \(c \in k\) and \(i, j \in I\) such that 
     \begin{equation*}
         a_i = \alpha_i(c)\cdot b_i \text{ and } b_j = \alpha_j(c) \cdot a_j
     \end{equation*}
     and \(a_l = b_l\) for all \(l \in I\setminus\{i,j\}\). Denote by \(\otimes a_i\) the equivalence class of \((a_i)_{i\in I}\).

     For each \(l \in I\), let \(\iota_l : B_l \to \bigotimes_k B_i\) be the canonical map into the \(l\)-th coordinate. It extends naturally to a morphism of semirings \(\iota_l^+ : B_l^+ \to \left(\bigotimes_k B_i \right)^+\). Using these morphisms, we define the null set of \(\bigotimes_k B_i\) as
     \begin{equation*}
         N_{\bigotimes_k B_i} \coloneq  \left<\bigcup_{l \in I} \iota_l^+(N_{B_l})\right>_{\bigotimes_k B_i}.
     \end{equation*}
 \end{definition}

 By definition of \(\sim_k\), the morphism of bands \(\iota_l \circ \alpha_l : k \to \otimes_k B_i\) does not depends on \(l \in I\), making \(\otimes_k B_i\) into a \(k\)-algebra.

 By \cite[Proposition 1.42]{Lorscheid3}, we know that the tensor product \(\bigotimes_k B_i\) with the morphisms \(\iota_l\) is the coproduct in \(\textbf{BAlg}_k\) of the family of \(k\)-algebras \((B_i)_{i \in I}\). 

    Additionally, if \(f : B \to C\) and \(g : B \to C\) are morphisms of bands, then \(C \otimes_B C\), where the first \(C\) is viewed as a \(B\)-algebra by \(f\) and the second is viewed as a \(B\)-algebra by \(g\) is the \textit{coequalizer} of \(f\) and \(g\).

\subsection{Band schemes}

   In this subsection, we recall the definition of band schemes as band spaces locally isomorphic to affine band schemes. We refer to them as \textit{geometric} band schemes to distinguish them from the \textit{relative} band schemes defined in Section \ref{sec:Relative Schemes}.

\subsubsection{Prime ideals and the spectrum of a band}

    In the theory of bands, there are three notions of ideals. The first is that of a \textit{null ideal}, defined as an ideal of the semiring \(B^+\). In addition, there are the notions of \(m\)-ideals and \(k\)-ideals, which are ideals of the pointed monoid \(B\).

    \begin{definition}
        Let \(B\) be a band. A subset \(I \subseteq B\) is an \emph{\(m\)-ideal} if it satisfy the following: 
        \begin{enumerate}
            \item \(0 \in I\).
            \item \(B \cdot I \subseteq I\).
        \end{enumerate}
        A \emph{\(k\)-ideal} is an \(m\)-ideal satisfying the following additional condition:
        \begin{enumerate}
            \item[\((3)\)] If \(a + \sum_i b_i \in N_B\) and \(b_i \in I\), then \(a \in I\).
        \end{enumerate}
        Let \(S \subseteq B\). We denote by \(\left< S\right>_m\) (resp. \(\left< S\right>_k\)) the \(m\)-ideal (resp. \(k\)-ideal) of \(B\) generated by \(S\).
    \end{definition}

    As only the notion of \(m\)-ideals is required in the definition of band schemes, we restrict our exposition to \(m\)-ideals. The notion of \(k\)-ideals will be used in Section \ref{sec: hyperring schemes}, where we develop the relation between hyperrings and bands.

    We know that \(m\)-ideals have some similarities with ideals of rings. For example, they have the following properties (see Lemma 1.29 and Proposition 1.30 of \cite{Lorscheid3}):
    \begin{enumerate}
        \item Let \(f : B \to C\) be a morphism of bands. If \(I \subseteq C\) is an \(m\)-ideal, then \(f^{-1}(I)\) is also an \(m\)-ideal.
        \item Let \(B\) be a band, \(S \subseteq B\) a multiplicative subset and \(\iota_S : B \to S^{-1}B\) the localization map. Let \(I \subseteq B\) and \(J \subseteq S^{-1}B\) be \(m\)-ideals. Then, \(S^{-1}I\) is an \(m\)-ideal of \(S^{-1}B\) and \(S^{-1}(\iota_S^{-1}(J)) = J\). 
    \end{enumerate}

    We also have the notions of prime \(m\)-ideals, which leads to the notion of spectrum of a band.

    \begin{definition}
        Let \(B\) be a band.
        \begin{enumerate}
            \item An \(m\)-ideal \(\mathfrak{p}\) is prime if \(B \setminus \mathfrak{p}\) is a multiplicative subset of \(B\). 
            \item The (\textit{geometric}) \textit{prime spectruem} of \(B\), denoted by \(\Spec^{\mathrm{geo}} B\), is the set of all prime \(m\)-ideals of \(B\) together with the topology generated by the open subsets 
            \begin{equation*}
                U_h \coloneq  \{\mathfrak{p} \in \Spec^{\mathrm{geo}} B\; |\; h \not\in \mathfrak{p} \},\quad \text{for all } h \in B.
            \end{equation*}
        \end{enumerate}
    \end{definition}

    As in the case of rings, a morphism of bands \(f: B \to C\) induces a continuous map \(\Spec(f) : \Spec^{\mathrm{geo}} C \to \Spec^{\mathrm{geo}} B\). We also have a correspondence between primes ideals \(\mathfrak{p}\) of \(B\) such that \(\mathfrak{p} \cap S = \emptyset\) with prime ideals of \(S^{-1}B\), where \(S\) is a multiplicative subset of \(B\). 

    \begin{remark}\label{rmk: band are local}
        Every band \(B\) is ``local", that is, it has a unique maximal \(m\)-ideal \(\mathfrak{m} = B \setminus B^\times\), which is also a prime \(m\)-ideal. This leads, in particular, to the useful fact that for every covering of \(\Spec^{\mathrm{geo}} B\) by principal open subsets \(\{ U_{h_i} \}_{i \in I}\), there exists a \(j \in I\) such that \(U_{h_j} = \Spec^{\mathrm{geo}} B\) (see Proposition 2.1 of \cite{Lorscheid3}). 
    \end{remark}    

    \begin{definition}
        A morphism of bands \(f: B \to C\) is said to be \emph{local} if it satisfies \(f(\mathfrak{m}_B) \subseteq \mathfrak{m}_C\), where \(\mathfrak{m}_B\) and \(\mathfrak{m}_C\) are the maximal \(m\)-ideal of \(B\) and \(C\), respectively. 
    \end{definition}

    By Remark \ref{rmk: band are local}, we can define the structure sheaf of \(\Spec^{\mathrm{geo}} B\) as the sheaf of bands \(\OO_{\Spec^{\mathrm{geo}} B}\) satisfying 
    \begin{equation*}
        \OO_{\Spec^{\mathrm{geo}} B}(U_h) \coloneq  B[h^{-1}]
    \end{equation*}
    for every \(h \in B\), where \(B[h^{-1}]\) is the localization of \(B\) by multiplicative subset \(\{1, h, h^2, \cdots \}\).

\subsubsection{Band spaces and band schemes}

The analoge of a locally ringed space for the case of bands is a \textit{band space}, which is a pair \((X, \OO_X)\) of a topological space \(X\) and a sheaf of bands \(\OO_X\). The basic example of band space is the (\textit{geometric}) \textit{affine band scheme} \((\Spec B, \OO_{\Spec^{\mathrm{geo}} B})\). The stalk of \(\OO_X\) at \(x \in X\) is defined as the band 
\begin{equation*}
    \OO_{X,x} \coloneqq  \varinjlim_{x \in U \subseteq X \text{ open}} \OO_X(U).
\end{equation*}

 A morphism of band spaces \((f, f^\#) : (X, \OO_X) \to (Y, \OO_Y)\) is a pair of a continuous map \(f : X \to Y\) and a morphism of sheaves \(f^\# : \OO_Y \to f_* \OO_X\) such that for every \(x \in X\), the induced map \(f_x^\# : \OO_{Y,f(x)} \to \OO_{X,x}\) satisfies \(f_x^\#(\mathfrak{m}_{ \OO_{Y,f(x)}}) \subseteq \mathfrak{m}_{\OO_{X,x}}\), where \(\mathfrak{m}_B = B \setminus B^\times\) is the maximal \(m\)-ideal of the band \(B\). We denote by \textbf{BS} the category of band spaces.

\begin{proposition}\label{proposition 5}
    The category \(\textbf{BS}\) is cocomplete.
\end{proposition}

\begin{proof}
    We adapt the proof of \cite[I.1.1.6]{Demazure} and show that \(\textbf{BS}\) has arbitrary coproducts and coequalizers. 
    
\textbf{Coproducts:} Let \((X_i,\OO_{X_i})_{i \in I}\) be a family of band spaces. Their coproduct \((S,\OO_S)\) consists of the topological space \(S \coloneqq \coprod_{i \in I} X_i\) together with the sheaf of bands \(\OO_S\), defined as follows: for every open subset \(U = \coprod_{i \in I} U_i\), where \(U_i\) is an open subset of \(X_i\), we define \(\OO_S(U)\) as the product, in the category of bands, of the bands \(\OO_{X_i}(U_i)\). Moreover, for every \(x \in X_i \subseteq S\), we have an isomorphism
\begin{equation*}
    \OO_{S,x} \cong \OO_{X_i,x}.
\end{equation*}
This shows that \((S,\OO_S)\), together with the natural morphisms \(\iota_i: (X_i,\OO_{X_i}) \to (S,\OO_S)\), is indeed the coproduct in \(\textbf{BS}\) of the family \((X_i,\OO_{X_i})_{i \in I}\).
    
\textbf{Coequalizers:} Let \((f,f^\#)\) and \((g,g^\#)\) be two morphisms in \(\textbf{BS}\) from \((X,\OO_X)\) to \((Y,\OO_Y)\). Define the band space \((Z,\OO_Z)\) as follows:
\begin{itemize}
    \item \(Z\) is the coequalizer, in the category of topological spaces, of the continuous maps \(f\) and \(g\); that is, \(Z\) is the quotient space \(Y/\sim\), where \(\sim\) is the smallest equivalence relation such that \(f(x) \sim g(x)\) for all \(x \in X\).
    
    \item Let \(p \colon Y \to Z\) be the canonical projection, and let \(W\) be an open subset of \(Z\). Observe that if \(V = p^{-1}(W)\), then \(f^{-1}(V) = g^{-1}(V)\). Denote this common set by \(U\). We define \(\OO_Z\) by
\begin{equation*}
    \OO_Z(W) \coloneqq \left\{ s \in \OO_Y(V)\;\middle|\; f^\#_U(s) = g^\#_U(s) \in \OO_X(U) \right\}.
\end{equation*}
For open subsets \(W' \subseteq W\) of \(Z\), the restriction map \(\OO_Z(W) \to \OO_Z(W')\) is induced by the restriction maps of \(\OO_Y\), and the morphism of presheaves \(p^\# \colon \OO_Z \to p_*\OO_Y\) is defined, for every open subset \(W \subseteq Z\), as the inclusion \(\OO_Z(W) \hookrightarrow \OO_Y(p^{-1}(W))\). The sheaf axioms for \(\OO_Z\) follow from those for \(\OO_Y\).
    
\end{itemize}
    Finally, we prove that for every \(y \in Y\), the morphism
    \begin{align*}
        p^\#_{y}: \OO_{Z,p(y)} &\longrightarrow \OO_{Y,y}\\
        [(s,W)] &\longmapsto [(s,p^{-1}(W))],
    \end{align*}
     where \((s,W)\) is a pair of an open subset \(W \subseteq Z\) and a section \(s \in \OO_Z(W)\), is a local morphism of bands. Indeed, suppose that there exists \([(t, U)] \in \OO_{Y,y}\) such that 
     \begin{equation*}
         [(s,p^{-1}(W))]\cdot [(t, U)] = [(s|_{p^{-1}(W) \cap U} \cdot t|_{p^{-1}(W) \cap U}, p^{-1}(W) \cap U)] = [(1, p^{-1}(W) \cap U)].
     \end{equation*}
    Then, there exists an open subset \(W' \subseteq Z\) such that \(y \in p^{-1}(W') \subseteq p^{-1}(W) \cap U\) and 
    \begin{equation}\label{identityasdf}
         s|_{p^{-1}(W')} \cdot t|_{p^{-1}(W')} = 1.
    \end{equation}
    Since \(s \in \OO_Z(W)\), we have
    \[
        f^\#(s|_{p^{-1}(W')}) = g^\#(s|_{p^{-1}(W')}).
    \]
    Applying \(f^\#\) and \(g^\#\) to the identity \eqref{identityasdf} and using cancellation, we obtain
    \[
        f^\#(t|_{p^{-1}(W')}) = g^\#(t|_{p^{-1}(W')}).
    \]
    Hence \(t|_{p^{-1}(W')} \in \OO_Z(W')\), which shows that \([(t|_{p^{-1}(W')}, p^{-1}(W'))]\) defines an element of \(\OO_{Z,p(y)}\) which is the inverse of \( [(s,W)]\). Therefore, \(p^\#_y\) is a local morphism of bands.
    \end{proof}

\begin{definition}
    A (\textit{geometric}) \textit{band scheme} is a band space covered by open subspaces, each of which is isomorphic to a geometric affine band scheme. A morphism of band schemes is a morphism of band spaces.
\end{definition}

We write \(\BAff^{\mathrm{geo}}\) and \(\BSch^{\mathrm{geo}}\) for the categories of geometric affine band schemes and geometric band schemes, respectively. The category of geometric band schemes shares several similarities with the category of schemes and the category of monoid schemes.

\begin{proposition}[Theorem 2.9 of \cite{Lorscheid3}]
    Let \(\Gamma : \BSch^{\mathrm{geo}} \to \Band\) be the functor assigning for each geometric band scheme \(X\), its global section \(\Gamma(X) \coloneqq  \OO_X(X)\). Then, \(\Gamma\) is the right adjoint of the fully faithful functor \(\Spec : \Band \to \BSch^{\mathrm{geo}}\). In particular, \(\Spec : \Band \to \BAff^{\mathrm{geo}}\) is an antiequivalence of categories where \(\Gamma|_{\BAff^{\mathrm{geo}}}\) is its inverse.
\end{proposition}

\begin{proposition}[Theorem 2.16 of {\cite{Lorscheid3}}]
    The category \(\BSch^{\mathrm{geo}}\) is closed under finite limits and arbitrary coproducts. In particular, \(\Spec^{\mathrm{geo}}(\FF_1^{\pm})\) is terminal and \(\Spec^{\mathrm{geo}}(0)\) is initial, there exists the fiber product of \(X\) and \(Y\) over \(Z\), denoted by \(X \times_Z Y\), and the disjoint union \(\coprod X_i\) is the coproduct of a family of geometric band schemes \(X_i\).
\end{proposition}

\subsubsection{Open immersions}

 In this subsection, we collect some results on open subschemes and open immersions that will be useful in the subsequent sections.

\begin{definition}[Zariski topology]\label{Zariski topology geom}
    Let \(X\) be a geometric band scheme. 
    \begin{enumerate}
        \item An \textit{open subscheme} of \(X\) is an open subset \(U \subseteq X\) together with the restriction \(\OO_U \coloneqq \OO_X|_U\) of the structure sheaf of \(X\) to \(U\).
        \item An \textit{open immersion} is a morphism of schemes \(\iota: Y \to X\) that restricts to an isomorphism of \(Y\) with an open subscheme of \(X\). 
        \item A family \(\{\varphi_i: X_i \to X\}_{i \in I}\) of open immersions is a \emph{Zariski covering} if it is globally surjective on the underlying topological spaces, that is, \(\bigcup_{i \in I}\varphi_i(X_i) = X\).
        \item The \emph{Zariski topology} on \(X\) is the Grothendieck pretopology formed by the Zariski coverings.
    \end{enumerate}
\end{definition}

\begin{proposition}[Section 2.2.1 and 2.2.6 of \cite{Lorscheid3}]\label{lemma 34}
     A morphism of geometric band schemes \(f: X \to Y\) is an open immersion if and only if for every affine band scheme \(\Spec B\) over \(Y\), the induced morphism \(X \times_Y \Spec B \to \Spec B\) is an open immersion.
 \end{proposition}

\begin{proposition}\label{proposition 9}
    The Zariski topology of geometric band schemes is subcanonical. Also, the category of affine geometric band schemes is dense in the category of band schemes, in the sense that every geometric band scheme is a colimit of a diagram contained in the subcategory of affine geometric band schemes.
\end{proposition}

\begin{proof}
    The first part follows from \cite[Section 2.2.6]{Lorscheid3}. Let \(X\) be a band scheme, \(\{U_i \to X\}_{i \in I}\) an affine open covering and \(U_{i,j} \coloneqq U_i \cap U_j\). Then, the following isomorphism holds.
    \begin{equation*}
        \varinjlim\left(\coprod_{(i,j)\in I^2} U_{i,j} \rightrightarrows \coprod_{i\in I} U_i  \right) \overset{\sim}{\longrightarrow} X.
    \end{equation*}

\end{proof}

%% file: Modulesoverbands.tex
\section{Modules over Bands}\label{sec:Modules over Bands}

    This section is the core of this paper. We develop the theory of modules over bands, as introduced by M.~Jarra, O.~Lorscheid, and E.~Vital in \cite{Lorscheid4} in their study of matroid bundles, although their work did not include a detailed investigation of its properties.

    The main goal is to prove that the category of modules over a band is a closed symmetric monoidal category that is complete and cocomplete, and for that we need to verify the existence of several basic constructions, such as limits, colimits and tensor product.

\subsection{Definitions and examples}

\begin{definition}[Null set]
Let \((M,*)\) be a pointed set. A \textit{null set} on \((M,*)\) is a subset \(I\) of the free monoid
\[
M^+ \coloneqq \NN[M]/\langle \text{empty sum} \sim *\rangle
\]
satisfying the following properties:
\begin{enumerate}
    \item[(NS1)] \(* \in I\);
    \item[(NS2)] If \(x,y \in I\), then \(x+y \in I\);
    \item[(NS3)] For every \(m \in M\), there exists a unique \(n \in M\), denoted by \(-m\), such that \(m+n \in I\).
\end{enumerate}
\end{definition}

\begin{definition}[Band module]
Let \(B=(B,N_B)\) be a band. A (\textit{band}) \(B\)-\textit{module} is a pointed set \((M,*)\) equipped with a null set \(I_M\) and a \(B\)-action
\[
B \times M \longrightarrow M,\qquad (b,m)\longmapsto b\cdot m,
\]
satisfying the following properties, for all \(a, b, a_i \in B\) and \(m, m_j \in M\):
\begin{enumerate}
    \item[(BA1)] \(1\cdot m = m,\;\; 0\cdot m = *,\;\; a\cdot * = *\);
    \item[(BA2)] \((ab)\cdot m = a\cdot (b\cdot m)\);
    \item[(BA3)] If \(\sum_i a_i \in N_B\) or \(\sum_j m_j \in I_M\), then
    \(\sum_{i,j} a_i m_j \in I_M\).
\end{enumerate}
\end{definition}

\begin{remark}
Axiom~(BA3) is adapted from the corresponding axiom in the definition of modules over blueprints \cite{Lorscheid1}. Its purpose is to ensure that the induced map
\[
B^+ \times M^+ \longrightarrow M^+
\]
preserves, in an appropriate sense, the relations determined by the null sets. In the case of modules over rings, this axiom recovers the usual distributive laws:
\[
a\cdot (m_1+m_2)=a\cdot m_1+a\cdot m_2,\qquad
(a_1+a_2)\cdot m=a_1\cdot m+a_2\cdot m.
\]
\end{remark}

We now recall the definition of morphisms of \(B\)-modules.

\begin{definition}\label{def:bandmodule}
Let \(B=(B,N_B)\) be a band, and let \((M,I_M)\) and \((N,I_N)\) be \(B\)-modules. A morphism of \(B\)-modules is a map
\(f\colon M\to N\) satisfying the following conditions:
\begin{enumerate}
    \item \(f(a\cdot m)=a\cdot f(m)\) for all \(a\in B\) and \(m\in M\);
    \item If \(\sum m_i\in I_M\), then \(\sum f(m_i)\in I_N\).
\end{enumerate}
We denote by \(\Hom_B(M,N)\) the set of such morphisms.
\end{definition}

From now on, we denote by \(\BMod_B\) the category of band \(B\)-modules as defined above.

\subsection{Examples}

    We now present some straightforward examples of band modules that will be useful in the rest of the paper.

    \begin{example}[Band as module] Let \(A\), \(B\), and \(C\) be bands.
    \begin{enumerate}
    \item If \(f \colon A \to B\) is a morphism of bands, then \(B\) carries a natural structure of an \(A\)-module, with null set \(N_B\).

    \item If \(g \colon B \to C\) is a morphism of \(A\)-algebras, then \(g\) is also a morphism of \(A\)-modules.
    \end{enumerate}
    \end{example}

    \begin{example}[Zero module]
    The singleton \(0 = \{*\}\), viewed as a pointed set, admits a unique null set \(I = \{*\}\).  Equipped with the trivial \(B\)-action, it becomes a \(B\)-module, denoted by \(0\). This module is called the \emph{zero module} and is the zero object (\textit{i.e.}, both the initial and final object) of the category \(\BMod_B\).
    \end{example}

    Next, we define submodules and strict submodules of modules over bands.

    \begin{definition}[Submodules]\label{def:submodules}
        Let \(M\) be band \(B\)-module. A \(B\)-\emph{submodule} of \(M\) is an injective morphism of \(B\)-modules \(f: N \to M\). 
    \end{definition}

    \begin{definition}[Strict submodules]\label{def:strict submodules}
    Let \(M\) be a band \(B\)-module. A \emph{strict} \(B\)-sub-module of \(M\) is a \(B\)-submodule \(f:N \to M\) satisfying the following additional condition:
    \[
        \sum_i n_i \in I_N \iff \sum_i f(n_i) \in I_M.
    \]
    \end{definition}

    \begin{remark}\label{inclusion strict}
    Every strict submodule of a \(B\)-module \(M\) corresponds to a subset \(N \subseteq M\) satisfying the following conditions
    \begin{enumerate}
    \item \(*_M \in N\), and
    \item For every \(b \in B\) and \(n \in N\), we have \(b \cdot n \in N\),
    \end{enumerate}
    together with null set \(I_N \coloneqq  I_M \cap N^{+}\).
\end{remark}

     As an concrete example, we observe that hypermodules over a hyperring form a special case of band modules.

     \begin{example}[Hypermodules as band modules]
    Let \(H\) be a hyperring and let \(M\) be a hypermodule over \(H\). View \(H\) as a band with null set \(N_H\) as in Example~\ref{Ex: hyperring as band}. Define a null set on \(M\) by
    \[
    I_M \coloneqq  \left\{ \sum_i m_i \;\middle|\; 0 \in \op_i m_i \right\},
    \]
    where \(\boxplus\) denotes the hyperaddition on \(M\).
    Then, \((M,I_M)\) becomes a band \(H\)-module. Moreover, every morphism of hypermodules over \(H\) is a morphism of the corresponding band \(H\)-modules.
    \end{example}

\subsection{Quotients}

Let \(M\) be a \(B\)-module over a band \(B\). Recall that a \emph{categorical quotient of} \(M\) is an epimorphism \(\phi_N : M \to N\), considered up to equivalence, where two epimorphisms \(\phi_N : M \to N\) and \(\phi_{N'} : M \to N'\) are equivalent if there exists an isomorphism \(\iota : N \to N'\) such that \(\iota \circ \phi_N = \phi_{N'}\). As in the case of bands, there is a correspondence between the categorical quotients of a \(B\)-module \(M\) and a certain family of subsets of \(M^{+}\), which we call (\emph{module}) null ideals.

\begin{definition}[Module null ideal]
    Let \(B\) be a band, \(M\) a \(B\)-module, and \(I \subseteq M^+\) a subset. Then, \(I\) is a \textit{module null ideal} if it satisfies the following properties: 
    \begin{enumerate}
        \item[\((1)\)] \(I_M \subseteq I\);
        \item[\((2)\)] \(B\cdot I \subseteq I\);
        \item[\((3)\)] \(I + I \subseteq I\);
        \item[\((4)\)] If \(a- c \in I\) and \(a + \sum_i c_i \in I\), then \(c + \sum_i c_i \in I\).
    \end{enumerate}
\end{definition}

\begin{definition}[Quotient by null ideals]
    Let \(M\) be a \(B\)-module and \(I\) be a module null ideal. The quotient of \(M\) by \(I\) is the \(B\)-module \(M\sslash I\) consisting of the pointed set \(M/\sim_I\), where \(\sim_I\) is the equivalence relation defined as: \(x \sim_I y\) if and only if \(x - y \in I\). Additionally, the null set of \(M \sslash I\) is defined as 
    \begin{equation*}
        I_{M\sslash I} \coloneqq  \left\{ \sum_i [m_i] \in \left(M/\sim_I \right)^+\;\middle|\; \sum_i m_i \in I \right\}.
    \end{equation*}
    With this structure, the projection map \(\pi_I \colon M \to M\sslash I\) becomes a surjective morphism of band \(B\)-modules.
\end{definition}

As in the case of band, the following two propositions also holds for modules.

\begin{proposition}[Adapted from {\cite[Proposition 1.13]{Lorscheid3}}]
    Let \(f: M \to N\) be a morphism of \(B\)-modules and let \(J \subseteq M^+\) be a module null set of \(M\). If \(\sum_i f(m_i) \in I_N\) for all \(\sum_i m_i \in J\), then there exists a unique morphism of \(B\)-modules \(f_J : M\sslash J \to N\) such that \(f = \pi_J \circ f_J\).
\end{proposition}

\begin{proposition}[Adapted from {\cite[Proposition 1.14]{Lorscheid3}}]
    The association \(I \to M \sslash I\) stablishes a bijection:
    \begin{equation*}
        \Phi: \left\{\text{ module null ideals of } M\; \right\} \to \left\{ \text{ categorical quotients of } M\; \right\}.
    \end{equation*}
    whose inverse \(\Psi\) sends a quotient \([\pi: M \twoheadrightarrow N]\) to the following module null ideal of \(M\)
    \begin{equation*}
        \Psi([\pi: M \twoheadrightarrow N]) \coloneqq \left\{\sum_i m_i \in M^+\;\middle|\; \sum_i \pi(m_i) \in I_N  \right\}.
    \end{equation*}
\end{proposition}

    We presente two particular cases that will be important in the rest of this thesis. The first is the quotient by an equivalence relation and the second is the quotient by a stric sub module.

    \begin{definition}[Quotients by equivalence relations]\label{def: quotient}
    Let \(M\) be a band \(B\)-module, and let \(\sim\) be an equivalence relation on \(M\) satisfying the following compatibility condition:
    \[
    m \sim m' \;\Longrightarrow\; b \cdot m \sim b \cdot m'
    \quad \text{for all } b \in B.
    \]
    Let \(\pi_{\sim} \colon M \to M/\!\sim\) be the canonical projection map. It induces a morphism of monoids
    \(\pi_{\sim}^{+} \colon M^{+} \to (M/\!\sim)^{+}\).
    Using this morphism, we define the null set of \(M/\!\sim\) by
    \[
    I_{\sim} \coloneqq  \pi_{\sim}^{+}(I_M).
    \]
    The action given by \(b \cdot [m] = [b \cdot m]\) is well-defined and makes \(M/\!\sim\) into a band \(B\)-module. With this structure, the projection map \(\pi_{\sim} \colon M \to M/\!\sim\) becomes a surjective morphism of band \(B\)-modules.
    \end{definition}

    \begin{definition}[Quotient by a submodule]\label{def:quotient submod}
        Let \(M\) be a band \(B\)-module and \(N \subseteq M\) a sub-\(B\)-module. Consider the following equivalence relation \(\sim_N\) on \(M\):
        \begin{equation*}
            m \sim_N m' \text{ if and only if } m = m' \text{ or } m,m' \in N.
        \end{equation*}
        Since this relation preserves the action of \(B\), it defines a quotient module, which we denote by \(M/N\).
    \end{definition}

\subsection{Basic constructions}

     The main motivation of this section is to prove the following proposition.

    \begin{proposition}\label{thm:bicomplete}
        The category \(\BMod_B\) is complete, cocomplete, and for every pair of \(B\)-modules \(M\) and \(N\), the set \(\Hom_B(M,N)\) has a natural structure of \(B\)-module.
    \end{proposition}

    To prove this, it suffices to verify the existence of products, coproducts, equalizers, and coequalizers in \(\BMod_B\), and to show that the pointed set \(\Hom_B(M,N)\) admits a natural structure of a \(B\)-module. We follow the ideas of the corresponding constructions for \(A\)-sets in \cite{Lorscheid1}.

\subsubsection{The module \(\Hom_B(M,N)\)}
  
    Consider \((\Hom_B(M,N), 0)\) as a pointed set, where \(0\) is the zero morphism, and define its null set as 
    \begin{equation*}
        I_{\Hom_B(M,N)} \coloneqq  \left\{\sum_i f_i \in (\Hom_B(M,N))^+\;\middle|\; \sum f_i(m) \in I_N\text{ for all } m \in M \right\}.
    \end{equation*}
    Additionally, define the action of \(b \in B\) on \(f \in \Hom_B(M,N)\) by
    \[
        (b \cdot f)(m) \coloneqq  b \cdot f(m)
        \;\; \text{for all } m \in M.
    \]
    This defines a structure of module over \(B\) on \(\Hom_B(M,N)\). 

    \begin{remark}
        The unit object of \(\BMod_B\) is the \(B\)-module \(B\), \emph{i.e.}, \(\Hom_B(B,M)\) is isomorphic to \(M\) for every \(B\)-module \(M\).
    \end{remark}

    \begin{remark}
        Let \(M\) be a \(B\)-module. Then the \(B\)-module \(\Hom_B(M,M)\) with the composition as a multiplication is a (noncommutative) band, and the natural map \(B \rightarrow \Hom_B(M,M)\) is a morphism of (noncommutative) bands. This explains why several of the subsequent results are similar to the ones obtained for \(A\)-sets in \cite{Lorscheid2}.
    \end{remark}

\subsubsection{Product}

    We define the product of a family of \(B\)-modules \((M_i, *_i, I_i)_{i \in \Lambda}\) as the pointed set
    \(
        \left( \prod_{i \in \Lambda} M_i ,\, (*_i)_i \right),
    \)
    with null set defined by
    \begin{equation*}
        I_{\prod_{i \in \Lambda} M_i} \coloneqq  \left\{\sum_j (m_{i,j})_i \in \left( \prod_i M_i \right)^{+}\:\middle|\; \sum_j m_{i,j} \in I_i\text{ for all } i \in I  \right\}
    \end{equation*}
    
    The \(B\)-action is defined by
    \[
        b \cdot (m_i)_i \coloneqq  (b \cdot m_i)_i,
    \]
    which makes \(\prod M_i\) into a \(B\)-module.

    Finally, by construction, the projections \(p_k : \prod M_i \to M_k\) are morphisms of \(B\)-modules, and for every family of morphisms of \(B\)-modules \((f_k : N \to M_k)_{k \in \Lambda}\), define
    \[
        q : N \longrightarrow \prod M_i, \;\; q(n) = (f_k(n))_k.
    \]
    This map is the unique morphism of \(B\)-modules satisfying \(f_k = p_k \circ q\) for all \(k\).

\subsubsection{Coproduct}

    The coproduct of a family of \(B\)-modules \((M_i,*_i,I_i)_{i \in \Lambda}\) consists of the pointed set
    \(
        \left( \bigvee_{i \in \Lambda} M_i ,\, * \right),
    \)
    where \(\bigvee_{i \in \Lambda} M_i\) is the quotient of the disjoint union \(\bigsqcup_{i \in \Lambda} M_i\) by the equivalence relation generated by identifying all base points \(*_i\), and we denote by \(*\) the equivalence class of these base points.

    Let \(\iota_j : M_j \to \bigvee_i M_i\) be the inclusion map for every \(j \in \Lambda\). It extends to a morphism of monoids \(\iota_j^+ : (M_j)^+ \to (\bigvee_i M_i)^+\). Using this morphism, we define the null set of \(\bigvee_i M_i\) by 
    \begin{equation*}
        I_{\bigvee_i M_i} \coloneqq  \left<\bigcup_{j \in \lambda} \iota_j^+(I_{M_j}) \right>_{(\bigvee_i M_i)^+}.
    \end{equation*}
    The \(B\)-action on \(\bigvee M_i\) is defined by extending the \(B\)-actions on each \(M_i\), which makes \(\bigvee M_i\) into a \(B\)-module.

    By construction, each inclusion map
    \[
        \iota_j : M_j \longrightarrow \bigvee M_i
    \]
    is a morphism of \(B\)-modules.

    Finally, given any family of morphisms of \(B\)-modules \((f_i : M_i \to N)_{i \in \Lambda}\), there exists a unique morphism of \(B\)-modules
    \[
        \varphi : \bigvee M_i \longrightarrow N
    \]
    such that \(\varphi \circ \iota_i = f_i\) for all \(i\). Explicitly, define \(\varphi(m) = f_i(m)\) whenever \(m \in M_i\); this is well-defined since all base points are mapped to the same element of \(N\).

\subsubsection{Equalizer}

    Given two morphisms of \(B\)-modules \(f, g : M \to N\), their equalizer is the submodule
    \(
        (\operatorname{eq}(f,g), *_M),
    \)
    of \(M\), where
    \[
        \operatorname{eq}(f,g) = \{\, m \in M \mid f(m) = g(m) \,\}.
    \]
    Next, let \(\phi: P \rightarrow M\) be a morphism of \(B\)-modules such that \(f \circ \phi = g \circ \phi\). The image of \(\phi\) in \(M\) is contained in \(\eq(f,g)\), hence we can think of it as a morphism of \(B\)-modules \(\tilde{\phi}: P \rightarrow \eq(f,g)\).

\subsubsection{Coequalizer}

    Given two morphisms of \(B\)-modules \(f, g : M \to N\), their coequalizer \(\operatorname{coeq}(f,g)\) is the quotient of \(N\) by the equivalence relation generated by the following relation:
    \[
        n \sim n' \quad \text{if there exists } m \in M \text{ such that } n = f(m) \text{ and } n' = g(m).
    \]
    By construction, this equivalence relation preserves the action of \(B\), hence \(\coeq(f,g)\) is a \(B\)-module and the canonical projection \(\pi : N \to \operatorname{coeq}(f,g)\) is a morphism of \(B\)-modules.

    Next, suppose that \(h: N \to L\) is a morphism of \(B\)-modules such that \(h \circ f = h \circ g\). Then, for any \(n \sim n'\), we have \(h(n) = h(n')\). Hence, the map
    \[
        \varphi : \operatorname{coeq}(f,g) \to L, \qquad \varphi([n]) = h(n)
    \]
    is well defined and is the unique morphism of \(B\)-modules satisfying \(\varphi \circ \pi = h\).

\subsection{Monomorphisms and epimorphisms}
    
    We adopt the standard categorical definitions of monomorphisms and epimorphisms. In our setting, we show that these notions are equivalent to injectivity and surjectivity, respectively.

\begin{lemma}
    Let \(f: N \rightarrow M\) be a morphism of \(B\)-modules. Then, \(f\) is a monomorphism (resp. epimorphism) if and only if it is injective (resp. surjective).
\end{lemma}

\begin{proof}
Since a morphism of \(B\)-modules is completely determined by the images of the underlying elements, every injective (resp.\ surjective) morphism is necessarily a monomorphism (resp.\ epimorphism).

Conversely, suppose that \(f\) is a monomorphism and that \(f(n)=f(n')\) for some \(n,n' \in N\). Consider the maps
\(g,h \colon B \to N\) defined by \(g(b)=b\cdot n\) and \(h(b)=b\cdot n'\) for all \(b \in B\).
Both maps are compatible with the action of \(B\). Moreover, if \(\sum_i b_i \in N_B\), then by definition both
\(\sum_i b_i \cdot n\) and \(\sum_i b_i \cdot n'\) belong to \(I_N\).
Hence \(g\) and \(h\) are well-defined morphisms of \(B\)-modules satisfying
\(f \circ g = f \circ h\).
Since \(f\) is a monomorphism, we conclude that \(g=h\), and therefore \(n=n'\).
Thus \(f\) is injective.

Finally, suppose that \(f\) is an epimorphism, and let
\(\pi \colon M \to M/\mathrm{im}(f)\) be the canonical projection.
By construction, we have \(\pi \circ f = 0 \circ f\).
Since \(f\) is an epimorphism, this implies that \(\pi = 0\).
Consequently, \(M/\mathrm{im}(f)=0\), and hence \(f\) is surjective.
\end{proof}

\subsection{Kernel and cokernel}

    We give a characterization of the kernel and cokernel of morphisms of modules over a band.

\begin{remark}
    Since $\BMod_B$ has a zero object, the kernel of a morphism $f : M \to N$ of $B$-modules is given by the equalizer of $f$ and the zero morphism.
\end{remark}

\begin{lemma}\label{lemma:cokernel}
    Let $f : N \to M$ be a morphism of $B$-modules. Then the cokernel of $f$ is given by the canonical projection \(\pi : M \longrightarrow M / \im(f)\).
\end{lemma}

\begin{proof}
    Let $\sim$ be the equivalence relation defining $M / \im(f)$, and let $\phi : M \to P$ be a morphism of $B$-modules such that $\phi \circ f = 0$.

    If $m \sim m'$, then $\phi(m) = \phi(m')$, so the map
    \[
        \tilde{\phi} : M / \im(f) \longrightarrow P, 
        \qquad \tilde{\phi}([m]) = \phi(m),
    \]
    is well defined. By the definition of the null set of $M / \im(f)$, we see that $\tilde{\phi}$ is a morphism of $B$-modules. Moreover, $\tilde{\phi}$ is the unique morphism of $B$-modules satisfying
    \(\tilde{\phi} \circ \pi = \phi\).
\end{proof}

\subsection{Localization of modules}

Let \(B\) be a band, \(M\) a module over \(B\) and  \(S \subseteq B\) a multiplicative subset. The \textit{localization} of \(M\) by \(S\) is the \(S^{-1}B\)-module defined in the following way: Its underlying pointed set is the quotient
 \begin{equation*}
     S^{-1}M = (S \times M)/\sim_S\;,
 \end{equation*}
 where the equivalent relation \(\sim_S\) is defined as: \((s,m) \sim_S (s',m')\) if and only if there exists a \(t\in S\) such that \(tsm' = ts'm\). We denote the class of \((s,m)\) by \(m/s\). 
 
 The action of \(S^{-1}B\) on \(S^{-1}M\) is defined as
 \begin{align*}
     S^{-1}B \times S^{-1}M &\longrightarrow S^{-1}M\\
     \left(\frac{b}{s}, \frac{m}{t}\right) &\longmapsto \frac{bm}{st}.
 \end{align*}
 
  Let \(\iota_S : M \to S^{-1}M\) be the map defined by \(\iota_S(m) \coloneqq  m/1\) for every \(m \in M\). It extends to a morphism of monoids \(\iota_S^+ : M^+ \to (S^{-1}M)^+\). Using this morphism, we define the null set of \(S^{-1}M\) as 
 \begin{equation*}
     I_{S^{-1}M} \coloneqq  \left< \iota_S^+(I_M) \right>_{S^{-1}B}.
 \end{equation*}
 This makes \(S^{-1}M\) into a \(S^{-1}B\)-module, and the map \(\iota_S\) into a morphism of \(B\)-modules.

 As in the case of rings, we have the following relation between localization and tensor product.

 \begin{lemma}\label{localization tensor}
    Let \(B\) be a band, \(M\) a \(B\)-module and \(S \subseteq B\) a multiplicative subset. Then, the following map is an isomorphism of \(B\)-modules
    \begin{align*}
        \Phi: M \otimes_B S^{-1}B &\overset{\sim}{\longrightarrow} S^{-1}M\\
        m \otimes \frac{b}{s} &\longmapsto \frac{b\cdot m}{s}.
    \end{align*}
\end{lemma}

\begin{proof}
    By usual arguments, we see that \(\Phi\) is a well-defined bijective morphism of \(B\)-modules. It remaines for us to verify the correspondence between their null sets. By definition, \(I_{S^{-1}M}\) is generated by formal sums \(\sum_j m_j/1 \in (S^{-1}M)^+\) with \(\sum_j m_j \in I_M\). Since each \(m_j/1\) is the image of \(m_j \otimes 1\), we conclude, by the definition of null sets of tensor product, that \(\sum_j m_j/1\) is in the image of \(I_{M \otimes_B S^{-1}B }\) by \(\Phi^+\).
\end{proof}

\subsection{Free and projective modules}

    We collect some results on free and projective modules over bands. We show that the two usual definitions of projectivity are equivalent for modules over idylls, where we also prove that projectivity is equivalent to freeness. 

    \begin{definition}
Let \(B\) be a band. A \(B\)-module \(M\) is called \emph{free} if there exists a set \(I\) and an isomorphism of \(B\)-modules
\[
M \;\cong\; \bigvee_{i \in I} B,
\]
where the right-hand side denotes the coproduct of \(I\) copies of \(B\).
\end{definition}

\begin{remark}
Internally, a \(B\)-module \(M\) is free if and only if there exists a family of elements \((e_i)_{i \in I}\) of \(M\) such that:
\begin{enumerate}
    \item For every \(m \in M\), there exist unique elements \(i \in I\) and \(b \in B\) with \(m = b \cdot e_i\), and
    \item The null set of \(M\) is of the form
    \[
I_M = \left\langle \sum_j b_j e_i \;\middle|\; \sum_j b_j \in N_B,\; i \in I \right\rangle_{B^{+}}.
\]
\end{enumerate}
\end{remark}

We now turn to the notion of projectivity. 

\begin{definition}
    A \(B\)-module \(P\) is called \emph{projective} if for every epimorphism of band modules \(\pi: N \rightarrow M\) and every morphism of band modules \(f: P \rightarrow M\), there exists a morphism \(\varphi: P \rightarrow N\) such that \(\pi \circ \varphi = f\), that is, the following diagram is commutative.
    \begin{equation*}
        \begin{tikzcd}
            & P \arrow[d, "f"] \arrow[dl, dotted, swap, "\varphi"] \\
            M \arrow[r, two heads, "\pi"] & N.
        \end{tikzcd}
    \end{equation*}
\end{definition}

\begin{proposition}\label{direct summand = free}
    Let \(B\) be a band and \(P\) be a \(B\)-module. Then, the following are equivalent:
    \begin{enumerate}
        \item \(P\) is free.
        \item There exists a \(B\)-module \(Q\) such that \(P \vee Q\) is free.
    \end{enumerate}
\end{proposition}

\begin{proof}
    \((1) \implies (2)\) is trivial. Suppose \((2)\) holds and let \(Q\) be a \(B\)-module such that there exists an isomorphism 
    \begin{equation*}
        \varphi:  \bigvee_{i \in I} B\cdot e_i \overset{\sim}{\longrightarrow} P \vee Q.
    \end{equation*}
    Let \(J\) be the set of indexes \(j \in I\) such that \(\varphi(e_j) \in P\). We claim that the map 
    \begin{equation*}
        \varphi_J: \bigvee_{j \in J} B\cdot e_j \longrightarrow P,\quad \varphi_J(e_j) \coloneqq  \varphi(e_j) \text{ for every } j \in J
    \end{equation*}
    is a well defined isomorphism of \(B\)-modules. Indeed, for every \(j \in J\) and \(b \in B\), we have \(\varphi(b\cdot e_j) = b\cdot \varphi(e_j) \in P\), hence \(\varphi_J\) is well define. Now, for every \(m \in P\), there exists an \(i \in I\) and an \(b \in N\) such that \(\varphi(b \cdot e_i) = m\). If, \(i \not\in J\), then we would have \(\varphi(b \cdot e_i) = b \cdot \varphi(e_i) \in Q\), which is a contradiction. So \(\varphi_J\) is surjective. The injectivity of \(\varphi_J\) follows from that of \(\varphi\). 

    Finally, we just need to verify the correspondence of the null sets. By the isomorphism \(\varphi\), and the definition of null set of \(P \vee Q\), we see that the null set of \(P\) is 
    \begin{equation*}
        I_P = \left\{\sum_{l} b_l \cdot \varphi(e_j)\;\middle|\; \sum_{l} b_l \in N_B,\; j \in J \right\}.
    \end{equation*}
    From this, we conclude that \(\varphi_I\) is an isomorphism.
\end{proof}

\begin{remark}
    For a module \(M\) over a commutative ring \(R\), \(M\) is projective if and only if it is a direct summand of a free \(R\)-module. This does not hold in the category of modules over bands. For instance, Let \(F\) be an idyll considered as an \(F \times F\)-module (observe that \(F \times F\) is not an idyll) with the folowwing action:
    \begin{equation*}
        (F \times F) \times F \longrightarrow F,\quad ((a,b),c) \longmapsto a \cdot c.
    \end{equation*}
    Then, \(F\) is a projective \(F \times F\)-module but, by comparing cardinalities, it is not a free \(F \times F\)-module, hence, by proposition \ref{direct summand = free}, it is not a direct summand of a free \(F \times F\)-module.
\end{remark}

For modules over idylls, we can prove the equivalence between the two usual definitions of projectivity.

\begin{proposition}\label{prop: idyll proj = free}
    Let \(F\) be an idyll and \(P\) be an \(F\)-module. Then, the following are equivalent:
    \begin{enumerate}
        \item \(P\) is projective.
        \item There exists an \(F\)-module \(Q\) such that \(P \vee Q\) is free.
        \item \(P\) is free.
    \end{enumerate} 
\end{proposition}

\begin{proof}
    The equivalence \((2) \iff (3)\) is the Proposition \ref{direct summand = free}, and the proof of \((2) \implies (1)\) is the same as the case of rings. We prove \((1) \implies (3)\). Suppose that \(P\) satisfies \((1)\) and consider the following epimorphism of \(F\)-modules:
    \begin{equation*}
        \pi: N\coloneqq  \bigvee_{p \in P}F\cdot e_p \longrightarrow P,\quad \pi(e_p) \coloneqq  p\text{ for every }p \in P.
    \end{equation*}
    Then, by hypothesis, there exists a morphism \(\varphi : P \rightarrow N\) such that \(\pi \circ \varphi = \id_P\). Consider the following subset of \(P\):
    \begin{align*}
        \Tilde{P} &\coloneqq  \left\{p \in P\;\middle|\; \text{There exists a } q \in P \text{ and } b \in F \text{ such that } \varphi(q) = b\cdot e_p  \right\}\\
        & \;= \left\{p \in P\;\middle|\; \text{There exists a } q \in P \text{ such that } \varphi(q) = e_p \right\}\\
        & \;= \left\{p \in P\;\middle|\; \varphi(p) = e_p \right\},
    \end{align*}
    where the second equality follows from the properties of the idyll \(F\), and the last equality follows from \(\pi \circ \varphi = \id_P\). We claim that 
    \[
    \tilde{\pi}: \tilde{N} \coloneqq  \bigvee_{p \in \Tilde{P}} F\cdot e_p \longrightarrow P,\quad \tilde{\pi}(e_p) \coloneqq  p \text{ for every } p \in \tilde{P},
    \]
    is an isomorphism. Since \(\tilde{\pi}\) is the composition of the inclusion \(\bigvee_{p \in \Tilde{P}} F\cdot e_p \hookrightarrow \bigvee_{p \in P}F\cdot e_p\) with \(\pi\), it is a morphism of \(F\)-modules. 
    
    Now, we prove the surjectivity. For every \(m \in P\), there exists a \(b \in F\) and \(q \in P\) such that \(\varphi(m) = b \cdot e_q\). By definition, \(q \in \tilde{P}\) and \(\tilde{\pi}(b\cdot e_q) = \pi(b\cdot e_q) = (\pi \circ \varphi)(m) = m\), hence \(\tilde{\pi}\) is surjective.

    Next, we prove the injectivity. Suppose \(\tilde{\pi}(b\cdot e_p) = \tilde{\pi}(c\cdot e_q)\) for some \(b,c \in F\) and \(p,q \in \tilde{P}\). Using that \(\varphi(e_p) = e_p\), \(\varphi(e_q) = e_q\) and \(\pi\circ \varphi = \id_P\), we conclude that \(b\cdot e_p = c \cdot e_q\), hence \(\tilde{\pi}\) is injective.
    
    From surjectiviry of \(\tilde{\pi}\) and the projectivity of \(P\), there exists a morphism \(\phi : P \rightarrow \tilde{N}\) such that \(\tilde{\pi} \circ \phi = \id_P\). This implies that \(\tilde{\pi}\) is an isomorphism.
\end{proof}

Using a similar idea from the second half of the proof of Proposition \ref{prop: idyll proj = free}, we get the following useful corollary:

\begin{corollary}\label{cor: idyll dim}
    Let \(F\) be an idyll, Then if there exists an isomorphism of \(F\)-modules 
    \begin{equation*}
        \bigvee_{i \in I} F\cdot e_i \cong \bigvee_{j \in J}F\cdot e_j,
    \end{equation*}
    then \(\# I = \# J\).
\end{corollary}

The following result will be useful in Section~\ref{sec: protoexact}.

\begin{lemma}
    The category \(\BMod_F^{\mathrm{fin, proj}}\) of finitely generated projective modules over an idyll \(F\) is closed by coproduct, kernel, and cokernel.
\end{lemma}

\begin{proof}
    Let \(M\) be a finitely generated projective module over \(F\). By Proposition \ref{prop: idyll proj = free} and Corollary \ref{cor: idyll dim}, \(M\) is isomorphic to a free module of the form 
    \begin{equation*}
        \phi: M \overset{\sim}{\longrightarrow} \bigvee_{i = 1}^n F\cdot e_i\;,
    \end{equation*}
    with a uniquely determined \(n\). Hence, closedness by coproduct follows directly. Now, suppose \(N\) is a strict submodule of \(M\), and consider the set \(I'\) defined as
    \begin{equation*}
        I' \coloneqq \left\{i \in I\;\middle|\; e_i \in N \right\}.
    \end{equation*}
    Then, since \(F\) is an idyll, the restriction of \(\phi\) to \(N\) gives an isomorphism from \(N\) to the coproduct \(\bigvee_{i \in I'}F \cdot e_{i'}\). Moreover, the quotient \(M/N\) is isomorphic to the coproduct 
    \begin{equation*}
        M/N \overset{\sim}{\longrightarrow} \bigvee_{i \in I\setminus I'} F\cdot e_i.
    \end{equation*}
    Finally, since, by definition, kernel is a strict submodule and, by Lemma \ref{lemma:cokernel}, cokernel is a quotient by a strict submodule, we conclude that \(\BMod_F^{\mathrm{fin, proj}}\) is also closed by kernel and cokernel.
\end{proof}

\subsection{Base extension to \(\ZZ\)}

\begin{definition}
    Let \(M\) be a \(B\)-module. Consider the abelian group 
    \begin{equation*}
        M_{\ZZ}^+ = B_{\ZZ}^+[M]/\left< I_M \right>,
    \end{equation*}
    where \(\left< I_M \right>\) is the subgroup generated by \(I_M\). Then, \(M_{\ZZ}^+\) is a \(B_{\ZZ}^+\)-module with the action defined as 
    \begin{align*}
        B_{\ZZ}^+ \times M_{\ZZ}^+ &\longrightarrow M_{\ZZ}^+\\
         (\sum_i b_i + \left<N_B \right>, \sum_j m_j + \left<I_M \right>) &\longmapsto \sum_{i,j} b_im_j + \left<I_M \right>.
    \end{align*}
    Let \(f : M \rightarrow N\) be a morphism of \(B\)-modules. Then, we define the corresponding morphism of \(B_{\ZZ}^+\)-modules as
    \begin{align*}
        f_{\ZZ}^+ : M_{\ZZ}^+ &\longrightarrow N_{\ZZ}^+\\
        \sum_j m_j + \left<I_M \right> &\longmapsto \sum_j f(m_j) + \left<I_M \right>.
    \end{align*}
    This defines a base extension functor \((-)_{\ZZ}^+\) from the category \(\BMod_B\) to the category \(\Mod_{B_{\ZZ}^+}\). 
\end{definition}

\begin{remark}
    \begin{enumerate}
        \item Observe that, since we have \(\left<N_B \right>\cdot \left<I_M \right> \subseteq \left<I_M \right>\) by the item (3) of the definition of \(B\)-modules, the definition of the \(B_{\ZZ}^+\)-action is well-defined. Analogously, the well-definitness of the morphism \(f_{\ZZ}^+\) follows from the definition of morphisms of \(B\)-modules.
        \item Attached to this extension, there is a morphism of \(B\)-modules \(\rho_{M,\ZZ}^+ : M\rightarrow M_{\ZZ}^+\) where we consider \(M_{\ZZ}^+\) as a \(B\)-module via the multiplicative map \(\rho_{B,\ZZ}^+ : B \rightarrow B_{\ZZ}^+\).
        \item \(M_{\ZZ}^+\) is a module over the ring \(B_{\ZZ}^+\) in the usual sense. It is important to note that the base extension \(M \otimes_B B_{\ZZ}^+\), where the tensor product is defined in Section \ref{subsec:tensor product}, is a \emph{band module} over \(B_{\ZZ}^+\).
    \end{enumerate}
\end{remark}

\begin{proposition}
    Let \(M\) be an object of \(\BMod_B\), \(P\) an object of \(\Mod_{B_{\ZZ}^+}\) and \(\varphi : M \rightarrow P\) be a morphism of \(B\)-modules, where \(P\) is being considered as an object of \(\BMod_B\). Then there exists a unique morphism \(\tilde{\varphi} \in \Hom_{\Mod_{B_{\ZZ}^+}}(M_{\ZZ}^+, P)\) such that \(\Tilde{\varphi} \circ \rho_{M,\ZZ}^+ = \varphi\).
    \[
    \begin{tikzcd}
        M \arrow[r, "\rho_{M,\ZZ}^+"] \arrow[d, "\varphi"] & M_{\ZZ}^+ \arrow[ld, dotted, "\Tilde{\varphi}"] \\
        P & 
    \end{tikzcd}
    \]
\end{proposition}

\begin{proof}
    Define the map \(\Tilde{\varphi}\) as 
    \begin{equation*}
        \tilde{\varphi} \left( \sum m_i + \left<I_M \right>\right) \coloneqq  \sum f(m_i),
    \end{equation*}
    for every \(\sum m_i + \left<I_M \right> \in M_{\ZZ}^+\). Observe that, since we are considering \(P\) as an object of \(\BMod_B\), its null set  is  
    \begin{equation*}
        I_P = \left\{ \sum p_i \in P^+ : \sum p_i = 0\text{ in } P \right\}
    \end{equation*}
    Hence, by the definition of morphisms of \(B\)-modules, for every \(\sum m_i \in I_M\), we have \(\sum f(m_i) = 0\) in \(N\), proving the well-definess of \(\tilde{\varphi}\). 

    From the definition of \(\tilde{\varphi}\), it follows that it preserves both the sum and the action of \(B_{\ZZ}^+\).
\end{proof}

\subsection{Tensor product}\label{subsec:tensor product}

     We already have the notion of tensor product of \(B\)-algebras. Here, we generalize this concept to the case of \(B\)-modules. Our approach adapts the definition of tensor product of \(A\)-sets given in \cite{Lorscheid2}.

     \begin{definition}[Tensor product of modules]
    Let \(B\) be a band and let \(M\) and \(N\) be two \(B\)-modules. We define the tensor product \(M \otimes_B N\) as the \(B\)-module whose underlying pointed set is the quotient of the coproduct \(\bigvee_{M \times N} B\) by the equivalence relation \(\sim\) defined as follows: Denote the element \(b\) of \(\bigvee_{M \times N} B\) in the component \(B\) at \((m,n) \in M \times N\) by \(b_{(m,n)}\). Then, the equivalence relation \(\sim\) is generated by
    \[
        b_{(c \cdot m, n)} \sim b_{(m, c \cdot n)} \quad \text{and} \quad b_{(m, c \cdot n)} \sim (c b)_{(m,n)}
    \]
    for all \(b, c \in B\) and \((m,n) \in M \times N\). Observe that \(\sim\) is compatible with the action of \(B\), i.e.,
    \[
        b_{(m,n)} \sim b'_{(m',n')} \implies c \cdot b_{(m,n)} \sim c \cdot b'_{(m',n')} \quad \text{for all } c \in B.
    \]
    so the quotient \(\bigvee_{M \times N} B / \sim\) is, a priori, already a \(B\)-module by Definition~\ref{def: quotient}. Although, for the tensor product, we equip this quotient with a different null set. Let \(I\) be the null set in \(\left(\bigvee_{M \times N} B/\sim\right)^+\) generated by
    \[
        \left\{ \sum_i 1_{m_i,  n} \;\middle|\; \sum_i m_i \in I_M \right\} \cup 
        \left\{ \sum_i 1_{m, n_i} \;\middle|\; \sum_i n_i \in I_N \right\}.
    \]
    The projection map is defined by
    \[
        \otimes : M \times N \longrightarrow M \otimes_B N, \quad (m,n) \mapsto 1_{(m,n)},
    \]
    and we denote the image \(1_{(m,n)}\) simply by \(m \otimes n\). Observe that, since \(b_{(m,n)} \sim 1_{(b \cdot m, n)}\), every element of \(M \otimes_B N\) can be expressed as a pure tensor.
\end{definition}

    \begin{theorem}[Universal property of tensor product]\label{tensor prod universal property}
    The map 
    \[
        \otimes : M \times N \longrightarrow M \otimes_B N
    \]
    is a bilinear morphism of \(B\)-modules, i.e., it is a morphism of \(B\)-modules in each variable. Moreover, it satisfies the following universal property: for every bilinear morphism of \(B\)-modules \(\varphi : M \times N \to P\), there exists a unique morphism of \(B\)-modules \(\tilde{\varphi} : M \otimes_B N \to P\) such that 
    \(\tilde{\varphi} \circ \otimes = \varphi\).
\end{theorem}

    \begin{proof}
    Bilinearity of \(\otimes_B\) follows directly from the defining relations of \(M \otimes_B N\):
    \[
        (b \cdot m) \otimes n = b \cdot (m \otimes n), \qquad
        m \otimes (b \cdot n) = b \cdot (m \otimes n),
    \]
    and from the definition of its null set.

    Let \(\varphi : M \times N \to P\) be a bilinear morphism of \(B\)-modules. The map
    \[
        \tilde{\varphi} : M \otimes_B N \longrightarrow P, \qquad \tilde{\varphi}(m \otimes n) \coloneqq  \varphi(m,n),
    \]
    is well-defined by the defining relations of the tensor product, the fact that every element of \(M\otimes_BN\) is a pure tensor, and the bilinearity of \(\varphi\). Let \(\sum_i m_i \otimes n \in I_{M\otimes_B N}\) such that \(\sum_i m_i \in I_M\). Then, by the bilinearity of \(\varphi\), we have \(\sum_i \varphi(m_i,n) \in I_P\). Analogously, if \(\sum_i m \otimes n_i \in I_{M\otimes_B N}\) such that \(\sum_i n_i \in I_N\), we also have \(\sum_i \varphi(m,n_i) \in I_P\). This proves that \(\varphi\) is a morphism of \(B\)-modules and is uniquely determined by the condition \(\tilde{\varphi} \circ \otimes = \varphi\).
\end{proof}

As in the case of modules over commutative rings, we get the following corollary.

\begin{corollary}\label{tensor product closedness}
    Let \(B\) be a band and \(M\) be a \(B\)-module. Then, the functor \(\Hom_B(M,-): \BMod_B \to \BMod_B\) is the right adjoint of the functor \(-\otimes_B M\).
\end{corollary}

\begin{lemma}\label{tensor product lemmas}
    Let \(B\) be a band and let \(M,N\) and \(P\) be \(B\)-modules. Then, we have the following isomorphisms.
    \begin{enumerate}
        \item[\((1)\)] \(\alpha: (M \otimes_B N) \otimes_B P \cong M \otimes_B (N \otimes_B P)\).
        \item[\((2)\)] \(\gamma: M \otimes_B N \cong_B N \otimes_B M\).
        \item[\((3)\)] \(\mu: B \otimes_B M \cong_B M \).
    \end{enumerate}
    Moreover, the isomorphisms are functorial in \(M\), \(N\) and \(P\).
\end{lemma}

\begin{proof}
    Since we have a universal property (Theorem \ref{tensor prod universal property}), we can use the same proof as for modules over commutative rings.
\end{proof}

    In the next lemma, we relate our definition of tensor product of \(B\)-modules to Definition \ref{tensor product algebras}. 

\begin{lemma}
    Let \(B\) be a band and let \(C\) and \(D\) be two \(B\)-algebras. Then, their tensor product as \(B\)-modules, denoted by \((C \otimes_B D)_{\mathrm{mod}}\), has a structure of \(B\)-algebra isomorphic to their tensor product as \(B\)-algebras, denoted by \((C \otimes_B D)_{\mathrm{alg}}\).
\end{lemma}

\begin{proof}
    First, we prove that \((C \otimes_B D)_{\mathrm{mod}}\) has a structure of \(B\)-algebra. For this, cosider the map 
    \begin{align*}
        \Phi: (C \otimes_B D)_{\mathrm{mod}} \times (C \otimes_B D)_{\mathrm{mod}} &\longrightarrow (C \otimes_B D)_{\mathrm{mod}}\\
        (c \otimes d, c' \otimes d')\;\;\;\;\;\;\;\;\;\; &\longmapsto (cc' \otimes dd').
    \end{align*}
    This map is a well-defined bilinear morphism of \(B\)-modules. Indeed, it clearly preserves the action of \(B\) and for \(\sum c_i \otimes d \in I_{(C \otimes_B D)_{\mathrm{mod}}}\) with \(\sum_i c_i \in I_C\), we have \(\sum c_ic' \in I_C\), hence \(\sum c_ic' \otimes dd' \in I_{(C \otimes_B D)_{\mathrm{mod}}}\). Since \(I_{(C \otimes_B D)_{\mathrm{mod}}}\) is generated by the sums \(\sum c_i \otimes d\) and \(\sum c\otimes d_i\) with \(\sum c_i \in I_C\) and \(\sum d_i \in I_D\), we conclude that \(\Phi\) is a bilinar morphism.

    This map makes the underlying set of \((C \otimes_B D)_{\mathrm{mod}}\) into a pointed monoid, with identity \(1 \otimes 1\) and zero \(0 \otimes 0\). Now, we verify that \(I_{(C \otimes_B D)_{\mathrm{mod}}}\) is a null set for this pointed monoid. Since we already have some properties from the null set as modules, we only need to check that it is closed by multiplication of elements of \((C \otimes_B D)_{\mathrm{mod}}\), which follows from the bilinarity proved above. Hence \((C \otimes_B D)_{\mathrm{mod}}\) is a band, and when we consider the morphism \(B \rightarrow (C \otimes_B D)_{\mathrm{mod}}\) sending \(b \in B\) to \(b \otimes 1\), we see that it is a \(B\)-algebra. 

    Finally, consider the map 
    \begin{align*}
        \Psi: (C \otimes_B D)_{\mathrm{mod}} &\longrightarrow (C \otimes_B D)_{\mathrm{alg}}\\
         c \otimes b &\longmapsto c \otimes b.
    \end{align*}
    Then, by the definitions of the underlying pointed monoids of both sides and their null sets, we conclude that \(\Psi\) is an isomorphism of \(B\)-algebras.
\end{proof}

Finally, we prove the first main result of this chapter.

\begin{theorem}\label{thm:symmetricmonoidalcategory}
    The category \(\BMod_B\) is a closed symmetric monoidal category that is complete and cocomplete.
\end{theorem}

\begin{proof}
    Completeness and cocompleteness follows from Proposition \ref{thm:bicomplete}. The tensor product \(\otimes_B\) makes \(\BMod_B\) into a closed symmetric monoidal category by Lemma \ref{tensor product lemmas} and Corollary \ref{tensor product closedness}. 
\end{proof}

\subsection{Commutative monoids}

    In this last subsection, we prove that a commutative monoid in the symmetric monoidal category \(\BMod_B\) is exactly a \(B\)-algebra. This proves that the category \(\BMod_B\) is indeed the right choice as the category of modules over \(B\).
 
    \begin{theorem}\label{thm:eqofcategories}
        The category \(\BAlg_B\) of algebras over \(B\) coincides with the category of commutative monoids in \(\BMod_B\).
    \end{theorem}

    \begin{proof}
    We produce two functors which are inverse to each other. First, let \(A\) be a $B$-algebra, and regard \(A\) as an \(B\)-module. The multiplication
\[
\mu_A : A\times A \longrightarrow A,\qquad (a,a')\mapsto a\cdot_A a',
\]
is a bilinear morphism of \(B\)-modules. By Theorem \ref{tensor prod universal property}, there is a unique morphism of $B$-modules
\[
\widetilde{\mu}_A : A\otimes_B A \longrightarrow A
\]
such that $\widetilde{\mu}_A(a\otimes a')=\mu_A(a,a')$ for all $a,a'\in A$. The identity $1_A$ defines a morphism of $B$-modules
\[
\eta_A : B \longrightarrow A,\qquad b\mapsto b\cdot_A 1_A. 
\]
This makes \((A,\tilde{\mu}_A,\eta_A)\) into a commutative monoid in the category \(\BMod_B\).

Conversely, let  $(M,\mu,\eta)$ be a commutative monoid in $\BMod_B$, i.e. $M$ is a $B$-module, $\mu:M\otimes_B M\to M$ and $\eta:B\to M$ are morphisms in $\BMod_B$. Define a multiplication on the underlying \(B\)-module $M$ by
\[
m\cdot_M m' \coloneqq  \mu(m\otimes m')\qquad(m,m'\in M),
\]
and define the unit element $1_M \coloneqq  \eta(1_B)$. This makes \(M\) into a \(B\)-algebra.

The two functors are inverse to each other, proving the isomorphism of categories.

\end{proof}

%% file: RelativeSchemes.tex
\section{Relative Band Schemes}\label{sec:Relative Schemes}

This section is devoted to the proof of Theorem~\ref{main theorem 2}. We begin by recalling the necessary definitions from the theory of schemes relative to a complete and cocomplete closed symmetric monoidal category. Our approach follows that of A.~Vezzani in \cite{Vezzani} for monoid schemes, and O.~Lorscheid in \cite{Lorscheid1} for semiring schemes and blue schemes.

\subsection{Review of relative algebraic geometry}

     Our presentation follows the original work by B. T\"oen et M. Vaqui\'e in \cite{ToenVaquie}. We also recomend \cite{Timothy}. 

    Let \(\mathscr{C} = (\mathscr{C}, \otimes, e)\) be a symmetric monoidal category, where \(\otimes: \mathscr{C} \times \mathscr{C} \to \mathscr{C}\) is the bifunctor and \(e\) is the unit object. We denote by \(\Comm(\mathscr{C})\) the category of commutative monoids in \(\mathscr{C}\). For every commutative monoid \(A \in \Comm(\mathscr{C})\), we denote by \(A\)-\(\Mod\) the category of pairs \((M,\sigma)\) consisting of an object \(M\) of \(\mathscr{C}\) and a morphism \(\sigma : A \otimes M \to M\) making the usual diagrams commutative (see \cite[Definition 1.3.2.7]{Timothy}). For each commutative monoid \(A\) of \(\mathscr{C}\), denote by \(\Alg_A\) the coslice category \(A/\Comm(\mathscr{C})\) and call its objects \(A\)-algebras. For simplicity, we fix the following nomenclature. 

     \begin{definition}[Cosmos]
    A \emph{cosmos} is a complete and cocomplete closed symmetric monoidal category.
    \end{definition}

    From now on, let \(\mathscr{C}\) be a fixed cosmos. Due to its cocompleteness, we can define the tensor product \(M \otimes_A N\) of two objects \(M\) and \(N\) of \(A\)-\(\Mod\) for a commutative monoid \(A \in \Comm(\mathscr{C})\).

    \begin{definition}[Tensor product in \(A\)-\(\Mod\), {\cite[Definition 11]{Vezzani}}]
        Let \(A\) be a commutative monoid in \(\mathscr{C}\), and let \((M,\sigma_M)\) and \((N,\sigma_N)\) be two objects of \(A\)-\(\Mod\). The tensor product of \(M\) and \(N\) over \(A\), denoted by \(M \otimes_A N\), is the coequalizer of the diagram 
        \begin{equation*}
            \begin{tikzcd}
                A \otimes M \otimes N \arrow[r, shift left, "\sigma_N \otimes N"] \arrow[r, shift right, swap, "\sigma_M \otimes M"] & M \otimes N
            \end{tikzcd}
        \end{equation*}
        It has a natural \(A\)-module structre.
    \end{definition}

    \begin{remark}
        The triple \((A \text{-}\Mod, \otimes_A, A)\) is itself a cosmos.
    \end{remark}

    \begin{example}
        The principal example of a cosmos is the category \(\Mod_{\ZZ}\) of modules over \(\ZZ\) with the usual tensor product \(\otimes_{\ZZ}\) and unit object \(\ZZ\). Here, the category \(\Comm(\Mod_{\ZZ})\) of commutative monoids is the category of commutative rings and for each commutative ring \(R\), the category \(R\)-\(\Mod\) is exactly the category \(\Mod_R\) of modules over \(R\).
    \end{example}

    \begin{example}
        From the results of Section \ref{sec:Modules over Bands}, the category \(\BMod_{\FF_1^\pm}\) is a cosmos and the category of commutative monoids in it is exactly the category of bands. If \(B\) be a band, then the category \(B\)-\(\Mod\) is exactly the category \(\BMod_B\) of band modules over \(B\).
    \end{example}

    The following basic properties of the tensor product also holds in \(\mathscr{C}\).

    \begin{proposition}[Proposition 12 of \cite{Vezzani}]
        Let \(A\) and \(B\) be two commutative monoids in \(\mathscr{C}\), and let \(f: A \to B\) be a morphism in \(\Comm(\mathscr{C})\). Then, the following holds.
        \begin{enumerate}
            \item There is a natural pullback functor \(B \text{-}\Mod \to A\text{-}\Mod\) that sends an object \(N\) to \(N\) itself, considered as a \(A\)-module with the action defined as the composite
            \begin{equation*}
                A \otimes N \to B \otimes N \to N.
            \end{equation*}
        In particular, the map \(f\) defines a natural structure of \(A\)-module on \(B\), with the action defined as above.
        \item The pullback functor has a left adjoint, indicated by \(-\otimes_A B\), which sends an \(A\)-module \(M\) to \(M \otimes_A B\), with a natural \(B\)-action.
        \item The pullback functor has a right adjoint, indicated by \(\Hom_{A\text{-}\Mod}(B,-)\), which sends an \(A\)-module \(M\) to \(\Hom_{A\text{-}\Mod}(B, M)\), with a natural \(B\)-action.
        \item The pushout in \(\Comm(\mathscr{C})\) of a diagram \(B \leftarrow A \rightarrow C\) is isomorphic as \(A\)-module to \(B \otimes_A C\).
        \end{enumerate} 
     \end{proposition}

     This allows us to define the Zariski topology on the category of affine schemes.

     \begin{definition}[Affine schemes]
         Let \(\mathscr{C}\) be a cosmos. We denote by \(\Aff_\mathscr{C}\) the opposite category of \(\Comm(\mathscr{C})\) and call its objects by \emph{affine schemes relative to} \(\mathscr{C}\). We denote by \(\Spec A\) the object of \(\Aff_\mathscr{C}\) corresponding to the commutative monoid \(A\) in \(\Comm(\mathscr{C})\).
     \end{definition}

\begin{definition}[Zariski open immersion  {\cite[Definition 2.9]{ToenVaquie}}]\label{definition2.9}
    Let \(A\) and \(B\) be commutative monoids in \(\mathscr{C}\) and let \(f: A \to B\) be a morphism in \(\Comm(\mathscr{C})\).
\begin{enumerate}
    \item \(f\) is \textit{flat} if the induced functor \(- \otimes C: A\text{-}\Mod \to B\text{-}\Mod\) preserves finite limits;
    \item \(f\) is of \textit{finite presentation} if for every directed system \((C_\lambda)_{\lambda \in \Lambda}\) of \(A\)-algebras, the natural morphism 
    \begin{equation*}
        \varinjlim_{\lambda \in \Lambda} \Hom_{\Alg_A}(B,C_\lambda) \longrightarrow \Hom_{\Alg_A}(B, \varinjlim_{\lambda \in \Lambda}C_\lambda)
    \end{equation*}
    is an isomorphism.
    \item The induced morphism \(\Spec f : \Spec B \to \Spec A\) is called a \textit{Zariski open immersion} if \(f\) is a flat epimorphism of finite presentation.
\end{enumerate}
\end{definition}

\begin{definition}[Zariski topology  {\cite[Definition 2.10]{ToenVaquie}}]\label{definition2.10}
\begin{enumerate}
    \item A family of morphisms \[\{\varphi_i : \Spec(A_i) \to \Spec(A) \}_{i \in I}\] in \(\Aff_\mathscr{C}\) is a \textit{fpqc covering} if all the morphisms \(\varphi_i\) are flat and if there is a finite subset \(J \subseteq I\) such that the functor 
    \begin{equation*}
        \Phi : \prod_{j \in J} (- \otimes_A A_j) : A\text{-}\Mod \longrightarrow \prod_{j \in J} A_j\text{-}\Mod
    \end{equation*}
    is conservative (that is, \(f: M \to N\) is an isomorphism if and only if \(\Phi(f)\) is an isomorphism).
    \item A family of morphisms
    \begin{equation*}
        \left\{ \varphi_i:  X_i \longrightarrow X\right\}_{i \in I}
    \end{equation*}
    in \(\Aff_\mathscr{C}\) is a \emph{Zariski covering} if it is a fpqc covering and if every morphism \(\varphi_i\) is a Zariski open immersion.
    \item The \textit{Zariski topology} of \(\Aff_\mathscr{C}\) is the Grothendieck pretopology defined by Zariski coverings.
\end{enumerate}
\end{definition}

 We denote by \(\Pr(\Aff_\mathscr{C})\) the category of presheaves of sets over \(\Aff_\mathscr{C}\) and by \(\Sh(\Aff_\mathscr{C})\) its subcategory of sheaves in the Zariski topology. By the Yoneda lemma we have a fully faithful functor 
    \begin{align*}
        h_{-}: \Aff_\mathscr{C} &\longrightarrow \Pr(\Aff_\mathscr{C})\\
        X &\longmapsto (h_X: Y \mapsto \Hom_{\Aff_\mathscr{C}}(Y,X)).
    \end{align*}
    So, from now on, we consider \(\Aff_\mathscr{C}\) as the subcategory of representable functors in \(\Pr(\Aff_\mathscr{C})\). Furthermore, by \cite[Corollary 2.11]{ToenVaquie}, the image of \(h_{-}\) is contained in \(\Sh(\Aff_\mathscr{C})\).

    \begin{definition}[Definition 2.12 of \cite{ToenVaquie}]\label{definition2.12}
    \begin{enumerate}
        \item Let \(X \in \Aff_\mathscr{C}\) be an affine scheme and let \(F\) be a sub-sheaf of \(h_X\). We say that \(F\) is a \emph{Zariski open} of \(h_X\) if there exists a family of Zariski open immersions \(\{X_i \to X \}_{i \in I}\) such that \(F\) is the image of the morphism of sheaves
        \begin{equation*}
            \coprod_{i \in I} h_{X_i} \longrightarrow h_X.
        \end{equation*}
        \item A morphism \(f: F \to G\) in \(\Sh(\Aff_\mathscr{C})\) is said to be a \emph{Zariski open immersion} if for every affine scheme \(X\) and every morphism \(h_X \to G\), the induced morphism 
        \begin{equation*}
            F \times_G h_X \longrightarrow h_X
        \end{equation*}
        is a monomorphism such that its image is a Zariski open of \(h_X\).
    \end{enumerate}  
    \end{definition}

    \begin{remark}
        By \cite[Lemma 2.14]{ToenVaquie}, A morphism \(f: Z \to Y\) of affine schemes is a Zariski open immersion in the sense of Definition \ref{definition2.9} if and only if it is a Zariski open in the sense of Definition \ref{definition2.12}.
    \end{remark}

    Finally, we can define the notion of schemes relative to \(\mathscr{C}\).

  \begin{definition}[Relative schemes, {\cite[Definition 2.15]{ToenVaquie}}]
        A sheaf \(F \in \Sh(\Aff_\mathscr{C})\) is said to be a \textit{scheme relative to} \(\mathscr{C}\) if there exits an \emph{affine Zariski covering} of \(F\), that is, a family of affine schemes \(X_i\) and a morphism 
        \begin{equation*}
            p: \coprod_{i \in I}h_{X_i} \longrightarrow F
        \end{equation*}
        satisfying the following conditions:
        \begin{enumerate}
            \item \(p\) is an epimorphism of sheaves.
            \item For every \(i \in I\), the morphism \(h_{X_i} \to F\) is a Zariski open immersion.
        \end{enumerate}
        \end{definition}
        The category of schemes relative to \(\mathscr{C}\), denoted by \(\Sch^{\text{rel}}_\mathscr{C}\), is the full subcategory of \(\Sh(\Aff_\mathscr{C})\) formed by the schemes relative to \(\mathscr{C}\) as defined above.

\subsection{Flat morphisms of bands}

      We adapt Definition \ref{definition2.9} for the case of bands.

\begin{definition}[Adapted from Definition \ref{definition2.9}]
    Let \(f: A \to B\) be a morphism of bands. Then, \(f\) is \emph{flat} if and only if the induced functor
    \[ - \otimes_A B : \BMod_A \rightarrow \BMod_B\] 
    preserves finite limits. 
\end{definition} 

The main result of this section is the following characterization of flat band algebras as localization maps.

\begin{proposition}\label{prop: localization = flat}
    Let \(f: B \to C\) be a morphism of bands. Then, \(f\) is flat if and only if \(f\) is a localization map (see Definition \ref{def: localization map}).
\end{proposition} 

Before we start the proof of the proposition, we prove two necessary lemmas.

\begin{lemma}\label{flat bijection}
    Let \(f : B \to C\) be a bijective morphism of bands. Then, \(f\) is flat if and only if it is an isomorphism.
\end{lemma}

\begin{proof}
    Every isomorphism is flat, so we just prove the converse. Suppose \(f\) is flat.  Assume, without loss of generality, that \(B = C\) as pointed monoids and \(f\) is the identity map. In this case, we have \(N_B \subseteq N_C\). We claim that \(N_B = N_C\). Suppose, by contradiction, that \(N_B \neq N_C\) and consider the formal sum 
    \begin{equation*}
        \sigma \coloneqq  x_1+ \cdots + x_n \in N_C\setminus N_B \quad (x_i \in C)
    \end{equation*}
    with minimal length. In particular, \(x_i \neq 0\) for every \(i\). Now, consider the following formal sums in \(N_{B\otimes_B C}\):
    \begin{align*}
        \tau&\coloneqq 1\otimes x_1 + \cdots + 1\otimes x_{n-1} + 1\otimes x_n &\in N_{B \otimes_B C}\\
        \nu&\coloneqq (-1)^0 \otimes 1 + \cdots (-1)^{n-2}\otimes1 + r_n \otimes 1 &\in N_{B \otimes_B C},
    \end{align*}
    where \(r_n = -1\) if \(n\) is even and \(r_n = 0\) if \(n\) is odd, that is, the sum \((-1)^0 + \cdots + (-1)^{n-2} + r_n \in N_B\). Observe that, since \(B = C\) as sets, we have 
    \begin{equation*}
        \tau = x_1 \otimes 1 + \cdots x_n \otimes 1.
    \end{equation*}
    Now, since \(f\) is flat, we have an isomorphism of \(C\)-modules 
    \begin{equation*}
      \Phi:  \left(B \times B\right) \otimes_B C \overset{\sim}{\longrightarrow} \left(B \otimes_B C \right) \times \left(B \otimes_B C \right);\quad (m_1, m_2) \otimes c \mapsto (m_1 \otimes c, m_2 \otimes c).
    \end{equation*}
    Observe that, by the definition of the null set of tensor product, the formal sum 
    \begin{align*}
        \mu = (x_1 \otimes 1, (-1)^0 \otimes 1) + (x_2 \otimes 1, &(-1)^1 \otimes 1) + \cdots\\ &\cdots + (x_{n-1} \otimes 1, (-1)^{n-2} \otimes 1) + (x_n \otimes 1, r_n\otimes 1),
    \end{align*}
    belongs to \(N_{\left(B \otimes_B C \right) \times \left(B \otimes_B C \right)}\). Hence, its pre-image by \(\Phi^+\)
    \begin{align*}
      (\Phi^+)^{-1}(\mu) = (x_1, (-1)^0)\otimes 1 + (x_2, (-1)^1)\otimes 1 + &(x_3, (-1)^2)\otimes 1 + \cdots\\
      &\cdots + (x_{n-1}, (-1)^{n-2}) + (x_n, r_n)\otimes 1
    \end{align*}
    belongs to \(N_{\left(B \times B\right) \otimes_B C}\). Now, by the definition of null set of tensor product, \((\Phi^+)^{-1}(\mu)\) is a \(B\)-linear combination of formal sums of the following two forms:
    \begin{align}
        \sum_i (m_1^i,m_2^i) &\otimes 1, \text{ with } \sum_i (m_1^i,m_2^i) \in N_{B \times B}, \text{ or }\\
        \sum_i(m_1, m_2) &\otimes c_i, \text{ with } \sum_i c_i \in N_C.
    \end{align} 
    Suppose the formal sums of type \((2)\) appears as subsums of \(\sigma\). Since \(\sum_i c_im_1 \in N_C\), the minimality of \(\sigma\) implies that \(\sum_i c_im_1 \in N_B\). Therefore, this implies that \(\sigma \in N_B\), a contradiction.
\end{proof}

\begin{lemma}\label{localization is flat}
    Let \(B\) be a band and \(S \subseteq B\) a multiplicative subset. Then, the localization map \(\iota_S: B \to S^{-1}B\) is flat.
\end{lemma}

\begin{proof}
    We need to check that \(- \otimes_B S^{-1}B\) preserves finite products and equalizers. Then, from Lemma \ref{localization tensor}, we conclude that \(\iota_S\) is flat. 
    
    \textbf{Finite products:} Let \((M_i)_{i = 1}^n\) be a family of \(B\)-modules and consider the following map 
    \begin{align*}
        \Phi: S^{-1}\left(\prod_{i =1}^n M_i\right) &\longrightarrow \prod_{i =1}^n S^{-1}M_i\\
        (m_i)_{i = 1}^n /s &\longmapsto \left(m_i/s \right)_{i = 1}^n.
    \end{align*}
    The proof that \(\Phi\) is well-defined and bijective follows by usual arguments. We check the correspondece of their null sets. By definition, we have the following equalities:
    \begin{align*}
        I_{S^{-1}(\prod M_i)} &= \left< \iota_S^+(I_{\prod M_i}) \right>_{S^{-1}B}\\
            &= \left<\sum_j \frac{(m_{i,j})_i}{1} \;\middle|\; \sum_j (m_{i,j})_i \in I_{\prod M_i} \right>_{S^{-1}B}\\
            &= \left<\sum_j \frac{(m_{i,j})_i}{1} \;\middle|\; \sum_j m_{i,j} \in I_{M_i} \text{ for all } i \right>_{S^{-1}B}.
    \end{align*}
    Hence, the image of \(I_{S^{-1}(\prod M_i)}\) by \(\Phi^+\) is 
    \begin{align*}
        \Phi^+(I_{S^{-1}(\prod M_i)}) &= \left<\sum_j \left(\frac{m_{i,j}}{1}\right)_i \;\middle|\; \sum_j m_{i,j} \in I_{M_i} \text{ for all } i \right>_{S^{-1}B}\\
        &=  \left<\sum_j \left(\frac{m_{i,j}}{1}\right)_i \;\middle|\; \sum_j \frac{m_{i,j}}{1} \in I_{S^{-1}M_i} \text{ for all } i \right>_{S^{-1}B}\\
        & = I_{\prod S^{-1}M_i}.
    \end{align*}
    This proves that \(\Phi\) is an isomorphism.

    \textbf{Equalizers:} Let \(f\) and \(g\) be morphisms of \(B\)-modules from \(M\) to \(N\). Then, we have the following bijection that preserves the action of \(B\).
    \begin{align*}
        \Psi: S^{-1}\eq(f,g) &\longrightarrow \eq(S^{-1}f, S^{-1}g)\\
        x/s &\longmapsto x/s
    \end{align*}
    We check the correspondence of their null sets. Observe that for every \(m/s \in \eq(S^{-1}f, S^{-1}g)\), there exists \(m' \in \eq(f,g)\) and \(s' \in S\) such that \(m/s = m'/s'\) in \(S^{-1}M\). With this in mind, we have the following equalities:
    \begin{align*}
        I_{\eq(S^{-1}f, S^{-1}g)} &= \left\{ \sum_i \frac{m_i}{s_i} \in I_{S^{-1}M}\;\middle|\; \frac{m_i}{s_i} \in \eq(S^{-1}f, S^{-1}g) \right\}\\
        &= \left\{ \sum_i \frac{m'_i}{s'} \in I_{S^{-1}M}\;\middle|\; m'_i \in \eq(f,g) \right\}\\
        &= \left< \sum_i \frac{m'_i}{1} \in I_{S^{-1}M} \;\middle|\; m'_i \in \eq(f,g) \right>_{S^{-1}B}\\
        &= I_{S^{-1}\eq(f,g)}.
    \end{align*}
    This proves that \(\Psi\) is an isomorphism.
\end{proof}

\begin{proof}[Proof of Proposition \ref{prop: localization = flat}]
    We follow the idea of the proof of \cite[Proposition 6.1]{Lorscheid1}. By Lemma \ref{localization is flat}, every localization map is flat, so it remains to prove the converse. Let \(f: B \to C\) be a flat morphism of bands and define \(S \coloneqq f^{-1}(C^\times)\). By \cite[Proposition 1.22]{Lorscheid3}, there exists a unique morphism \(f_S: S^{-1}B \to C\) making the diagram
    \begin{equation*}
        \begin{tikzcd}
            B \arrow[rr, "f"] \arrow[rd, swap, "\iota_S"] & & C \\
            & S^{-1}B \arrow[ru, swap, "f_S"]
        \end{tikzcd}
    \end{equation*}
    commute. We claim that \(f_S\) is an isomorphism.

    Since every element of a tensor product of band modules is a \emph{pure tensor}, the proof that \(f_S\) is surjective is exactly the same as in \cite[Proposition 6.1]{Lorscheid1}. Hence, we focus on the proof of injectivity. Let \(a/s\) and \(b/t\) be elements of \(S^{-1}B\) such that \(f_S(a/s) = f_S(b/t)\). By the definition of \(f_S\), this means that
    \begin{equation}\label{sonma}
        f(at) = f(bs).
    \end{equation}

    For each \(\alpha \in B\), let \(g_\alpha : B \to B\) be the morphism of \(B\)-modules defined by \(g_\alpha(x) = \alpha x\) for all \(x \in B\). Since \(f\) is flat, we have the following isomorphism of \(C\)-modules:
    \begin{align*}
        \eq(g_{at}, g_{bs}) \otimes_B C &\overset{\sim}{\longrightarrow} \eq(g_{at} \otimes_B \id_C, g_{bs} \otimes_B \id_C),\\
        x \otimes c &\longmapsto x \otimes c.
    \end{align*}
    Observing that, by \eqref{sonma}, \(g_{at} \otimes_B \id_C = g_{bs} \otimes_B \id_C\), we obtain the following isomorphism of \(C\)-modules:
    \begin{align*}
        \eq(g_{at}, g_{bs}) \otimes_B C &\overset{\sim}{\longrightarrow} C,\\
        x \otimes c &\longmapsto f(x)\, c.
    \end{align*}
    Hence, there exists \(l \in \eq(g_{at}, g_{bs})\) such that \(f(l) \in C^\times\), that is, \(l \in S\). Therefore, we have \(l a t = l b s\) and \(a/s = b/t\) in \(S^{-1}B\).

    Finally, since \(f_S\) is a flat bijective morphism of bands, Lemma~\ref{flat bijection} implies that \(f_S\) is an isomorphism.
\end{proof}

\subsection{Finite presentation of band algebras} \label{subsec: finite presentation}

    In this section, we focus on finitely presented morphisms of bands. We recall the definition used in relative algebraic geometry. 
        
    \begin{definition}[Adapted from Definition \ref{definition2.9}]\label{def: fp}
    A morphism \(f: B \rightarrow C\) of bands is of \textit{finite presentation} if for all directed systems of \(B\)-algebras \(\DD = (D_i)_{i \in \Lambda}\), the natural morphism 
    \begin{equation*}\label{canonical map}
        \Psi_{\DD} : \varinjlim_{i \in \Lambda} \Hom_B(C,D_i) \longrightarrow \Hom_B(C, \varinjlim_{i \in \Lambda} D_i)
    \end{equation*}
    is an isomorphism.
    \end{definition}

        Let \((\Lambda,\leq)\) be a directed set and let \(\DD = (D_i, g_{i,j})_{i,j \in \Lambda}\) be a directed system of \(B\)-algebras \(D_i\), with morphisms of \(B\)-algebras \(g_{i,j} : D_i \to D_j\) for every \(i \leq j\). Observe that, as in the case of blueprints \cite[Proof of Proposition 5.1]{Lorscheid1}, the colimit of \(\DD\) can be represented by the following \(B\)-algebra: Its underlying pointed monoid is the colimit of \(\DD\) as pointed monoids, that is, the quotient 
        \begin{equation}\label{colimit of algebras}
            \varinjlim \DD \coloneqq \left(\coprod_{i \in \Lambda} D_i \right)\slash \sim\;,
        \end{equation}
        where \(\sim\) is the equivalence relation defined as: For \(a_i \in D_i\) and \(b_j \in D_j\), we have \(a_i \sim b_j\) if and only if there is a \(k \in \Lambda\) such that \(i,j \leq k\) and \(g_{i,k}(a_i) = g_{j,k}(b_j)\).
        
        Let \(\iota_j: D_j \to \varinjlim \DD\) be the canonical map for every \(j \in \Lambda\). It naturally induces a morphism of semirings \(\iota_j^+ : D_j^+ \to (\varinjlim \DD)^+\). Using this morphism, we define the null set of \(\varinjlim \DD\) as  
        \begin{equation}\label{null set colimit}
            N_{\varinjlim \DD} \coloneqq  \bigcup_{j \in \Lambda} \iota_j^+(N_{D_j}).
        \end{equation}
        By definition, \(\iota_j\) is a morphism of \(B\)-algebras. 
        
        Now, let \(f: B \to C\) be a morphism of bands. Applying the above construction to the directed system \((\Hom_B(C,D_i))_{i\in \Lambda} \), we see that every morphism in \(\varinjlim_{i \in \Lambda} \Hom_B(C,D_i)\) is representable by a morphism in \(\Hom_B(C, D_i)\) for some \(j \in \Lambda\). This leads to the following explicit representation of \(\Psi_{\DD}\):
        \begin{align}\label{canonical map2}
        \begin{split}
        \Psi_{\DD} : \varinjlim_{i \in \Lambda} \Hom_B(C,D_i) &\longrightarrow \Hom_B(C, \varinjlim \DD)\\
        [\varphi_j \in \Hom_B(C,D_j)] &\longmapsto \iota_j \circ \varphi_j.
        \end{split}
        \end{align}

    \begin{remark}\label{rmk: colimit of free algebras}
        Let \(B\) be a band and consider the quotient of a free \(B\)-algebra \(D\coloneqq B[x_i]_{i \in I}\sslash \left<S\right>\). For every pair of finite subsets \(J \subseteq I\) and \(T \subseteq S\), where all formal sums in \(T\) involve only \(x_i\) with \(i \in J\), define the \(B\)-algebra
 \begin{equation*}
     D_{J,T} \coloneqq B[x_i]_{i \in J}\sslash \left<T \right>.
 \end{equation*}
 Additionally, for every such pairs of finite subsets \(J \subseteq J' \subseteq I\) and \(T \subseteq T' \subseteq S\), we have a natural morphism of \(B\)-algebras \(g_{(J,T), (J',T')} : D_{J,T} \to D_{J', T'}\). This defines a directed system of \(B\)-algebras \(\DD \coloneqq (D_{J,T})_{J,T}\) and its colimit \(\varinjlim \DD\), by the representation \eqref{colimit of algebras}, is isomorphic to \(B[x_i]_{i \in I}\sslash \left<S\right>\).
    \end{remark}

    \begin{proposition}\label{prop: characterization of fp}
    Let \(f : B \to C\) be a morphism of bands. Then, the following are equivalent:
    \begin{enumerate}
        \item[\((i)\)] \(f\) is of finite presentation;
        \item[\((ii)\)] There exists an isomorphism of \(B\)-algebras 
    \begin{equation}\label{finite presentation}
        \xi: B[x_l]_{l \in I}\sslash \left<S\right> \longrightarrow C\;,
    \end{equation}
    where \(\left<S \right>\) is the null ideal generated by \(S \subseteq B[x_l]_{l \in I}^+\), such that \(I\) and \(S\) are finite sets. 
    \end{enumerate}
    \end{proposition}

    We follow, with some modifications, the idea of the proof of \cite[Proposition 5.1]{Lorscheid1}. We will use the following easy observation during the proof.

\begin{lemma}\label{lemma: minimal}
    Let \(B\) be a band, \(I\) an index set and \(S \subseteq (B[x_i]_{i \in I})^+\). If \(a\prod_i x_i^{e_i} \in B[x_i]_{i \in I}\) is such that \(e_n = 0\) for some \(n \in I\), then we have the following isomorphism of \(B\)-algebras
    \begin{equation*}
        B[x_i]_{i \in I}\sslash \left<x_n -  a\prod_{i \in I} x_i^{e_i}, S\right> \overset{\sim}{\longrightarrow} B[x_i]_{i \in I \setminus \{n\}}\sslash \left<S'\right>,
    \end{equation*}
    where \(S'\) is the set of formal sums of \(S\) with all \(x_n\) substituted by \(a\prod_i x_i^{e_i}\).
\end{lemma}

\begin{proof}
    Apply \cite[Proposition 1.8]{Lorscheid1} for both sides.
\end{proof}

     \begin{proof}[Proof of Proposition {\ref{prop: characterization of fp}}]
        Suppose \(f\) satisfies \((ii)\). Let \(\DD = (D_i, g_{i,j})_{i,j \in \Lambda}\) be a directed system of \(B\)-algebras \(D_i\), with morphisms of \(B\)-algebras \(g_{i,j} : D_i \to D_j\) for every \(i \leq j\), and let \(\Psi_{\DD}\) be the canonical map \eqref{canonical map2}. First, we show that \(\Psi_{\DD}\) is injective. If \(\Psi_{\DD}([\varphi_i]) = \Psi_{\DD}([\varphi_j])\) for some \(\varphi_i \in \Hom_B(C,D_i)\) and \(\varphi_j \in \Hom_B(C,D_j)\), then, by the representation \eqref{colimit of algebras} of directed colimit of algebras, there exists \(k \geq i,j\) such that 
        \begin{equation*}
            (g_{i,k} \circ \varphi_i \circ \xi)([T_l]) = (g_{j,k} \circ \varphi_j \circ \xi) ([T_l])\text{ for every } l \in I.
        \end{equation*}
        Now, for every \(c = b\prod_{l \in I} T_l^{e_l} \in B[T_l]_{l \in I}\sslash \left<S\right>\), we have 
        \begin{equation*}
            (g_{i,k} \circ \varphi_i)(\xi(c)) = b\prod_{l \in I}(g_{i,k} \circ \varphi_i \circ \xi)([T_l])^{e_l} = b\prod_{l \in I}(g_{j,k} \circ \varphi_j \circ \xi)([T_l])^{e_l} = (g_{j,k} \circ \varphi_j)(\xi(c)).
        \end{equation*}
        Since \(\xi\) is surjective, this proves that, in \(\varinjlim_{i \in \Lambda} \Hom_B(C,D_i)\), we have
        \begin{equation*}
            [\varphi_i]=[g_{i,k} \circ \varphi_i] = [g_{j,k} \circ \varphi_j] = [\varphi_j].
        \end{equation*}

        Next, we show that \(\Psi_{\DD}\) is surjective. Let \(\varphi : C \rightarrow \varinjlim \DD\) be a morphism of \(B\)-algebras. Then, since \(I\)  and \(S\) are finite sets, there exist \(k \in \Lambda\) and a list \((y_l)_{l \in I}\) in \(D_k\) such that \(\iota_k(y_l) = (\varphi \circ \xi)[x_l]\) for every \(l \in I\), and \((\varphi^+\circ \xi^+)(S) \subseteq N_{D_k}\). Hence, by \cite[Proposition 1.8 and 1.13]{Lorscheid3}, there exists a morphism \(\psi : B[x_l]_{l \in I}\sslash \left<S\right> \to D_k\) such that \(\iota_k \circ (\psi \circ \xi^{-1}) = \varphi\), that is, \(\Psi_{\DD}([\psi \circ \xi^{-1}]) = \varphi\). Therefore, \(\Psi_{\DD}\) is a bijection.

        Conversely, suppose that \(f\) is of finite presentation. We know, by \cite[Propositions 1.8 and 1.13]{Lorscheid3}, that there exists an isomorphism of \(B\)-algebras of the form 
        \begin{equation*}
            B[x_i]_{i \in I}\sslash \left<S \right> \overset{\sim}{\longrightarrow} C.
        \end{equation*}
        Consider an isomorphism as above with \(I\) minimal and for this minimal \(I\) let \(S\) be minimal. We claim that both \(I\) and \(S\) are finite sets. Let \(\DD = (D_{J,T})_{J,T}\) be the directed system of \(B\)-algebras constructed as in Remark \ref{rmk: colimit of free algebras} for \(B[x_i]_{i \in I}\sslash \left<S\right>\). Since \(\Phi_{\DD}\) in bijective, the identity map \(\id_C\) is the image of the class of some morphism of \(B\)-algebras \(\varphi_{J,T} : C \rightarrow D_{J,T}\). That is, we have 
        \begin{equation}\label{eq: 1}
            \id_C = \iota_{J,T} \circ \varphi_{J,T},
        \end{equation}
        where \(\iota_{J,T} : D_{J,T} \to \varinjlim \DD \cong C\) is the canonical map. First, observe that \(J = I\). Indeed, if there were an \(n \in I \setminus J\), then, since \(\iota_{J,T}\) is surjective, there is an element \(a\prod_{i \in J}x_i^{e_i}\) such that \(x_n - a\prod_{i \in J}x_i^{e_i} \in \left<S\right>\), which by Lemma \ref{lemma: minimal} contradicts the minimality of \(I\). Next, notice that \eqref{eq: 1} holds for every finite subset \(T' \subseteq S\) containing \(T\), hence we may choose \(T\) containing, for every \(l \in I\), the sum of the form \(x_l - y_l\), where \(y_l \in B[x_i]_{i \in I}\) is such that \([y_l] = \varphi([x_l])\). We claim that \(\iota_{J,T}\) is an isomorphism. To prove this, observe that by the hypothesis on \(\left<T\right>\) above, we have \(\varphi([x_l]) = [x_l]\) for every \(l \in I\). Hence, \(\varphi\) and \(\iota_{I,T}\) are inverse morphisms to each other. Therefore, \(\iota_{J,T}\) is an ismorphism. 
    \end{proof}

    \begin{corollary}\label{corollary main}
    Let \(f: B \to C\) be a morphism of bands. Then, \(f\) is a flat epimorphism of finite presentation if and only if \(f\) is a finite localization map (see Definition \ref{def: localization map})
    \end{corollary}

    \begin{proof}
    Let \(f: B \to C\) be a finite localization map. Then, there exists a finitely generated multiplicative subset \(S = \left<s_1, \cdots, s_n \right>\) of \(B\) and an isomorphism \(\phi: S^{-1}B \to C\) such that \(\phi \circ \iota_S = f\). Since the following map
    \begin{equation*}
        \varphi: S^{-1}B \overset{\sim}{\longrightarrow} B[x_1, \cdots, x_n]\sslash \left<s_1x_1 - 1, \cdots, s_nx_n - 1 \right>, 
    \end{equation*}
    is an isomorphism, and every localization map is a flat epimorphism by Lemma \ref{localization is flat}, we conclude that \(f\) is a flat epimorphism of finite presentation.

    Conversely, suppose that \(f: B \to C\) is a flat epimorphism of finite presentation. Since \(f\) is of finite presentation, we may assume that \(C = B[x_1, \cdots, x_n]\sslash \left<T\right>\), where \(T\) is a finite set. Now, since \(f\) is flat, there exist a multiplicative subset  \(S \subseteq B\) and an isomorphism \(\phi: S^{-1}B \to C\) making the following diagram commutative:
    \begin{equation*}
        \begin{tikzcd}
            B \arrow[r, "f"] \arrow[d, "\iota_S"] & B[x_1, \cdots, x_n]\sslash \left<T \right>\\
            S^{-1}B \arrow[ru, "\phi"].
        \end{tikzcd}
    \end{equation*}
    By the proof of Proposition \ref{prop: localization = flat}, we know that
    \begin{align*}
        S &= f^{-1}\left\{\left( B[x_1, \cdots, x_n]\sslash \left<T\right> \right)^\times\right\}\\
        &= \left<  f^{-1}([x_i]) \;\middle|\; [x_i] \in \left(B[x_1, \cdots, x_n]\sslash \left<T\right>\right)^\times \right>.
    \end{align*}
    Therefore, \(S\) is finitely generated and \(f\) is a finite localization map.
   \end{proof}

\subsection{Equivalence between relative and geometric band schemes}

The proof of Theorem \ref{main theorem 2} follows the same ideas of the proof of \cite[Theorem 36]{Vezzani}.

\begin{proposition}\label{lemma: zariski covering isomorphism}
    Let \(B\) be a band and \(\{\varphi_i: \Spec B_i \to \Spec B\}_{i \in I}\) be a Zariski covering. Then there exists a \(j \in I\) such that \(\varphi_j: \Spec B_j \to \Spec B\) is an isomorphism.
\end{proposition}

\begin{proof}
    For every \(i \in I\), let \(f_i: B \to B_i\) be the map on \(\Comm(\BMod_{\FF_1^\pm})\) corresponding to \(\varphi_i\). Let \(J \subseteq I\) be a finite subset such that the functor \(\prod_{j \in J} (- \otimes_B B_j)\) is conservative. By definition of Zariski covering, each \(f_j\) is a flat epimorphisms of finite presentation, hence, by Corollary \ref{corollary main}, we have \(B_j \cong B[h_j^{-1}]\) for some \(h_j \in B\) for every \(j \in J\). Suppose that \(h_j \not\in B^\times\) for every \(j \in J\), and consider the \(B\)-module \(M\coloneqq  B/(B\setminus B^\times)\) where \(B\setminus B^\times\) is the \(B\)-submodule of non-invertible elements of \(B\). Then, for every \(j \in J\), we have an isomorphism \(M \otimes_B B_j \cong 0\) since 
    \begin{equation*}
        [b] \otimes \frac{c}{h_j} = [b] \otimes \frac{ch_j}{h_j^2} = [h_j]\otimes \frac{bc}{h_j^2} = 0.
    \end{equation*}
    This implies that \(M \to 0\) is an isomorphism, hence \(M = 0\) and \(B^\times = \emptyset\), which is not possible since \(B^\times\) always contains \(1\). Therefore, there exists a \(j \in J\) such that \(h_j \in B^\times\), which implies that the morphism \(f_j: B \to B_j\) is an isomorphism.
\end{proof}

We already know from the results of Section~\ref{sec:Modules over Bands} that \(\Aff_{\BMod_{\FF_1^\pm}}\) is equivalent to \(\BAff^{\text{geom}}\). Proposition \ref{lemma: zariski covering isomorphism} implies a correspondece between \emph{geometrical} Zariski coverings and \emph{relative} Zariski coverings. 

\begin{corollary}
    The Zariski topology on \(\BAff^{\text{geom}}\) in the sense of Definition~\ref{Zariski topology geom} is equivalent to the Zariski topology on \(\Aff_{\BMod_{\FF_1^\pm}}\) in the sense of Definition~\ref{definition2.10}.
\end{corollary}

\begin{proof}
     Follows from  \cite[Section~2.2.3]{Lorscheid3}.
\end{proof}

From now on, we denote the equivalent sites \(\Aff_{\BMod_{\FF_1^\pm}}\) and  \(\BAff^{\text{geom}}\) simply by \(\BAff\). 

\begin{lemma}\label{proposition 35}
    Let \(\mathcal{F}\) be a sheaf in \(\Sh(\Aff_{\BMod_{\FF_1^\pm}})\), and \(h_{\Spec B}\) an affine scheme in \(\Aff_{\BMod_{\FF_1^\pm}}\). Then, a morphism \(f: \mathcal{F} \to h_{\Spec B}\) is an open immersion if and only if \(\mathcal{F}\) is isomorphic over \(h_{\Spec B}\) to \(h_U \coloneqq \Hom(-, U)\), where \(U\) is an open geometrical band subscheme of \(\Spec B\).
\end{lemma}

\begin{proof}
    Same proof as \cite[Proposition 35]{Vezzani}.
\end{proof}

\begin{theorem}\label{thm:bandschemes=relativeschemes}
    The category of geometric band schemes is equivalent to the category of schemes relative to the cosmos \(\BMod_{\FF_1^\pm}\).
\end{theorem}

\begin{proof}
    The proof is verbatim the same as that of \cite[Proposition 36]{Vezzani}, except by the following modifications: 
    \begin{itemize}
        \item \cite[Proposition 5]{Vezzani} is replaced by Proposition \ref{proposition 5}.
        \item \cite[Proposition 9]{Vezzani} is replaced by Proposition \ref{proposition 9}.
        \item \cite[Lemma 34]{Vezzani} is replaced by Proposition \ref{lemma 34}.
        \item \cite[Definition 17]{Vezzani} is replaced by Proposition \ref{definition2.12}.
        \item \cite[Proposition 35]{Vezzani} is replaced by Proposition \ref{proposition 35}.
    \end{itemize}
\end{proof}

%% file: HyperringasBand.tex
\section{Hyperring as Bands}\label{sec: hyperring schemes}

In this section, we first characterize hyperrings within the category of bands. We study the relation between the algebraic theory of hyperrings and affine hyperring schemes to the setting of bands. In particular, we show that hyperideals correspond to \(k\)-ideals of bands, which leads to fundamental differences between hyperring schemes and band schemes.

 \subsection{Characterization of hyperrings as bands}

\begin{proposition}[Hyperring as bands]\label{hyperrings as bands}
    The category of hyperrings is equivalent to the full subcategory of bands consisting of those bands \(B\) that satisfy the following conditions. Let \(c_i, d_j\) be nonzero elements of \(B\):
    \begin{enumerate}
        \item[\((H1)\)] For every finite sequence \(x_1, \cdots, x_n \in B\), there exists some \(y \in B\) such that \(x_1 + \cdots + x_n + y \in N_B\).
        \item[\((H2)\)] If \(y \in B\) satisfies \(x + (-x) + y \in N_B\) for every \(x \in B\), then \(y = 0\).
        \item[\((H3)\)] For every sequence \(x_1, \cdots, x_n, z \in B\), if \(\mu \in B\) is such that \(zx_1 + \cdots + zx_n + \mu \in N_B\), then there exists \(\phi \in B\) such that \(\mu = z\phi\) and \(x_1 +  \cdots + x_n + \phi \in N_B\).
        \item[\((H4)\)] \(\sum_i c_i + \sum_j d_j \in N_B\) if and only if there exists \(\alpha \in B\) such that \(\sum_i c_i + \alpha \in N_B\) and \(\sum_j d_j - \alpha \in N_B\).
    \end{enumerate}
\end{proposition}

\begin{proof}
    Let \(B\) be a band satisfying \((H1)\) to \((H4)\). For every \(x, y \in B\), define the hyperoperation as follows:
    \[
        x \boxplus y \coloneqq \left\{ z \in B \;\middle|\; x + y + (-z) \in N_B \right\} \subseteq B.
    \]
    Condition \((H1)\) ensures that \(x \boxplus y \neq \emptyset\). The commutativity \(x \boxplus y = y \boxplus x\) holds in general, while \((H4)\) implies associativity. Indeed, using \((H4)\) we have, for every \(x, y\) and \(z\) in \(B\), the following equalities.
    \begin{align*}
        (x \boxplus y) \boxplus z &= \left\{ \mu \in B\;\middle|\; ^\exists a \in B \text{ such that } x + y + (-a) \in N_B \text{ and } a + z + (-\mu) \in N_B \right\}\\
        &= \left\{\mu \in B \;\middle|\; x + y + z + (-\mu) \in N_B\right\}\\
       &= \left\{ \mu \in B\;\middle|\; ^\exists b \in B \text{ such that } y + z + (-b) \in N_B \text{ and } x + b + (-\mu) \in N_B \right\}\\
       &= x \boxplus( y \boxplus z).
    \end{align*}
    By \((H2)\), we see that \(0\) is the unique element of \(B\) such that
    \[
        0_B \boxplus x = x \boxplus 0_B = \{x\},
    \]
    for all \(x \in B\). The axioms of a band also imply the existence and uniqueness of additive inverses, as well as the reversibility property
    \[
        x \in y \boxplus z \;\Longleftrightarrow\; z \in x \boxplus (-y).
    \]
    Let \(x,y\) and \(z\) be elements of \(B\), then, by \((H3)\), we have 
    \begin{align*}
        (x \boxplus y) z &= \left\{ \mu z \in B \;\middle|\; x + y + (-\mu) \in N_B \right\}\\
        &\overset{(H3)}{=} \left\{ \mu z \in B \;\middle|\; xz + yz + (-\mu z) \in N_B \right\}\\
        &= xz \boxplus yz
    \end{align*}
    Hence, \(B\) is a hyperring.

    Conversely, if \(R\) is a hyperring, then the associated band \(R^{\mathrm{band}}\) satisfies \((H1)\) to \((H4)\).
\end{proof}

\begin{remark}\label{k-ideal generated}
    If \(B\) is a band satisfying \((H1)\) to \((H4)\), then the \(k\)-ideal generated by \(x\), \(\langle x \rangle_k\), is  \(B \cdot x\) for every \(x \in B\).
\end{remark}

 \begin{proposition}[Hypermodules as band modules]
     The category of hypermodules over a hyperring \(R\) is equivalent to the full subcategory of band modules \(M\) over \(R^{\mathrm{band}}\) satisfying the following conditions:
     \begin{enumerate}
        \item[\((M1)\)] For every finite sequence \(m_1, \cdots, m_n \in M\), there exists some \(n \in M\) such that \(m_1 + \cdots + m_n + n \in I_M\).
        \item[\((M2)\)] If \(n \in M\) satisfies \(m + (-m) + n \in I_M\) for every \(m \in M\), then \(n = 0\).
        \item[\((M4)\)] For every sequence \(m_1, \cdots, m_n \in M\) and \(a \in B\), if \(\eta \in M\) is such that \(am_1 + \cdots + am_n + \eta \in I_M\), then there exists \(\xi \in B\) such that \(\eta = a\xi\) and \(m_1 +  \cdots + m_n + \xi \in I_M\).
        \item[\((M4)\)] \(\sum_i m_i + \sum_j n_j \in I_M\) if and only if there exists \(\alpha \in M\) such that \(\sum_i m_i + \alpha \in I_M\) and \(\sum_j n_j - \alpha \in I_M\).
    \end{enumerate}
 \end{proposition}

 \begin{proof}
     Same idea as Proposition \ref{hyperrings as bands}.
 \end{proof}

 \subsection{Affine hyperring schemes as bands}

 \begin{proposition}\label{hyperideal =k ideal}
     Let \(R\) be a hyperring and denote by \(R^{\mathrm{band}}\) its induced band as in Example \ref{Ex: hyperring as band}. Then, there exists a correspondece between (resp. prime, maximal) hyperideals of \(R\) and (resp. prime, maximal) \(k\)-ideals of \(R^{\mathrm{band}}\)
 \end{proposition}

 \begin{proof}
    Let \(I\) be a hyperideal of \(R\). Then it is clearly an \(m\)-ideal of \(R^{\mathrm{band}}\). We claim that the condition \(I \boxplus I \subseteq I\) is equivalent to condition (3) in the definition of \(k\)-ideals (Definition \ref{hyperideals}).

    Suppose first that \(I \boxplus I \subseteq I\). Let \(a, b \in I\) and let \(y \in R^{\mathrm{band}}\) be such that \(a + b + y \in N_{R^{\mathrm{band}}}\).
    By definition of the hyperaddition, this implies that \(-y \in a \boxplus b \subseteq I \boxplus I\). Hence \(-y \in I\), and therefore \(y \in I\). This shows that \(I\) is a \(k\)-ideal.

    Conversely, suppose that \(I\) is a \(k\)-ideal of \(R^{\mathrm{band}}\). Let \(a, b \in I\) and let \(z \in a \boxplus b\). By definition, this means that \(a + b + (-z) \in N_{R^{\mathrm{band}}}\). Since \(I\) is a \(k\)-ideal, it follows that \(z \in I\). Hence \(a \boxplus b \subseteq I\), and therefore \(I \boxplus I \subseteq I\).
\end{proof}

 This implies, in particular, that Proposition 2.11, Proposition 2.14 and Lemma 3.21 of \cite{Jaiung} are already generalized for the case of bands in Proposition 1.35, Proposition 1.30 and Proposition 1.37 of \cite{Lorscheid3}, respectively. Additionally, we also have a correspondence between the important construction of a hyperrings as the quotient of a ring by a subgroup of its group of units.

 \begin{definition}
     Let \(B\) be a band and let \(G\) be a subgroup of the group of units \(B^\times\). Let \(B/G\) be the set of cosets, and consider it as a band with a multiplication defined as:
     \begin{equation*}
         B/G \times B/G \to B/G;\quad (aG,bG) \mapsto abG.
     \end{equation*}
     and with null set \(N_{B/G}\) defined as:
     \begin{equation*}
         N_{B/G} \coloneqq  \left\{\sum_i x_iG\;\middle|\; \text{For every } i \text{ there exists } t_i \text{ such that } \sum_i t_i\cdot x_i \in N_B \right\}.
     \end{equation*}
 \end{definition} 

 \begin{remark}
     \(N_{B/G}\) is indeed a null set for \(B/G\). For, for each \(xG \in B/G\) if \(yG \in B/G\) is such that \(xG + yG \in N_{B/G}\), then there exist \(t,s \in G\) satsifying \(t\cdot x + s\cdot y \in N_B\). This implies that \(y = t^{-1}s \cdot(-x)\), where we conclude that \(yG = t^{-1}s \cdot(-x)G = (-x)G\). Therefore, such \(yG\) is unique. The other axioms are trivial.
 \end{remark}

 \begin{corollary}
     Let \(A\) be a commutative ring and denote by \(A^{\mathrm{band}}\) its induced band as in Example \ref{Ex: ring as band}. Then, for every multiplicative subgroup \(G \subseteq A^\times\), we have the equality
     \begin{equation*}
         \left(A/G \right)^{\mathrm{band}} =  A^{\mathrm{band}}/G.
     \end{equation*}
 \end{corollary}

 \begin{proof}
     This follows from the definition of the hyperaddition on \(A/G\) as given in Example \ref{ex: quotient hyperring}.
 \end{proof}

 Now, we study the definition of affine hyperring schemes.

\begin{definition}
    Let \(B\) be a band. The \emph{\(k\)-spectrum} of \(B\), denoted by \(\Spec^k(B)\), is the set of prime \(k\)-ideals of \(B\) with the topology generated by the closed subsets 
    \begin{equation*}
        V^k(I) =\left\{ \mathfrak{p} \in \Spec^k(B) \;\middle|\; I \subseteq \mathfrak{p} \right\},
    \end{equation*}
    for every subset \(I\) of \(B\).
\end{definition}

\begin{notation}
    Let \(f \in B\). Denote by \(D^k(f)\) the subset of \(\Spec^k(B)\) of prime \(k\)-ideals not containing \(f\). It form a basis for the topology on \(\Spec^k(B)\).
\end{notation}

\begin{corollary}
    Let \(R\) be a hyperring and denote by \(R^{\mathrm{band}}\) its induced band. Then, there exists a homeomorphism 
    \begin{equation*}
        \Spec^{\textrm{\textbf{HypRings}}}(R) \overset{\sim}{\longrightarrow} \Spec^k(R^{\mathrm{band}}),
    \end{equation*}
    where the left-hand side is the spectrum of a hyperring as defined in Definition \ref{spec hyperring}.
\end{corollary}

\begin{proof}
    Follows from Proposition \ref{hyperideal =k ideal}.
\end{proof}

\begin{definition}[Structure sheaf of \(\Spec^k(B)\)]\label{k-sheaf}
    Let \(B\) be a band, \(X = \Spec^k(B)\) its associate \(k\)-spectrum and \(U \subseteq X\) an open subset. Define the band \(\OO_X(U)\) as the pointed monoid
    \begin{equation*}
        \OO^k_X(U) \coloneqq  \left\{s: U \to \bigsqcup B_{\mathfrak{p}} \;\middle|\;(\star_1),\; (\star_2)\;  \right\}
    \end{equation*}
    where \((\star_1)\) and \((\star_2)\) are the following conditions:
    \begin{enumerate}
        \item[\((\star_1)\)] \(s(\mathfrak{p}) \in B_{\mathfrak{p}}\) for all \(\mathfrak{p} \in U\).
        \item[\((\star_2)\)] For every \(\mathfrak{p} \in U\), there exists an open neighborhood \(\mathfrak{p} \in V_{\mathfrak{p}} \subseteq U\) and \(a, f \in B\) such that for every \(\mathfrak{q} \in V_{\mathfrak{p}}\), \(f \not\in \mathfrak{q}\) and \[s(\mathfrak{q}) = \frac{a}{f} \in B_{\mathfrak{q}}.\]
    \end{enumerate}
    The multiplication of \(s, t \in \OO^k_X(U)\) is defined as 
    \begin{equation*}
        (s\cdot t) (\mathfrak{p}) \coloneqq  s(\mathfrak{p})\cdot t(\mathfrak{p})\; \text{ for every } \mathfrak{p} \in U.
    \end{equation*}
    and the null set is defined as 
    \begin{equation*}
        N_{\OO^k_X(U)} \coloneqq  \left\{\sum_i s_i \in \left(\OO^k_X(U)\right)^+\;\middle|\; \sum_i s(\mathfrak{p}) \in N_{B_{\mathfrak{p}}}\text{ for every }\mathfrak{p} \in U \right\}.
    \end{equation*}
\end{definition}

\begin{proposition}\label{stalk localization}
Let \(B\) be a band, \(X = \Spec^k(B)\), and \(\mathfrak{p} \in X\). Then, the following holds.
\begin{equation*}
    \OO^k_{X,\mathfrak{p}} \coloneqq \varinjlim_{\mathfrak{p} \in D(f)} \OO^k_X(D(f)) \overset{\sim}{\longrightarrow} B_\mathfrak{p}.
\end{equation*}
\end{proposition}

\begin{proof}
    We claim that the following map is an isomorphism of bands.
    \begin{align*}
        \Phi : \OO^k_{X,\mathfrak{p}} &\longrightarrow B_\mathfrak{p}\\
         [s] &\longmapsto s(\mathfrak{p}).
    \end{align*}
    where \([s]\) is the equivalence class of a section \(s \in \OO^k_X(D(f))\). By the definition of null sets of \(\OO_X^k(D(f))\) and the representation of filtered colimits of bands in Section \ref{subsec: finite presentation}, we have that \(\Phi\) is a well-defined morphism of bands. 

    We show that \(\Phi\) is injective. Let \([s], [t] \in \OO^k_{X,\mathfrak{p}}\) be such that \(s(\mathfrak{p}) = t(\mathfrak{p})\). Then, by property \((\star_2)\) in the definition of \(\OO^k_X\), there exist a open neighborhood \(V\) of \(\mathfrak{p}\) and elements \(a,f, b\) and \(g\) of \(B\) such that \(s(\mathfrak{q}) = a/f\) and \(t(\mathfrak{q}) = b/g\) for all \(\mathfrak{q} \in V\). In particular, \(a/f = b/g\) in \(B_\mathfrak{p}\) implying the existence of an \(h \not\in \mathfrak{p}\) such that \(hag = hbf\) in \(B\). Hence, \(V \cap D(h)\) is an open neighborhood of \(\mathfrak{p}\) such that \(s|_{V \cap D(h)} = t|_{V \cap D(h)}\), and \([s] = [t]\) in \(\OO^k_{X,\mathfrak{p}}\).

    We show that \(\Phi\) is surjective. For every \(a/f \in B_\mathfrak{p}\), consider the open neighborhood \(D^k(f)\) of \(\mathfrak{p}\) and the section \(s\) on \(D^k(f)\) defined by \(s(\mathfrak{q}) = a/f\) for every \(\mathfrak{q} \in D^k(f)\). Then, \(\Phi([s]) = a/f\).

    Finally, we verify the correspondece between their null sets. Let \(\sum_{i=1}^n [s_i]\) be such that \(\sum_{i=1}^n s_i(\mathfrak{p}) \in N_{B_\mathfrak{p}}\). Let \(U\) be an open neighborhood of \(\mathfrak{p}\) such that \(s_i\) is defined on (after a restriction if necessary) \(U\) for every \(i\). By definition, we heve \(\sum_i s_i \in N_{\OO^k_X(U)}\), hence by the representation of \(N_{\OO^k_{X,\mathfrak{p}}}\) as in Section \ref{subsec: finite presentation}, we conclude that \(\sum_i [s_i] \in N_{\OO^k_{X,\mathfrak{p}}}\). This proves that \(\Phi\) is an isomorphism.
\end{proof}

Lastly, we generalize Theorem \ref{main theorem hyperrings} for the case of bands. We will restrict our result for bands \(B\) satisfying the following conditions:

\begin{enumerate}
    \item[\((D1)\)] If \(b\cdot \left(\sum_i a_i \right) \in N_B\) and \(b\neq 0\), then \(\sum_i a_i \in N_B\).
    \item[\((D2)\)] \(\langle x \rangle_k = B \cdot x\) for every \(x \in B\).
\end{enumerate}

This condition implies that if \(c \cdot a = c\cdot b\) and \(c \neq 0\), then \(a = c\), that is, \(B\) has no non-zero divisors.  

\begin{theorem}\label{main theorem hypbands}
    Let \(B\) be a band satisfying \((D1)\) and \((D2)\), \(X = \Spec^k(B)\) and \(\OO^k_X\) be the sheaf of bands as in Definition \ref{k-sheaf}. Then, \(\OO^k_X(D^k(f))\) is isomorphic to the band \(B_f\). In particular, if \(f = 1\), we have
        \begin{equation*}
             \OO^k_X(X) \cong B.
        \end{equation*}
\end{theorem}

We adapt the proof of \cite[Theorem 4.23]{Jaiung}. First, we prove the following technical lemmas. 

\begin{lemma}\label{k-ideal fg}
    Let \(B\) be a band and \(S\) a subset of \(B\). If \(f \in \langle S \rangle_k\), then there exists a finite subset \(F \subseteq S\) such that \(f \in \langle F \rangle_k\).
\end{lemma}

\begin{proof}
    By \cite[Lemma 1.26]{Lorscheid3}, we have \( \langle S \rangle_k = \bigcup_{n=0}^{\infty} \langle S \rangle_k^{(n)}\), where \(\langle S \rangle_k^{(0)} \coloneqq S\) and, for every \(n \ge 1\),
    \begin{equation}\label{recurrence}
        \langle S \rangle_k^{(n)}
        \coloneqq
        \left\{
            a \in B
            \;\middle|\;
            a-\sum_i b_i s_i \in N_B
            \text{ for some } b_i \in B
            \text{ and } s_i \in \langle S \rangle_k^{(n-1)}
        \right\}.
    \end{equation}

    Let \(l \ge 0\) be such that \(f \in \langle S \rangle_k^{(l)}\).
    If \(l=0\), then \(f \in S\), and there is nothing to prove.
    Thus, assume that \(l \ge 1\). By \eqref{recurrence}, there exists a finite subset
    \(F^{l-1} \subseteq \langle S \rangle_k^{(l-1)}\)
    such that \(f \in \langle F^{l-1} \rangle_k\).
    If \(l-1=0\), we may take \(F=F^{l-1}\). Otherwise, for each \(g \in F^{l-1}\), applying \eqref{recurrence} again yields a finite subset
    \(F_g^{\,l-2} \subseteq \langle S \rangle_k^{(l-2)}\)
    such that \(g \in \langle F_g^{\,l-2} \rangle_k\).
    Since \(F^{l-1}\) is finite, the union
    \begin{equation*}
        F^{l-2}
        \coloneqq
        \bigcup_{g \in F^{l-1}} F_g^{\,l-2}
    \end{equation*}
    is finite, and \(h \in \langle F^{l-2} \rangle_k\) for every
    \(h \in F^{l-1}\).

    Repeating this construction, we eventually obtain a finite subset
    \(F^{0} \subseteq \langle S \rangle_k^{(0)} = S\).
    A straightforward descending induction on \(i\) shows that
    \begin{equation*}
        f \in \langle F^{0} \rangle_k.
    \end{equation*}
    Setting \(F=F^{0}\) completes the proof.
\end{proof}

\begin{lemma}\label{technical lemma 2}
    Let \(B\) be a band satisfying \((D2)\), and let \(h_1, \cdots, h_n\) and \(a_1, \cdots, a_n\) be two lists of elements of \(B\) satisfying \(a_i h_j = a_j h_i\) for every pair \(i,j\). Then, for every element \(f \in \left<h_1, \cdots, h_n \right>_k\) and every \(j\), there exists \(\beta_j \in B\) such that \(f a_j = \beta_j h_j\).
\end{lemma}

\begin{proof}
    By \cite[Lemma~1.26]{Lorscheid3}, we have \(\left<h_1, \cdots, h_n \right>_k = \bigcup_{n = 0}^\infty \left<h_1, \cdots, h_n \right>_k^{(n)}\). We proceed by induction on \(n\). Suppose that for every \(s \in \left<h_1, \cdots, h_n \right>_k^{(n-1)}\) and every \(j\), there exists \(\beta_{s,j} \in B\) such that \(s a_j = \beta_{s,j} h_j\). Let \(f \in \left<h_1, \cdots, h_n \right>_k^{(n)}\). By definition, there exist \(s_i \in \left<h_1, \cdots, h_n \right>_k^{(n-1)}\) and \(b_i \in B\) such that \(f - \sum_i b_i s_i \in N_B\). Taking the product with \(a_j\), we obtain
    \begin{equation*}
        f a_j - \sum_i b_i (s_i a_j) = f a_j - \left(\sum_i b_i \beta_{s_i,j}\right) h_j.
    \end{equation*}
    That is, \(f a_j \in \left< h_j \right>_k\). By property \((D2)\), there exists \(\beta_j \in B\) such that \(f a_j = \beta_j h_j\).
\end{proof}

\begin{proof}[Proof of Theorem~\ref{main theorem hypbands}]
    We follow the proof of \cite[Theorem~4.23]{Jaiung1}, and indicate only the necessary modifications.

    Consider the following well-defined morphism of bands:
    \begin{equation*}
        \Phi \colon R_f \longrightarrow \OO^k_X(D^k(f)),\quad 
        \frac{a}{f^n} \longmapsto s,\quad \text{where } s(\mathfrak{p}) = \frac{a}{f^n} \text{ in } R_{\mathfrak{p}}.
    \end{equation*}
    
    The injectivity of \(\Phi\) follows as in \textit{loc.\,cit.}, the only difference being the use of property \((D1)\) to show that the implication
    \[
        h f^m a - h f^n b \in N_D \;\Longrightarrow\; f^m a = f^n b \in N_D
    \]
    holds.

    We now prove that \(\Phi\) is surjective. Let \(s \in \OO^k_X(D^k(f))\). Using \cite[Proposition~1.37]{Lorscheid3} and property \((D2)\), one shows, as in the classical case, that \(D^k(f)\) admits a covering by basic open subsets \(\{D^k(h_i)\}_{i \in I}\) such that for each \(i \in I\), there exists \(a_i \in B\) satisfying
    \begin{equation*}
        s(\mathfrak{q}) = \frac{a_i}{h_i} \quad \text{for all } \mathfrak{q} \in D^k(h_i).
    \end{equation*}
    Let \(I\) be the \(k\)-ideal generated by the elements \(h_i\). Since \(D^k(f) = \bigcup_{i \in I} D^k(h_i)\), we obtain
    \[
        V^k((f)) = \bigcap_{i \in I} V^k((h_i)) = V^k(I).
    \]
    By \cite[Proposition~1.37]{Lorscheid3}, it follows that \(f^n \in I\) for some \(n\). By Lemma~\ref{k-ideal fg}, there exists a finite subset \(J \subseteq I\) such that
    \begin{equation}\label{1432}
        f^n \in \langle h_j \mid j \in J \rangle_k.
    \end{equation}
    Hence \(D^k(f)\) is covered by the finite family \(\{D^k(h_j)\}_{j \in J}\). Write this family as \(D^k(h_1),\) \(\dots, D^k(h_n)\).

    Using property \((D1)\), one shows, as in \cite[Theorem~4.23]{Jaiung1}, that
    \[
        a_i h_j = a_j h_i \quad \text{for all } i,j.
    \]
     By \eqref{1432} and Lemma~\ref{technical lemma 2}, for each \(j\) there exists \(\beta_j \in B\) such that
    \[
        f^n a_j = \beta_j h_j.
    \]
    Applying again property \((D1)\), one deduces, as in \cite[Theorem~4.23]{Jaiung1}, that \(\beta_i = \beta_j\) for all \(i,j\). Denoting this common value by \(\beta\), we obtain
    \[
        \Phi\!\left(\frac{\beta}{f^n}\right) = s,
    \]
    which proves that \(\Phi\) is surjective.

    Finally, we verify the compatibility of the null sets. Let \(\sum_i \frac{a_i}{f^n} \in (B_f)^+\) (after taking a common denominator) be such that \(\sum_i s_i \in N_{\OO^k_X(D^k(f))}\), where \(\Phi\!\left(\frac{a_i}{f^n}\right) = s_i\) for every \(i\). By definition, for every \(\mathfrak{p} \in D^k(f)\), we have
\[
    \sum_i s_i(\mathfrak{p}) = \sum_i \frac{a_i}{f^n} \in N_{B_{\mathfrak{p}}} = \left\langle \sum_j \frac{b_j}{1} \;\middle|\; \sum_j b_j \in N_B \right\rangle_{B_\mathfrak{p}}.
\]
Hence, there exist \(b_i \in B\) and \(h \notin \mathfrak{p}\) such that \(a_i/f^n = b_i/h\) for every \(i\), and \(\sum_i b_i \in N_B\). By property \((D1)\), we have \(a_i h = b_i f^n\) for every \(i\). Taking the sum, we obtain
\begin{equation*}
    h \cdot \left(\sum_i a_i\right) = f^n \cdot \left(\sum_i b_i\right) \in N_B.
\end{equation*}
Again, by property \((D1)\), we have \(\sum_i a_i \in N_B\), and therefore
\[
    \sum_i \frac{a_i}{f^n} \in N_{B_f}.
\]
This shows that \(\Phi\) is an isomorphism of bands.
\end{proof}

%% file: protoexact.tex
\section{Proto-Exactness of Band Modules}\label{sec: protoexact}

As presented in Section \ref{sec:preliminaries}, proto-exact categories requires two classes of morphisms. Here, we define the notion of \textit{strict} morphisms for band modules. The idea is to find the right generalization of the case of stric morphism of hypermodules. Using this we define the classes of admissible monomorphisms (resp. adimissible epimorphism) as the class of strict monomorphisms (resp. strict epimorphisms) and prove the proto-exactness of \(\BMod_B\).

\subsection{Admissible monomorphisms and epimorphisms} 

\begin{definition}[Strict morphism]
    A morphism of \(B\)-modules \(f : M \to N\) is called \emph{strict} if for every sum \(\sum m_i \in M^{+}\) such that \(\sum f(m_i) \in I_N\), there exists a finite sequence \((c_j)_j\) of elements of \(\ker(f)\) for which \(\sum m_i + \sum c_j \in I_M\).
\end{definition}

\begin{remark}\label{remnark:monomorphism}
    For monomorphisms, the condition above admits a simpler formulation. A monomorphism \(f : N \to M\) is strict if and only if the following equivalence holds:
    \[
        \sum_i m_i \in I_M \quad\text{if and only if}\quad \sum_i f(m_i) \in I_N .
    \]
    In particular, every isomorphism of \(B\)-modules is strict.
\end{remark}

\begin{remark}
    By Remark \ref{inclusion strict} and Definition \ref{def: quotient}, we know that both the inclusion morphisms from a submodule and the canonical projections to the quotient modules by an equivalence relation, hence also by a submodule, are strict.
\end{remark}

\begin{remark}
    Let \(f : N \to M\) be a morphism of hypermodules over a hyperring \(H\), regarded as a morphism in \(\BMod_H\).
    Then \(f\) is strict in the sense above if and only if it is strict as a morphism of hypermodules; that is,
    \[
        f(n_1 + n_2) = f(n_1) + f(n_2) \qquad \text{for all } n_1, n_2 \in N.
    \]
\end{remark}

Strictness are also closely related to kernels and cokernels.

\begin{definition}
    A monomorphism (resp. epimorphism) is said to be \textit{normal} if it is the kernel (resp. cokernel) of a morphism.
\end{definition}

\begin{lemma}
    A monomorphism is normal if and only if it is strict.
\end{lemma}

\begin{proof}
    Let \(f : N \hookrightarrow M\) be a monomorphism, and let 
    \(\pi : M \to M/\im(f)\) be the canonical projection. 
    As sets, we have 
    \(
        N = \eq(\pi, 0).
    \)
    It remains to verify that the band–pre-additions agree.

    Since \(\im(f)\) is a sub-\(B\)-module of \(M\), the equalizer 
    \(\eq(\pi,0)\) is the sub-\(B\)-module whose underlying pointed set is 
    \(\{m \in M \mid \pi(m) = \star\} = \im(f)\), hence identified with \(N\).
    Therefore the question is whether 
    \(I_N = I_{\eq(\pi,0)}\).

    If \(f\) is strict, then for any sum \(\sum N_i \in N^{+}\), we have
    \[
        \sum n_i \in I_N 
        \quad\Longleftrightarrow\quad
        \sum f(n_i) \in I_M,
    \]
    which exactly means \(I_N = I_{\eq(\pi,0)}\). 
    Hence \(f\) is the kernel of \(\pi\), so \(f\) is normal.

    Conversely, if \(f\) is the kernel of some morphism, then by definition it is the 
    equalizer \(\eq(\pi,0)\), so the pre-additions agree and \(f\) is strict.
\end{proof}

\begin{lemma}
    An epimorphism is normal if and only if it is strict and all of its fibers 
    contain at most one element, except possibly the fiber over \(\star\).
\end{lemma}

\begin{proof}
    First observe that the projection 
    \(\pi : M \to M/\im(f)\) satisfies all the 
    conditions in the statement.

    Conversely, suppose \(f : N \to M\) is an epimorphism satisfying the two 
    conditions in the statement. We show that \(f\) is the cokernel of the strict 
    monomorphism 
    \[
        f^{-1}(\star) \hookrightarrow N.
    \]

    Consider the map
    \[
        \phi : N / f^{-1}(\star) \to M,
        \qquad
        \phi([n]) = f(n).
    \]
    This map is well defined because if \(n \sim n'\), then either \(n = n'\) or 
    both lie in \(f^{-1}(\star)\), and in both cases we have \(f(n) = f(n')\). 
    Moreover, \(\phi\) preserves the \(B\)-action.

    To show that \(\phi\) is strict, let \(\sum [n_i]\) be a sum in 
    \(I_{N/f^{-1}(\star)}\). Then there exist \(c_j \in f^{-1}(\star)\) such that
    \[
        \sum n_i + \sum c_j \in I_N.
    \]
    Since \(f\) is strict, this is equivalent to
    \[
        \sum f(n_i) \in I_M,
    \]
    which means exactly that \(\sum \phi([n_i]) \in I_M\). Thus \(\phi\) is strict.

    We now show that \(\phi\) is bijective. Surjectivity follows from the 
    surjectivity of \(f\). For injectivity, suppose \(\phi([n]) = \phi([n'])\), 
    i.e.\ \(f(n) = f(n')\). By hypothesis on the fibers, either \(n = n'\) or both 
    lie in \(f^{-1}(\star)\). In both cases we obtain \([n] = [n']\). Hence 
    \(\phi\) is injective.

    Therefore \(\phi\) is an isomorphism of \(B\)-modules, satisfying \(\phi \circ \pi = f\). This shows that 
    \(f\) is the cokernel of the monomorphism \(f^{-1}(\star) \to N\), and hence 
    \(f\) is normal.
\end{proof}

When restricted to strict morphisms, the isomorphism theorem holds.

\begin{proposition}[Isomorphism theorem]\label{thm:isomorphism}
    Let \(f: M \rightarrow N\) be a strict morphism of \(B\)-modules. Then, \(f\) induces an isomorphism of \(B\)-modules 
    \begin{align*}
        \overline{f} : M/\ker(f) &\longrightarrow \im(f)\\
          [m] &\longmapsto f(m).
    \end{align*}
\end{proposition}

\begin{proof}
    First, we prove that \(\overline{f}\) is well defined. Indeed, if \(m \sim m'\), then \(m = m'\) or \(m, m' \in \ker(f)\), and both of the cases imply that \(f(m) = f(m')\).

    It is easy to see that \(\overline{f}\) preserves the action of \(B\). Hence, to prove that \(\overline{f}\) is a morphism of \(B\)-modules consider a sum \(\sum [m_i] \in I_{M/\ker(f)}\). By the definition of quotient, we have \(\sum m_i' \in I_M\) for some \(m_i' \sim m_i\) for every \(i\). Since \(f\) is a morphism of \(B\)-modules we conclude that \[\sum f(m_i) = \sum f(m_i') \in  I_N \cap \im(f)^+ = I_{\im(f)}.\] 

    \(\overline{f}\) is clearly surjective. To prove the injectivity, let \(\overline{f}([m]) = \overline{f}([m'])\). Then, \(f(m) + f(-m') \in I_N\). Since \(f\) is strict, there exists \(c_i \in \ker(f)\) such that \(m - m' + \sum c_i \in I_M\). Taking the projection to \(M/\ker(f)\) and by the uniqueness of \(-[m]\), we conclude that \([m] = [m']\). 

    Finally, we prove that \(\overline{f}\) is strict. Let \(\sum [m_i] \in (M/\ker(f))^+\)be such that \(\sum f(m_i) \in I_{\im(f)} \subseteq I_N\). By the strictness of \(f\), there exist \(c_j \in \ker(f)\) such that \(\sum m_i + \sum c_j \in I_M\). Taking the projection to \(M/\ker(f)\) we get that \(\sum [m_i] \in I_{M/\ker(f)}\).
\end{proof}

\subsection{proto-exactness of \(\BMod_B\)}

Let \(B\) be a fixed band and let \(\mathscr{E}\) be a full subcategory of \(\BMod_B\) satisfying the following properties.
\begin{enumerate}
        \item If \(M \in \mathscr{E}\) and \(N\) is a sub \(B\)-module of \(M\), then \(N \in \mathscr{E}\)
        \item If \(M \in \mathscr{E}\) and \(N\) is a sub \(B\)-module of \(M\), then \(M/N \in \mathscr{E}\).
    \end{enumerate}

Denote by \(\mathfrak{M}_\mathscr{E}\) (resp. \(\mathfrak{E}_\mathscr{E}\)) the class of strict monomorphisms (resp. strict epimorphisms) in \(\mathscr{E}\).

\begin{theorem}\label{Last theorem}
    The triple \((\mathscr{E}, \mathfrak{M}_\mathscr{E}, \mathfrak{E}_\mathscr{E})\) is proto-exact. 
\end{theorem}

\begin{corollary}\label{main thm protoexactness}
    The triples
    \begin{enumerate}
        \item \((\BMod_B, \mathfrak{M}_{\BMod_B}, \mathfrak{E}_{\BMod_B})\), and
        \item \((\BMod_F^{\mathrm{fin,proj}}, \mathfrak{M}_{\BMod_F^{\mathrm{fin,proj}}}, \mathfrak{E}_{\BMod_F^{\mathrm{fin,proj}}})\)
    \end{enumerate} 
    are proto-exact.
\end{corollary}

We devide the proof of Theorem \ref{Last theorem} into several lemmas.

\begin{lemma}\label{lemma:bicartesian}
    Let \(M \in \mathscr{E}\), \(N\) be a sub \(B\)-module of \(M\), and let \(i: N \hookrightarrow M\) be the inclusion. If \(P\) is a submodule of \(N\), then it is also a submodule of \(M\) and we have the following biCartesian commutative diagram of strict morphisms:
    \begin{equation}\label{diag:bicartesian1}
        \begin{tikzcd}
        N \arrow[r, hook, "i"] \arrow[d, two heads, "j"] & M \arrow[d, two heads, "j'"]\\
        N/P \arrow[r, hook, "i'"]   & M/P
    \end{tikzcd}
    \end{equation}
\end{lemma}

\begin{proof} 
    First, we prove that \(i'\) is well defined. Denote by \(\sim_P\) the equivalence relation in \(N\) defined by the sub \(B\)-module \(P\) (see Definition \ref{def:quotient submod}). If \(a \sim_P a'\) in \(N\), then \(i(a) \sim_P i(a')\) in \(M\), hence \(i'([a]) = [i(a)]\) is well defined. Since \(i\) is injective, and we are taking the quotient of both \(M\) and \(N\) by \(P\), \(i'\) is also injective. Now, let \(\sum [n_i] \in (N/P)^+\). Then, it belongs to \(I_{N/P}\) if and only if \(\sum n_i' \in I_N\) for some \(n_i' \sim n_i\). Since \(i\) is a strict monomorphism, this is equivalent to \(\sum i(n_i') \in I_M\), which in turn is equivalent to \(\sum i'([n_i]) = \sum [i(n_i')] \in I_{M/P}\). This concludes that \(i'\) is a strict monomorphism. The commutativity of the diagram is immediate by the definition of \(i'\). 

    We prove that the diagram is coCartesian. Let \(Q\) be a \(B\)-module with two morphisms \(\alpha : N/P \rightarrow Q\) and \(\beta: M/P \rightarrow Q\) such that \(\alpha \circ j = \beta \circ i\).
    \[
    \begin{tikzcd}
        N \arrow[r, hook, "i"] \arrow[d, two heads, "j"] & M \arrow[d, two heads, "j'"] \arrow[rdd, bend left, "\beta"]& \\
        N/P \arrow[r, hook, "i'"] \arrow[rrd, bend right, "\alpha"]  & M/P \arrow[rd, dotted, "\gamma"] & \\
         & & D
    \end{tikzcd}
    \]
    Define the map \(\gamma : M/P \rightarrow D\) as \(\gamma([m]) = \beta(m)\). This is well defined since \(m \sim m'\) implies \(m = m'\) or \(m, m' \in P\). The former case is obvious. Suppose the latter cases occur, that is, both \(m,m' \in P\). Then, since \(P \subseteq N\), \(\beta(m) = (\beta \circ i)(m) = (\alpha \circ j)(m) = \alpha([m]) = \star_D\). In the same way we get \(\beta(m') = \star_D\), hence \(\beta(m) = \beta(m')\). 

    The morphism \(\gamma\) clearly preserves the action of \(B\). Now, let \(\sum [m_i] \in I_{M/P}\). This implies that \(\sum m_i' \in I_M\) for some \(m'_i \sim m_i\), from which we get that \(\sum \gamma([m_i]) =\sum \beta(m_i') \in I_D\). Hence \(\gamma\) is a morphism of \(B\)-modules.

    Finally, we prove that \(\gamma \circ i' = \alpha\): for every \(n \in N\), we have 
    \begin{equation*}
        (\gamma \circ i')([n]) = (\gamma \circ i' \circ j)(n) = (\gamma \circ j' \circ i)(n) = (\beta \circ i)(n) = (\alpha \circ j)(n) = \alpha([n]).
    \end{equation*}

    We prove that the diagram is Cartesian. Let \(R\) be a \(B\)-module with two morphisms \(\alpha: R \rightarrow N/P\) and \(\beta : R \rightarrow M\) such that \(i' \circ \alpha = j' \circ \beta\).
    \[
    \begin{tikzcd}
        R \arrow[rrd, bend left, "\beta"] \arrow[rdd, bend right, "\alpha"] \arrow[rd, dotted, "\gamma"] & & \\
        & N \arrow[r, hook, "i"] \arrow[d, two heads, "j"] & M \arrow[d, two heads, "j'"] \\
        & N/P \arrow[r, hook, "i'"] & M/P
         & & 
    \end{tikzcd}
    \]
    We claim that \(\im(\beta) \subseteq \im(i)\). Indeed, for every \(r \in R\), \((j' \circ \beta)(r) = (i' \circ \alpha)(r)\). Since \(j\) is surjective, there exists an \(n \in N\) such that \(j(n) = \alpha(r)\), hence \(j'(\beta(r)) = (i' \circ j)(n) = j'(i(n))\) which implies that \(\beta(r) \sim i(n)\). This means that either \(\beta(r) = i(n)\) or \(\beta(r), i(n) \in P\), and since \(P \subseteq N\), both cases implie that \(\beta(r) \in \im(i)\). Hence, \(\beta\) induces a morphism of \(B\)-modules \(\gamma: R \to N\) such that \(i \circ \gamma = \beta\). 

    Finally, we prove that \(j \circ \gamma = \alpha\): For every \(r \in R\), we have 
    \[(i' \circ j \circ \gamma)(r) = (j' \circ i \circ \gamma)(r) = (j' \circ \beta)(r) = (i' \circ \alpha)(r),\]
    and by the injectivity of \(i'\), we conclude that \((j \circ \gamma)(r) = \alpha(r)\).
\end{proof}

\begin{lemma}\label{lemma:PE4}
    Every diagram in \(\mathscr{E}\) of the form 
    \[
    \begin{tikzcd}
        & M \arrow[d, two heads ,"j'"]\\
        N' \arrow[r, hook, "i'"] & M'
    \end{tikzcd}
    \]
    where \(i' \in \mathfrak{M}_\mathscr{E}\) and \(j' \in \mathfrak{E}_\mathscr{E}\), can be completed to a biCartesian square 
    \begin{equation}\label{diag:bicartesian2}
    \begin{tikzcd}
        N \arrow[r, hook, "i"] \arrow[d, two heads, "j"] & M \arrow[d, two heads , "j'"]\\
        N' \arrow[r, hook, "i'"] & M'
    \end{tikzcd}
    \end{equation}
    with \(i \in \mathfrak{M}_\mathscr{E}\) and \(j \in \mathfrak{E}_\mathscr{E}\).
\end{lemma}

\begin{proof}
    Define \(N\) as \((j')^{-1}(\im(i')) \in \mathscr{E}\). Then, \(N\) is a \(B\)-submodule of \(M\) and the inclusion morphism \(i: N \hookrightarrow M\) is a strict monomorphism. 

    Observe that by the definition of \(N\) we can define a map \(j : N \rightarrow N'\) as \(j(n) = (i')^{-1}(j'(i(n)))\) for every \(n \in N\). We claim that \(j\) is a strict epimorphism of \(B\)-modules. Indeed, by definition, it is surjective and it clearly preserves the action of \(B\). Let \(\sum n_i \in N^+\). By the strictness of \(i, j'\) and \(i'\), we get that 
    \begin{equation*}
        \sum n_i \in I_N \iff \sum i(n_i) \in I_M \iff \sum (j'\circ i)(n_i) \in I_{M'} \iff \sum j(n_i) \in I_N',
    \end{equation*}
    hence, \(j\) is a strict epimorphism.

    Finally, we prove that the square \eqref{diag:bicartesian2} is isomorphic to a square as in \eqref{diag:bicartesian1} which we know is bi-cartesian. Define \(P\) as the \(B\)-submodule \(i^{-1}(\ker(j'))\) of \(N\). Since \(i\) is injective, we know that \(P \cong \ker(j')\) and by the Proposition \ref{thm:isomorphism}, we get that \(N/P \cong N'\) and \(M/P \cong M'\) and the following diagram commutes
     \[
    \begin{tikzcd}[row sep=scriptsize, column sep=scriptsize]
    N \arrow[dd, two heads, "j"] \arrow[rr, hook, "i"] \arrow[dr, equal] & & M \arrow[dd, two heads,near start, "j'"] \arrow[rd, equal] & \\
    & N \arrow[rr, crossing over, hook, near start, "i"] & & M \arrow[dd, two heads, "\pi'"] \\
    N' \arrow[rr, hook, near end, "i'"] \arrow[dr, "\sim"] & & M' \arrow[rd, "\sim"] & \\
    & N/P \arrow[from=uu, two heads, crossing over, near start, "\pi"] \arrow[rr, hook] & & M/P 
    \end{tikzcd}
    \]
    This proves that the diagram (\ref{diag:bicartesian2}) is bi-cartesian.
\end{proof}

\begin{lemma}\label{lemma:PE5}
    Every diagram in \(\mathscr{E}\) of the form 
    \begin{equation}
        \begin{tikzcd}\label{diag:bicartesian3}
        N \arrow[r, hook, "i"] \arrow[d, two heads, "j"] & M \\
        N' & 
    \end{tikzcd}
    \end{equation}
    where \(i \in \mathfrak{M}_\mathscr{E}\) and \(j \in \mathfrak{E}_\mathscr{E}\), can be completed to a biCartesian square 
    \begin{equation*}
    \begin{tikzcd}
        N \arrow[r, hook, "i"] \arrow[d, two heads, "j"] & M \arrow[d, two heads , "j'"]\\
        N' \arrow[r, hook, "i'"] & M'
    \end{tikzcd}
    \end{equation*}
    with \(i' \in \mathfrak{M}_\mathscr{E}\) and \(j' \in \mathfrak{E}_\mathscr{E}\).
\end{lemma}

\begin{proof}
    Let \(P \coloneq  \ker(j)\), \(M' \coloneq  M/P\), and \(j' : M \rightarrow M/P\) be the canonical projection. By Theorem \ref{thm:isomorphism} we have an isomorphism \( \phi: N' \overset{\sim}{\rightarrow} N/P\). Put \(i' \coloneqq \psi \circ \phi\). By Lemma \ref{lemma:bicartesian}, there exists a morphism \(\psi : N/P \hookrightarrow M/P\) making the following diagram commutate
    \[
    \begin{tikzcd}
        M \arrow[r, hook, "i"] \arrow[d, two heads, "j"] & M \arrow[dd, two heads, "j'"] \\
        N' \arrow[d, "\phi"] \arrow[dr, "i'"] & \\
        N/P \arrow[r, hook, "\psi"] &  M/P\; (= M').
    \end{tikzcd}
    \]
\end{proof}

Lastly, we prove a simple corollary.

\begin{corollary}
    \begin{equation*}
        K_0(\BMod_F^{\mathrm{fin,proj}}) \cong \ZZ.
    \end{equation*}
\end{corollary}

\begin{proof}
    Observe that every admissible short exact sequence of \(\BMod_F^{\mathrm{fin,proj}}\) is of the following form 
    \begin{equation*}
    \begin{tikzcd}
        \bigvee_{i = 1}^nF\cdot e_i \arrow[r, hook] & \bigvee_{i = 1}^{n + m}F\cdot e_i \arrow[r, two heads] & \bigvee_{i = 1}^mF\cdot e_i.
    \end{tikzcd} 
    \end{equation*}
    Hence for every \(M = \bigvee_{i = 1}^n F\cdot e_i\), we have 
    \begin{equation*}
        [M] = n[F],
    \end{equation*}
     and we define \(\dim_F(M) \coloneqq n\). Since \(\BMod_F^{\mathrm{fin,proj}}\) admits coproducts and split admissible sequences, the standard description of the \(K_0\)-group of such a proto-exact category yields an isomorphism of groups
    \begin{equation*}
        K_0(\BMod_F^{\mathrm{fin,proj}}) \longrightarrow \ZZ,\quad [M] \longmapsto \dim_F(M).
    \end{equation*}
\end{proof}

%% file: main.bbl
\begin{thebibliography}{99}

\bibitem[1]{Baker1}
    \textsc{M. Baker and N. Bowler},
    Matroids over partial hyperstructures,
    Adv. Math. 343 (2019), 821--863.

\bibitem[2]{Baker2}
    \textsc{M. Baker and O. Lorscheid},
    The moduli space of matroids,
    Adv. math. 390 (2022), 107883.

\bibitem[3]{Lorscheid3}
    \textsc{M. Baker, T. Jin, O. Lorscheid},
    New building blocks for \(\FF_1\)-geometry: Bands and band schemes,
    J. Lond. Math. Soc. (2) 2025;111:e70125.

\bibitem[4]{Bassat}
    \textsc{O. Ben-Bassat and K. Kremnizer},
    Non-Archimedean analytic geometry as relative algebraic geometry,
    Annales de la Faculté des sciences de Toulouse : Mathématiques, Série 6, Tome 26 (2017) no. 1, pp. 49--126.

\bibitem[5]{Lorscheid2}
    \textsc{C. Chu, O. Lorscheid, R. Santhanam},
    Sheaves and \(K\)-theory of \(\FF_1\)-Schemes,
    Adv. Math. 229 (2012), 2239--2286.

\bibitem[6]{CC}
    \textsc{A. Connes and C. Consani},
    The hyperring of adele classes,
    J.\ Number Theory 131 (2) (2011), 159--194.

\bibitem[7]{Connes1}
    \textsc{A. Connes and C. Consani},
    The hyperring of ad\`ele classes,
    J. Number Theory 131 (2011), 159--194.

\bibitem[8]{Connes2}
    \textsc{A. Connes and C. Consani},
    From monoids to hyperstructures: in search of an absolute arithmetic,
    Casimir Force, Casimir Operators and the Riemann Hypothesis, de Gruyter (2010), 147--198.

\bibitem[9]{Davvaz}
    \textsc{B. Davvaz and A. Salasi},
    A realization of hyperrings,
    Comm. Algebra 34 (12) (2006), 4389--4400.

\bibitem[10]{Deitmar}
    \textsc{A. Deitmar},
    Schemes over \(\FF_1\),
    Number fields and function fields—two parallel worlds, Progr. Math., vol. 239, 2005.

\bibitem[11]{Demazure}
    \textsc{M. Demazure and P. gabriel},
    Groupes alg\'ebriques. Tome I: G\'eom\'etrie alg\'ebrique, g\'en\'eralit\'es, groupes commutatifs. Masson \& Cie \'Editeur, Paris (1970). (Avec un appendice Corps de classes local par Michiel Hazewinkel)

\bibitem[12]{Kapranov}
    \textsc{T. Dyckerhoff and M. Kapranov},
    Higher Segal Spaces,
    Lecture Notes in mathematics, Springer, 2019.

\bibitem[13]{Jaiung3}
    \textsc{C. Eppolito, J. Jun, and M. Szczesny},
    Proto-exact categories of matroids, Hall algebras, and K-theory,
    Math.\ Z., 296 (2020), 147--167.

\bibitem[14]{Timothy}
    \textsc{T. Hosgood},
    Under Spec\(\mathbb{Z}\): A reader's companion,
    Dissertation, Oxford University, 2016, \url{https://thosgood.net/assets/files/under-spec-z.pdf}.

\bibitem[15]{Jaiung}
    \textsc{J. Jun},
    Algebraic geometry of hyperrings,
    Adv. Math., 323 (2018), 142--192.

\bibitem[16]{Jaiung1}
    \textsc{J. Jun, M. Szczesny, J. Tolliver},
    Proto-exact categories of modules over semirings and hyperrings
    J. Algebra 631 (2023), 517--557.

\bibitem[17]{Jaiung2}
    \textsc{J. Jun, A. Sistko, and C. Wright},
    Proto-exact Categories of Matroids over Idylls and Tropical Toric Reflexive Sheaves,
    arXiv:2509.08144.

\bibitem[18]{Krasner}
    \textsc{M. Krasner},
    Quelques m\'ethodes nouvelles dans la th\'eorie des corps valu\'es complets. Coll. Int. du C.N.R.S. \textbf{24} (1949).

\bibitem[19]{Krasner2}
    \textsc{M. Krasner},
    A class of hyperrings and hyperfields,
    Internat.\ J.\ Math.\ Math.\ Sci.\ 6 (2) (1983), 307--311.

\bibitem[20]{Lorscheid4}
    \textsc{M. Jarra, O. Lorscheid, and E. Vital},
    Quiver matroids: Matroid morphisms, quiver Grassmannians, their Euler characteristics and F1-points, 
    arXiv: 2404.09255v1.

\bibitem[21]{Lorscheid1}
    \textsc{O. Lorscheid},
    Blue schemes, semiring schemes, and relative schemes after To\"en and Vaqui\'e,
    J. Algebra 482 (2017), 264--302.

\bibitem[22]{Manin}
    \textsc{Y. Manin},
     Lectures on zeta functions and motives (according to Deninger and Kurokawa),
     Ast\'erisque, 228 (4), 121--163.

\bibitem[23]{LorscheidSurvey}
    \textsc{J. L. Pe\~na and O. Lorscheid},
    Mapping \(\FF_1\)-land: an overview of geometries over the field with one element,
    Noncommutative geometry, arithmetic, and related topics, Johns Hopkins Univ. Press, Baltimore, MD (2011), 241--265.
     
\bibitem[24]{Rota}
    \textsc{R. Procesi Ciampi and R. Rota},
    The hyperring spectrum,
    Riv.\ Mat.\ Pura Appl., 1 (1987), 71--80.

\bibitem[25]{Tits}
    \textsc{M. J. Tits},
    Sur les analogues algébriques des groupes semi-simples complexes,
    Colloque d'algèbre supérieure, tenu à Bruxelles du 19 au 22 décembre 1956, Centre Belge de Recherches Mathématiques Établissements Ceuterick, Louvain, Paris: Librairie Gauthier-Villars, 261--289.

\bibitem[26]{ToenVaquie}
    \textsc{B. To\"en and M. Vaqui\'e},
    Au-dessous de Spec\(\mathbb{Z}\),
    Journal of K-theory 3(3) (2009), 437--500.

\bibitem[27]{Tolliver}
    \textsc{J. Tolliver},
    An equivalence between two approaches to limits of local fields,
    J. Number Theory 166 (2016), 473--492.

\bibitem[28]{Vezzani}
    \textsc{A. Vezzani},
    Deitmar's versus To\"en-Vaqui\'e's schemes over \(\mathbb{F}_1\),
    Math. Z. 271 (2012), 911--926.

\end{thebibliography}
